\documentclass[twoside,11pt]{article}

\usepackage[nohyperref]{jmlr2e}

\usepackage{amsmath} 
\usepackage{tikz}
\usetikzlibrary{fit, shapes}
\tikzset{both/.style={rectangle, path picture={%
\draw[black]
(path picture bounding box.north) -- (path picture bounding box.west) 
(path picture bounding box.west) -- (path picture bounding box.south)
(path picture bounding box.south) -- (path picture bounding box.east) 
(path picture bounding box.east) -- (path picture bounding box.north);
}}}
\usepackage{algorithm, algorithmic}
\usepackage{booktabs}

\usepackage[colorlinks=false,allbordercolors={1 1 1}, breaklinks=TRUE, bookmarksnumbered=FALSE]{hyperref}

\newcommand\ind{\protect\mathpalette{\protect\independenT}{\perp}} 
\def\independenT#1#2{\mathrel{\rlap{$#1#2$}\mkern4mu{#1#2}}}
\DeclareMathOperator{\de}{de}
\DeclareMathOperator{\Forbb}{forb}
\newcommand{\fb}[2][X,Y]{\Forbb(\mathbf{#1},#2)}

\defcitealias{HenckelPerkovicMaathuis2019}{HPM19}
\defcitealias{RotnitzkySmucler2019}{RS20}

\firstpageno{1}

	\usepackage{lastpage}
	\jmlrheading{21}{2020}{1-45}{2/20; Revised
		11/20}{12/20}{20-175}{Janine Witte, Leonard Henckel, Marloes H.\ Maathuis and Vanessa Didelez}
	\ShortHeadings{On Efficient Adjustment in Causal Graphs}{Witte, Henckel, Maathuis and Didelez}

\begin{document}
	
	\begin{tikzpicture}[overlay]
	\node[draw, fill=white, thick,minimum width=6cm,minimum height=0.6cm] (b) at (7,1){\url{https://www.jmlr.org/papers/v21/20-175.html}};
	\end{tikzpicture}
	
	\vspace{-0.4cm}

\title{On Efficient Adjustment in Causal Graphs}

\author{\name Janine Witte \email witte@leibniz-bips.de \\
       \addr Leibniz Institute for Prevention Research and Epidemiology---BIPS, Bremen, Germany\\
       and Faculty of Mathematics and Computer Science, University of Bremen, Germany       
       \AND
       \name Leonard Henckel \email henckel@stat.math.ethz.ch \\
       \addr Seminar for Statistics, ETH Zurich, Switzerland
       \AND
       \name Marloes H.\ Maathuis \email maathuis@stat.math.ethz.ch\\
       \addr Seminar for Statistics, ETH Zurich, Switzerland
   		\AND
   		\name Vanessa Didelez \email didelez@leibniz-bips.de \\
   		\addr Leibniz Institute for Prevention Research and Epidemiology---BIPS, Bremen, Germany\\
   		and Faculty of Mathematics and Computer Science, University of Bremen, Germany   }

\editor{Peter Spirtes}

\maketitle

\begin{abstract}
We consider estimation of a total causal effect from observational data via covariate adjustment. Ideally, adjustment sets are selected based on a given causal graph, reflecting knowledge of the underlying causal structure. Valid adjustment sets are, however, not unique. Recent research has introduced a graphical criterion for an `optimal’ valid adjustment set (\(\mathbf{O}\)-set). For a given graph, adjustment by the \(\mathbf{O}\)-set yields the smallest asymptotic variance compared to other adjustment sets in certain parametric and non-parametric models. In this paper, we provide three new results on the \(\mathbf{O}\)-set. First, we give a novel, more intuitive graphical characterisation: We show that the \(\mathbf{O}\)-set is the parent set of the outcome node(s) in a suitable latent projection graph, which we call the forbidden projection. An important property is that the forbidden projection preserves all information relevant to total causal effect estimation via covariate adjustment, making it a useful methodological tool in its own right. Second, we extend the existing IDA algorithm to use the $\mathbf{O}$-set, and argue that the algorithm remains semi-local. This is implemented in the \texttt{R}-package \texttt{pcalg}. Third, we present assumptions under which the \(\mathbf{O}\)-set can be viewed as the target set of popular non-graphical variable selection algorithms such as stepwise backward selection.
\end{abstract}


\section{Introduction}
\label{sec:intro}

In typical analyses of observational data, we wish to estimate the total causal effect of a (possibly multivariate) treatment or exposure \(\mathbf{X}\) on a (possibly multivariate) outcome \(\mathbf{Y}\). Ideally, we can fully specify the underlying causal directed acyclic graph (DAG). We can then use a graphical adjustment criterion, e.g.\ Pearl's back-door criterion \citep{Pearl2009} or the generalised adjustment criterion \citep{Perkovicetal2015, Perkovicetal2018, ShpitserVanderWeeleRobins2010}, to check whether a set of covariates is valid for adjustment. However, there may be more than one valid adjustment set. Although all resulting estimators are then consistent, their variances may differ considerably.

There are several approaches to choose an adjustment set among all valid adjustment sets. For example, one can pick a minimal adjustment set \citep{deLunaWaernbaumRichardson2011,TextorLiskiewicz2011}. An alternative strategy is to aim at decreasing the causal effect estimator's variance by including variables associated with the outcome (e.g.\ \citealp{Brookhartetal2006, LuncefordDavidian2004, ShortreedErtefaie2017}). \cite{WitteDidelez2019} referred to this strategy as the `outcome-oriented' approach. It is especially popular when little graphical knowledge is available. A major advancement for the outcome-oriented approach was the graphical characterisation of the `optimal' adjustment set (\(\mathbf{O}\)-set) by \cite{HenckelPerkovicMaathuis2019} (HPM19). They showed that under a linear model, adjusting for the \(\mathbf{O}\)-set yields the smallest asymptotic variance for the causal effect estimator compared to all other valid adjustment sets, under assumptions detailed below. Strengthening this result, \cite{RotnitzkySmucler2019} (RS20) recently showed that the minimal variance property of the \(\mathbf{O}\)-set is retained for a class of non-parametric estimators. All these results apply to DAGs, as well as so-called amenable completed partially directed acyclic graphs (CPDAGs; see e.g.\ \citealp{AnderssonMadiganPerlman1997}) and amenable maximally oriented partially directed acyclic graphs (maxPDAGs; see \citealp{PerkovicKalischMaathuis2017}). These are larger classes of graphs allowing for undirected edges where the direction cannot be decided. Amenability implies that despite the undirected edges, an adjustment set can be identified from the CPDAG (or maxPDAG) so that this set is valid for adjustment in all DAGs in the equivalence class. If a CPDAG (or maxPDAG) is not amenable, no common adjustment set for all DAGs in the equivalence class exists \citep{Perkovicetal2018}, and hence different DAGs may imply different true causal effects of \(\mathbf{X}\) on  \(\mathbf{Y}\). However, it is then still possible to estimate a multiset of possible causal effects (meaning that all effects in the multiset are compatible with the non-amenable graph) using the IDA algorithm by \citet{MaathuisKalischBuhlmann2009, Maathuisetal2010}.

In this paper, we provide three new results on efficient causal effect estimation. First, after briefly reviewing the results of \citetalias{HenckelPerkovicMaathuis2019} and \citetalias{RotnitzkySmucler2019} (Section~\ref{sec:review}), we provide an alternative, intuitive characterisation of the \(\mathbf{O}\)-set. This is based on the new concept of a forbidden projection, which has many interesting properties regarding adjustment for confounding (Section~\ref{sec:new_def}). Second, we extend the application of the \(\mathbf{O}\)-set to non-amenable CPDAGs and maxPDAGs, by incorporating optimal adjustment into the IDA algorithm (Section~\ref{sec:IDA}). Third, we discuss how and under what assumptions the \(\mathbf{O}\)-set can be viewed as the target set of data-driven variable selection methods such as backward model selection (Section~\ref{sec:data}).

\section{Optimal Adjustment for Known Causal Structure}
\label{sec:review}

We begin by clarifying our setting and defining the \(\mathbf{O}\)-set, before proposing an alternative definition in Section~\ref{sec:new_def}. We defer most of the terminology and formal definitions to Appendix~\ref{appA}; here we only state some key concepts.

\textbf{(Possibly) causal nodes and forbidden nodes.} Let \(\mathcal{G}\) be a causal DAG, CPDAG or maxPDAG. A path \((V_1,\dots,V_m)\) in \(\mathcal{G}\) is called \textit{causal} from \(V_1\) to \(V_m\) if \(V_i\rightarrow V_{i+1}\) for all \(i\in\{1, \dots, m-1\}\). It is called \textit{possibly causal} if there are no \(i,j\in\{1,\dots,m\}\), \(i<j\), such that \(V_i\leftarrow V_j\). Otherwise it is called \textit{non-causal} from \(V_1\) to \(V_m\). A path from \(\mathbf{X}\) to \(\mathbf{Y}\) is \textit{proper} if only its first node is in \(\mathbf{X}\). If there is a causal path from \(V_1\) to \(V_m\) in \(\mathcal{G}\), then \(V_m\) is called a \textit{descendant} of \(V_1\) in \(\mathcal{G}\). Analogously, if there is a possibly causal path from \(V_1\) to \(V_m\) in \(\mathcal{G}\), then \(V_m\) is called a \textit{possible descendant} of \(V_1\) in \(\mathcal{G}\). The set of all descendants of \(V_1\) in \(\mathcal{G}\) is denoted by \(\mathrm{de}(V_1, \mathcal{G})\), and the set of all possible descendants by \(\mathrm{possde}(V_1, \mathcal{G})\). The \textit{causal nodes} with respect to (\(\mathbf{X}\), \(\mathbf{Y}\)) in \(\mathcal{G}\), denoted by \(\mathrm{cn}(\mathbf{X},\mathbf{Y},\mathcal{G})\), are the nodes on proper causal paths from \(\mathbf{X}\) to \(\mathbf{Y}\), excluding \(\mathbf{X}\) itself. The \textit{possibly causal nodes} \(\mathrm{posscn}(\mathbf{X},\mathbf{Y},\mathcal{G})\) are defined analogously. The \textit{forbidden set} with respect to (\(\mathbf{X}\), \(\mathbf{Y}\)) and \(\mathcal{G}\) is defined as \(\mathrm{forb}(\mathbf{X},\mathbf{Y},\mathcal{G})=\mathrm{possde}(\mathrm{posscn}(\mathbf{X},\mathbf{Y},\mathcal{G}),\mathcal{G})\cup\mathbf{X}\). In a DAG, this simplifies to \(\mathrm{forb}(\mathbf{X},\mathbf{Y},\mathcal{G})=\mathrm{de}(\mathrm{cn}(\mathbf{X},\mathbf{Y},\mathcal{G}),\mathcal{G})\cup\mathbf{X}\). The nodes in the forbidden set are called \textit{forbidden nodes}. It can be shown that valid adjustment sets never contain forbidden nodes \citep{Perkovicetal2018}.

\textbf{Valid adjustment sets.} We consider a set of treatments \(\mathbf{X}\) and a set of outcomes \(\mathbf{Y}\). A (possibly empty) set \(\mathbf{Z}\) is a \textit{valid adjustment set} relative to \((\mathbf{X},\mathbf{Y})\) if the interventional distribution \(f(\mathbf{y}\mid do(\mathbf{x}))\) of \(\mathbf{Y}\), given we set \(\mathbf{X}\) to \(\mathbf{x}\) by intervention, factorises as follows:
\begin{equation*}
f(\mathbf{y}\mid do(\mathbf{x}))=\begin{cases}
f(\mathbf{y}\mid\mathbf{x})&\textit{if }\mathbf{Z}=\emptyset,\\
\int_{\mathbf{z}}f(\mathbf{y}\mid\mathbf{x},\mathbf{z})f(\mathbf{z})d\mathbf{z}&\text{otherwise.}
\end{cases}
\end{equation*}
Valid adjustment sets can be read off from a given causal DAG, CPDAG or maxPDAG \(\mathcal{G}\) using the \textit{generalised adjustment criterion} \citep{PerkovicKalischMaathuis2017,Perkovicetal2018,ShpitserVanderWeeleRobins2010}, which generalises Pearl's back-door criterion \citep{Pearl2009}: \(\mathbf{Z}\) is a valid adjustment set relative to \((\mathbf{X},\mathbf{Y})\) in \(\mathcal{G}\) if and only if the following three conditions hold: (a) every proper possibly causal path from \((\mathbf{X}\) to \(\mathbf{Y})\) starts with a directed edge out of \(\mathbf{X}\), (b) \(\mathbf{Z}\cap\mathrm{forb}(\mathbf{X},\mathbf{Y},\mathcal{G})=\emptyset\), (c) all proper non-causal definite-status paths from \(\mathbf{X}\) to \(\mathbf{Y}\) are blocked by \(\mathbf{Z}\). Property (a) is called \textit{amenability}. See Appendix \ref{appA} for the definition of a definite-status path. In a DAG, all paths are of definite status.

We consider two model classes and corresponding strategies for estimating causal effects when a valid adjustment set is available: (i) the causal linear model with possibly non-Gaussian error terms, where causal effects are estimated using linear regression \citepalias{HenckelPerkovicMaathuis2019}, and (ii) the more general non-parametric causal model, where estimation proceeds non-parametrically \citepalias{RotnitzkySmucler2019}. In both settings, we assume an underlying causal DAG, and that we observe all variables displayed as nodes in the DAG, i.e.\ there are no latent variables.

\textbf{Causal linear models (\citetalias{HenckelPerkovicMaathuis2019}).} A causal linear model is a causal DAG where every edge represents a linear causal effect. In a causal linear model, the (joint) causal effect of \(\mathbf{X}=\{X_1,\dots,X_{k_x}\}\) on \(\mathbf{Y}=\{Y_1,\dots,Y_{k_y}\}\) is defined as the matrix \(\boldsymbol{\tau}_{\mathbf{yx}}\) with elements
\begin{align*}
(\boldsymbol{\tau}_{\mathbf{yx}})_{j,i}&=\frac{\partial}{\partial x_i}\mathrm{E}(Y_j\mid do(x_1,\dots,x_{k_x}))\\
&=\mathrm{E}(Y_j\mid do(x_1,\dots,x_i+1,\dots,x_{k_x})) - \mathrm{E}(Y_j\mid do(x_1,\dots,x_{k_x})),
\end{align*}
where element \((\boldsymbol{\tau}_{\mathbf{yx}})_{j,i}\) corresponds to the controlled direct effect \citep{RobinsGreenland1992, Pearl2001} of \(X_i\) on \(Y_j\) relative to \(\mathbf{X}\). In other words, \((\boldsymbol{\tau}_{\mathbf{yx}})_{j,i}\) is the difference in \(\mathrm{E}(Y_j)\) when \(\mathbf{X}\) is set to \((x_1,\cdots,x_i+1,\dots,x_{k_x})\) by intervention, compared to when \(\mathbf{X}\) is set to \((x_1,\dots,x_{k_x})\) by intervention. We can compute the effect of more general interventions as functions of the elements of \(\boldsymbol{\tau}_{\mathbf{yx}}\); for example, the sum of the first row corresponds to the effect on \(Y_1\) of increasing all elements of \((x_1,\dots,x_{k_x})\) by one. Given a valid adjustment set \(\mathbf{Z}\) for the effect of \(\mathbf{X}\) on \(\mathbf{Y}\), \(\boldsymbol{\tau}_{\mathbf{yx}}\) can be rewritten as a matrix of regression coefficients as follows: Denote by \(\boldsymbol{\beta}_{\mathbf{yx}.\mathbf{z}}\) the \((k_y\times k_x)\)-matrix whose \((j,i)\)-th element is the regression coefficient \(\beta_{y_jx_i.\mathbf{x_{-i}z}}\) of \(X_i\) in a linear regression of \(Y_j\) on \(X_i\) and \(\mathbf{Z}\cup\mathbf{X}_{-i}\), where \(\mathbf{X}_{-i}=\mathbf{X}\setminus\{X_i\}\). Then \(\boldsymbol{\tau}_{\mathbf{yx}}=\boldsymbol{\beta}_{\mathbf{yx}.\mathbf{z}}\). The ordinary least squares (OLS) estimator \(\hat{\boldsymbol{\beta}}_{\mathbf{yx}.\mathbf{z}}\) is a consistent estimator of \(\boldsymbol{\beta}_{\mathbf{yx}.\mathbf{z}}\). We denote the asymptotic variance of \(\hat{\beta}_{y_jx_i.\mathbf{x_{-i}z}}\) by \(a.var(\hat{\beta}_{y_jx_i.\mathbf{x_{-i}z}})\).

\textbf{Non-parametric estimation of causal effects (\citetalias{RotnitzkySmucler2019}).} In the more general setting of a causal DAG without linearity or other assumptions on the functional form, we define the causal effect of \(\mathbf{X}\) on \(\mathbf{Y}\) as follows. Let \(\mathcal{X}\) be the set of values that \(\mathbf{X}\) can take. For a pair of vectors \(\mathbf{x}, \mathbf{x'}\in\mathcal{X}\), the causal effect of intervening to set \(\mathbf{X}\) to \(\mathbf{x}\) vs.\ \(\mathbf{x'}\) is the vector \(\boldsymbol{\Delta}_{\mathbf{yxx'}}\) with elements
\[(\boldsymbol{\Delta}_{\mathbf{yxx'}})_j=\mathrm{E}(Y_j\mid do(\mathbf{x}))-\mathrm{E}(Y_j\mid do(\mathbf{x'})).\]
Note that in the non-parametric case, it is not possible to compactly represent the causal effect of \(\mathbf{X}\) on \(\mathbf{Y}\) in a \((k_y\times k_x)\)-matrix. \citetalias{RotnitzkySmucler2019} considered the class of regular asymptotically linear estimators for the non-parametric estimation of \(\boldsymbol{\Delta}_{\mathbf{yxx'}}\). This class includes inverse probability weighting by a non-parametrically estimated propensity score \citep{HiranoImbensRidder2003}, non-parametric outcome regression \citep{Hahn1998}, and double machine learning \citep{Chernozhukovetal2018, SmuclerRotnitzkyRobins2019}. We use \(\hat{\boldsymbol{\Delta}}_{\mathbf{yxx'}.\mathbf{z}}\) to denote an estimator from this class that estimates \(\boldsymbol{\Delta}_{\mathbf{yxx'}}\) adjusting for a valid adjustment set \(\mathbf{Z}\). Under a causal DAG model and certain smoothness and complexity restrictions, \(\hat{\boldsymbol{\Delta}}_{\mathbf{yxx'}.\mathbf{z}}\) is a consistent estimator of \(\boldsymbol{\Delta}_{\mathbf{yxx'}}\). For given \(\mathbf{y}\), \(\mathbf{x}\) and \(\mathbf{x'}\), the asymptotic distribution of estimators from this class depends only on \(\mathbf{Z}\), therefore we do not further distinguish between the estimators. We denote the asymptotic variance of \((\hat{\boldsymbol{\Delta}}_{\mathbf{yxx'}.\mathbf{z}})_j=\hat{\Delta}_{y_j\mathbf{xx'}.\mathbf{z}}\) by \(a.var(\hat{\Delta}_{y_j\mathbf{xx'.z}})\). See \citetalias{RotnitzkySmucler2019} and the references therein for more details on regular asymptotically linear estimators.

\begin{definition}[\(\mathbf{O}\)-set; \citetalias{HenckelPerkovicMaathuis2019} Definition 3.8]
	\label{Oset1}
	Let \(\mathbf{X}\) and \(\mathbf{Y}\) be disjoint node sets in a DAG, CPDAG or maxPDAG \(\mathcal{G}\). Then \(\mathbf{O}(\mathbf{X},\mathbf{Y},\mathcal{G})\) is defined as:
	\[\mathbf{O}(\mathbf{X},\mathbf{Y},\mathcal{G})=\mathrm{pa}(\mathrm{cn}(\mathbf{X},\mathbf{Y},\mathcal{G}),\mathcal{G})\setminus \mathrm{forb}(\mathbf{X},\mathbf{Y},\mathcal{G}).\]
\end{definition}

An example is given in Figure~\ref{fig:ExO}. It shows the causal relations between 12 symptoms of prodromal schizophrenia as measured by the \textit{Schizotypic Syndrome Questionnaire} \citep{vanKampen2006}. The DAG was constructed using a combination of expert knowledge and data-driven structure learning \citep{vanKampen2014}. For illustration, we here take this given DAG as ground truth. Suppose we are interested in the causal effect of Alienation (ALN) on Delusional Thinking (DET). The bold edges indicate the causal paths with causal nodes \{PER, SUS, FTW, DET\} (circles). The parents of the causal nodes are \{ALN, PER, SUS, FTW, AIS, CDR\}, the forbidden set is \{ALN, PER, SUS, FTW, DET, HOS, EGC\} and the \(\mathbf{O}\)-set is \{ALN, PER, SUS, FTW, AIS, CDR\} \(\setminus\) \{ALN, PER, SUS, FTW, DET, HOS, EGC\}\(=\)\{AIS, CDR\} (shown in boxes). Other valid adjustment sets are, for example, \{AFF, SAN\}, \{AIS, CDR, AFF\} and \{AFF, APA, AIS, CDR, SAN\}. This can be checked using the generalised adjustment criterion stated above.

\begin{figure}[h]
	\begin{center}
		\begin{tikzpicture}[node distance=20mm, >=latex]
		\node (AFF) {AFF};
		\node[below of=AFF, yshift=-20mm] (APA) {APA};
		\node[right of=APA, xshift=10mm, circle, fill=gray!50, draw, inner sep=0pt, minimum size=32pt] (ALN) {ALN};
		\node[right of=AFF, xshift=60mm, rectangle, draw] (CDR) {CDR};
		\node[right of=ALN, xshift=70mm, circle, fill=gray!50, draw, inner sep=0pt, minimum size=32pt] (DET) {DET};
		\node[above of=ALN] (SAN) {SAN};
		\node[right of=SAN, xshift=15mm, rectangle, draw] (AIS) {AIS};
		\node[right of=ALN, yshift=-20mm, circle, draw, inner sep=0pt, minimum size=32pt] (SUS) {SUS};
		\node[right of=SUS, xshift=20mm, circle, draw, inner sep=0pt, minimum size=32pt] (FTW) {FTW};
		\node[above of=FTW, yshift=10mm, circle, draw, inner sep=0pt, minimum size=32pt] (PER) {PER};
		\node[below of=SUS, xshift=-20mm, circle, draw, inner sep=0pt, minimum size=32pt] (HOS) {HOS};
		\node[below of=SUS, xshift=20mm, circle, draw, inner sep=0pt, minimum size=32pt] (EGC) {EGC};
		\draw[->, line width=1mm] (ALN) to (DET);
		\draw[->, thick] (AFF) to (ALN);
		\draw[->, thick] (SAN) to (AFF);
		\draw[->, thick] (SAN) to (AIS);
		\draw[->, thick] (SAN) to (APA);
		\draw[->, thick] (SAN) to (ALN);
		\draw[->, thick] (SAN) to (CDR);
		\draw[->, thick] (CDR) to (DET);
		\draw[->, thick] (AFF) to (APA);
		\draw[->, thick] (AIS) to (AFF);
		\draw[->, thick] (AFF) to (CDR);
		\draw[->, thick] (ALN) to (APA);
		\draw[->, line width=1mm] (ALN) to (PER);
		\draw[->, line width=1mm] (ALN) to (SUS);
		\draw[->, line width=1mm] (ALN) to (FTW);
		\draw[->, thick] (AIS) to (SUS);
		\draw[->, thick] (AIS) to (EGC);
		\draw[->, thick] (SUS) to (HOS);
		\draw[->, thick] (EGC) to (HOS);
		\draw[->, thick] (FTW) to (EGC);
		\draw[->, thick] (SUS) to (EGC);
		\draw[->, line width=1mm] (PER) to (DET);
		\draw[->, line width=1mm] (SUS) to (FTW);
		\draw[->, line width=1mm] (FTW) to (DET);
		\end{tikzpicture}
	\end{center}
	\caption{DAG from \cite{vanKampen2014} illustrating the assumed causal relations between 12 prodromal symptoms of schizophrenia: AFF=Affective Flattening, AIS=Active Isolation, ALN=Alienation, APA=Apathy, CDR=Cognitive Derailment, DET=Delusional Thinking, EGC=Egocentrism, FTW=Living in a Fantasy World, HOS=Hostility, PER=Perceptual Aberrations, SAN=Social Anxiety, SUS=Suspiciousness. We are interested in the causal effect of ALN on DET, both shown in grey circles. Bold arrows show the causal paths from ALN to DET. The forbidden nodes are shown as circles, nodes in the \(\mathbf{O}\)-set are shown as boxes.}
	\label{fig:ExO}
\end{figure}
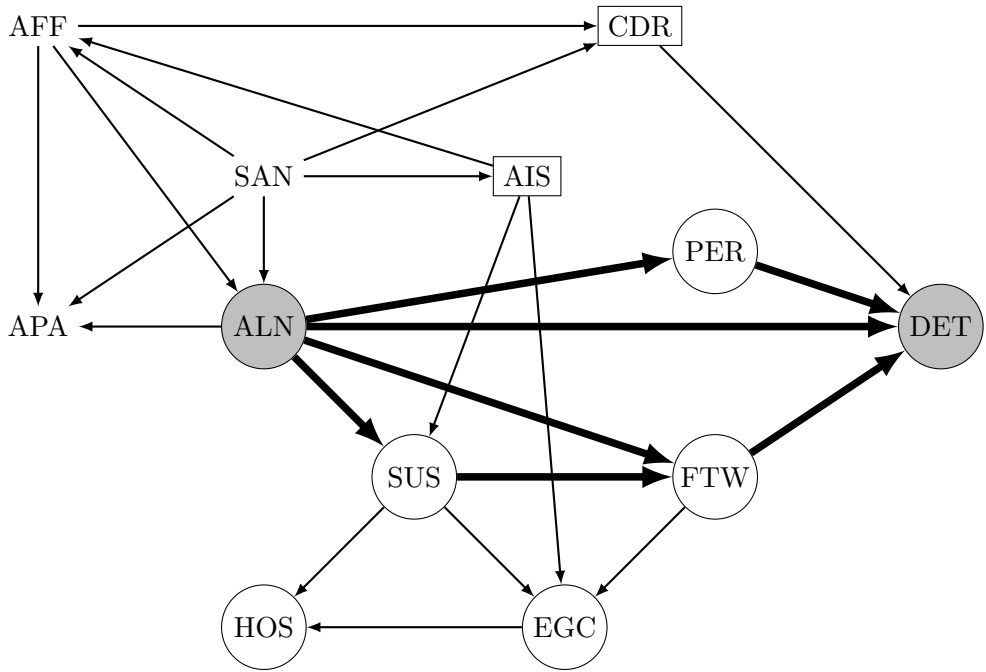

Note that in many applications it might be possible to augment a causal graph e.g.\ with further parents of \(Y\) that are marginally independent of all other non-descendants of \(Y\). This induces a different \(\mathbf{O}\)-set illustrating that this set depends on what variables are included in the graph. Note also that the \(\mathbf{O}\)-set is defined even if no valid adjustment set exists, but this case will rarely be of interest.

\begin{proposition}[\citetalias{HenckelPerkovicMaathuis2019} Theorem 3.10 (1)]
	\label{Osuff}
	Let \(\mathbf{X}\) and \(\mathbf{Y}\) be disjoint subsets of the node set \(\mathbf{V}\) of a causal DAG, CPDAG or maxPDAG \(\mathcal{G}\). The set \(\mathbf{O}(\mathbf{X},\mathbf{Y},\mathcal{G})\) is a valid adjustment set relative to \((\mathbf{X},\mathbf{Y})\) in \(\mathcal{G}\) if (i) \(\mathbf{Y}\subseteq\mathrm{possde}(\mathbf{X},\mathcal{G})\) and (ii) a valid adjustment set relative to \((\mathbf{X},\mathbf{Y})\) in \(\mathcal{G}\) exists.
\end{proposition}

Condition (i) can be checked using a simple query on \(\mathcal{G}\). If \(\mathbf{Y}\not\subseteq\mathrm{possde}(\mathbf{X},\mathcal{G})\), we know that the causal effect of \(\mathbf{X}\) on \(\mathbf{Y}\setminus\mathrm{possde}(\mathbf{X},\mathcal{G})\) is zero. Hence, without loss of generality, we can consider the set of outcome variables \(\mathbf{Y}\cap\mathrm{possde}(\mathbf{X},\mathcal{G})\) instead of \(\mathbf{Y}\). Condition (ii) is satisfied if \(\mathbf{O}(\mathbf{X},\mathbf{Y},\mathcal{G})\) or any other subset of \(\mathbf{V}\setminus\{\mathbf{X},\mathbf{Y}\}\) fulfils the generalised adjustment criterion stated above. For the DAG in Figure~\ref{fig:ExO}, it can easily be seen that \(\text{DET}\in\mathrm{de}(\text{ALN})\), hence condition (i) is satisfied. Under condition (i), condition (ii) is always satisfied for univariate treatment and outcome in a DAG, because the parents of treatment then form a valid adjustment set (see \citealp{Pearl2009}, p.\,72f.).

The following proposition, which builds on earlier work by \cite{KurokiMiyakawa2003} and \cite{KurokiCai2004}, establishes the optimality of the \(\mathbf{O}\)-set in terms of the asymptotic variance in the linear and in the non-parametric setting.

\begin{proposition}
	\label{eff}
	Let \(\mathbf{X}\) and \(\mathbf{Y}\) be disjoint subsets of the node set \(\mathbf{V}\) of a causal DAG, CPDAG or maxPDAG \(\mathcal{G}\), such that \(\mathbf{Y}\subseteq\mathrm{possde}(\mathbf{X},\mathcal{G})\). Let \(\mathbf{Z}\) be a valid adjustment set relative to \((\mathbf{X},\mathbf{Y})\) in \(\mathcal{G}\) and let \(\mathbf{O}=\mathbf{O}(\mathbf{X},\mathbf{Y},\mathcal{G})\).
	\begin{itemize}
		\item[(a)] (\citetalias{HenckelPerkovicMaathuis2019} Theorem 3.10 (2)) If the variables \(\mathbf{V}\) follow a linear causal model compatible with \(\mathcal{G}\), then, for every \(X_i\in\mathbf{X}\) and \(Y_j\in\mathbf{Y}\), \(a.var(\hat{\beta}_{y_jx_i.\mathbf{x}_{-i}\mathbf{o}})\le a.var(\hat{\beta}_{y_jx_i.\mathbf{x}_{-i}\mathbf{z}})\).
	
		\item[(b)] (\citetalias{RotnitzkySmucler2019} Theorem 2) For every \(Y_j\in\mathbf{Y}\) and pair of vectors \(\mathbf{x}, \mathbf{x'}\in\mathcal{X}\),
		
		\(a.var(\hat{\Delta}_{y_j\mathbf{xx'.o}}) \le a.var(\hat{\Delta}_{y_j\mathbf{xx'.z}})\).
	\end{itemize}
\end{proposition}

In other words, for a given causal linear model, the \(\mathbf{O}\)-set yields the smallest asymptotic variance for the OLS estimator among all valid adjustment sets. If linearity cannot be assumed, the \(\mathbf{O}\)-set yields the smallest variance for regular asymptotically linear estimators. Thus, assume that Figure~\ref{fig:ExO} represents a causal linear model. Proposition \ref{eff} then implies that if we estimate the effect of ALN on DET by linearly regressing DET on ALN and the \(\mathbf{O}\)-set \{AIS, CDR\}, then the estimator will have a smaller asymptotic variance than if we regress DET on ALN and a different valid adjustment set, say the parent set of ALN, which equals \{AFF, SAN\}. Moreover, when we relax linearity, non-parametric adjustment for the \(\mathbf{O}\)-set \{AIS, CDR\} is more efficient than non-parametric adjustment for any other valid adjustment set, provided the estimator is in the class of regular asymptotically linear estimators.

\section{The O-Set via Forbidden Projection}
\label{sec:new_def}

In this section we provide an alternative, intuitive construction of the \(\mathbf{O}\)-set. For the sake of clarity, we restrict ourselves to DAGs; generalisations to amenable maxPDAGs are given in Appendix~\ref{appC}.

To motivate our alternative construction, we posit that a useful adjustment set should be i) valid, ii) easy to compute, and iii) efficient. Consider singleton treatment \(X\) and outcome \(Y\), where the latter is not an ancestor of  \(X\). The parents of \(X\) are easy to determine and guaranteed to be valid (see \citealp{Pearl2009}, p.\,72f.). However, it is well-known that adjusting for variables strongly associated with treatment tends to reduce the efficiency of OLS and other estimators of the treatment effect. Hence, adjusting for the parents of treatment is typically inefficient compared to other valid adjustment sets.
In contrast, it is also well-known that regression adjustment for variables strongly associated with the outcome tends to improve the efficiency of OLS and other estimators. Hence, the parents of the outcome would appear a natural, easy to determine and more efficient alternative for adjustment. However, the parents of \(Y\) are not guaranteed to be a valid adjustment set; they may contain forbidden nodes, specifically mediators between treatment and outcome. For example, in Figure~\ref{fig:ExO}, FTW is a parent of the outcome DET, but a descendant of the treatment ALN and hence cannot be used for adjustment. Simply omitting such nodes from the parents of \(Y\) does not generally lead to a valid adjustment set either. For example, CDR alone does not form a valid adjustment set in Figure~\ref{fig:ExO}, since there are open confounding paths, e.g.\ \(ALN\leftarrow SAN\rightarrow AIS\rightarrow SUS\rightarrow FTW\rightarrow DET\).

Nonetheless, the intuition of using the parents of \(Y\) is correct if applied to a modified graph. As we show below, marginalising out, i.e.\ projecting over, the forbidden nodes results in a graph where the parent set of \(Y\) indeed coincides with the \(\mathbf{O}\)-set, and is thus guaranteed to yield an estimator with minimal asymptotic variance in the settings we consider, see Proposition \ref{eff}. This characterization of the \(\mathbf{O}\)-set thus combines validity, graphical simplicity and efficiency. We will now explain this formally.

Consider again the case of a DAG \(\mathcal{D}\) containing sets \(\mathbf{X}\) and \(\mathbf{Y}\). We first need the concept of latent projection, used to \textit{marginalise} or \textit{collapse over} latent, i.e.\ unobserved nodes, while preserving the remaining causal relations and (in)dependencies between the observed nodes.

\begin{definition}[Latent projection; \citealp{VermaPearl1990, Shpitseretal2014}]
	\label{lproj}
	
	\textcolor{white}{i} Let \(\mathcal{D}\) be a DAG with node set \(\mathbf{W}\cup \mathbf{L}\) and \(\mathbf{W}\cap \mathbf{L}=\emptyset\). The \emph{latent projection} \(\mathcal{D}(\mathbf{W})\) \emph{over} \(\mathbf{L}\) \emph{on} \(\mathbf{W}\) is a graph with node set \(\mathbf{W}\) and edges as follows: For distinct nodes \(W_i,W_j\in\mathbf{W}\),
	\begin{enumerate}
		\item \(\mathcal{D}(\mathbf{W})\) contains a directed edge \(W_i\rightarrow W_j\) if and only if \(\mathcal{D}\) contains a directed path \(W_i\rightarrow\dots\rightarrow W_j\) on which all non-endpoint nodes are in \(\mathbf{L}\),
		\item\(\mathcal{D}(\mathbf{W})\) contains a bi-directed edge \(W_i\leftrightarrow W_j\) if and only if \(\mathcal{D}\) contains a path, with at least one non-endpoint node, of the form \(W_i\leftarrow\dots\rightarrow W_j\) on which all non-endpoint nodes are non-colliders and in \(\mathbf{L}\).
	\end{enumerate}
\end{definition}

In the latent projection \(\mathcal{D}(\mathbf{W})\), two nodes may be connected by a directed and a bi-directed edge at the same time. (In)dependence relations can be read off from a latent projection using the \textit{m}-separation criterion \citep{Richardson2003}. For disjoint \(\mathbf{A},\mathbf{B},\mathbf{C}\subset\mathbf{W}\), \(\mathbf{A}\) and \(\mathbf{B}\) are \textit{d}-separated given \(\mathbf{C}\) in \(\mathcal{D}\) if and only if \(\mathbf{A}\) and \(\mathbf{B}\) are \textit{m}-separated given \(\mathbf{C}\) in \(\mathcal{D}(\mathbf{W})\) \citep{Richardsonetal2017}.

For our definition of the \(\mathbf{O}\)-set, we project over the forbidden nodes, save \(\mathbf{X}\) and \(\mathbf{Y}\), which motivates the following definition:

\begin{definition}[Forbidden projection]
	\label{fproj}
	Let \(\mathcal{D}\) be a DAG with node set \(\mathbf{V}\) and let \(\mathbf{X}\) and \(\mathbf{Y}\) be disjoint subsets of \(\mathbf{V}\). We call the graph \(\mathcal{D}^{\mathbf{XY}}=\mathcal{D}((\mathbf{V}\setminus\mathrm{forb}(\mathbf{X},\mathbf{Y},\mathcal{D}))\cup\mathbf{X}\cup\mathbf{Y})\) the \emph{forbidden projection} of \(\mathcal{D}\) with respect to \((\mathbf{X},\mathbf{Y})\).
\end{definition}

Figures \ref{fig:Examples} and \ref{fig:ExOproj} show some examples, where the forbidden nodes are shown as circles. In panels \(\mathbf{A}\) and \(\mathbf{B}\) of Figure \ref{fig:Examples}, the forbidden sets only contain nodes in \(\mathbf{X}\cup\mathbf{Y}\), hence nothing is projected over. Panels \(\mathbf{E}\) and \(\mathbf{F}\) show DAGs where the forbidden projection has bi-directed edges, which will become relevant in Proposition \ref{nobiDAG}.

While we primarily introduce the forbidden projection to provide an alternative characterisation of the \(\mathbf{O}\)-set, it is a useful tool in its own right. In particular, as we show next, the forbidden projection of a causal DAG preserves all information relevant to the estimation of a causal effect via adjustment. All proofs are given in Appendix~\ref{appB} and generalised to maxPDAGs in Appendix~\ref{appC}.

\newpage

\begin{figure}[H]
	\begin{tikzpicture}[node distance=20mm, >=latex]
	\node[anchor=north west, rectangle, draw, minimum width=16cm, minimum height=17cm] (outerframe) {};
	\node[anchor=north west, rectangle, draw, minimum width=8cm, minimum height=2.5cm] (Aframe) {};
	\node[xshift=80mm, anchor=north west, rectangle, draw, minimum width=8cm, minimum height=2.5cm] (Bframe) {};
	\node[yshift=-25mm, anchor=north west, rectangle, draw, minimum width=16cm, minimum height=3cm] (Cframe) {};
	\node[yshift=-55mm, anchor=north west, rectangle, draw, minimum width=16cm, minimum height=4.5cm] (Dframe) {};
	\node[yshift=-100mm, anchor=north west, rectangle, draw, minimum width=16cm, minimum height=3.5cm] (Eframe) {};
	\node[yshift=-135mm, anchor=north west, rectangle, draw, minimum width=16cm, minimum height=3.5cm] (Fframe) {};
	
	\node[anchor=north west, rectangle, draw, minimum size=3mm] (A) {\textbf{A}};
	\node[right of=A, xshift=-15mm, anchor=north west] (DA) {\(\mathcal{D}_\mathbf{A}=\mathcal{D}^{XY}_\mathbf{A}\)};
	\node[below of=A, xshift=20mm, yshift=5mm] (V1A) {\(V_1\)};
	\node[right of=V1A, circle, fill=gray!50, draw] (XA) {\(X\)};
	\node[right of=XA, xshift=-5mm, yshift=10mm, rectangle, draw] (V2A) {\(V_2\)};
	\node[right of=XA, xshift=10mm, circle, fill=gray!50, draw] (YA) {\(Y\)};
	
	\draw[->, thick] (V1A) to (XA);
	\draw[->, thick] (V2A) to (XA);
	\draw[->, thick] (V2A) to (YA);
	\draw[->, line width=0.7mm] (XA) to (YA);
	
	\node[xshift=80mm, anchor=north west, rectangle, draw, minimum size=3mm] (B) {\textbf{B}};
	\node[right of=B, xshift=-15mm, anchor=north west] (DB) {\(\mathcal{D}_\mathbf{B}=\mathcal{D}^{\mathbf{X}Y}_\mathbf{B}\)};
	\node[below of=B, xshift=10mm, yshift=5mm, circle, fill=gray!50, draw, scale=0.85] (X1B) {\(X_1\)};
	\node[right of=X1B] (V1B) {\(V_1\)};
	\node[right of=V1B, circle, fill=gray!50, draw] (YB) {\(Y\)};
	\node[above of=YB, xshift=-10mm, yshift=-10mm, rectangle, draw] (V2B) {\(V_2\)};
	\node[right of=YB, circle, fill=gray!50, draw, scale=0.85] (X2B) {\(X_2\)};
	
	\draw[->, thick] (X1B) to (V1B);
	\draw[->, thick] (V2B) to (YB);
	\draw[->, thick] (V2B) to (V1B);
	\draw[->, thick] (YB) to (X2B);
	
	\node[yshift=-25mm, anchor=north west, rectangle, draw, minimum size=3mm] (C) {\textbf{C}};
	\node[right of=C, xshift=-15mm, anchor=north west] (DC) {\(\mathcal{D}_\mathbf{C}\)};
	\node[below of=C] (V1C) {\(V_1\)};
	\node[right of=V1C, circle, fill=gray!50, draw] (XC) {X};
	\node[right of=XC, xshift=-10mm, yshift=15mm, rectangle, draw] (V2C) {\(V_2\)};
	\node[right of=XC, circle, draw, scale=0.95] (V3C) {\(V_3\)};
	\node[right of=V3C, circle, fill=gray!50, draw] (YC) {\(Y\)};
	\node[above of=YC, xshift=10mm, yshift=-5mm, rectangle, draw] (V4C) {\(V_4\)};	
	
	\draw[->, thick] (V1C) to (XC);
	\draw[->, thick] (V2C) to (XC);
	\draw[->, thick] (V2C) to (V3C);
	\draw[->, thick] (V4C) to (YC);
	\draw[->, line width=0.7mm] (XC) to (V3C);
	\draw[->, line width=0.7mm] (V3C) to (YC);
	
	\node[right of=DC, xshift=60mm] (DC*) {\(\mathcal{D}^{XY}_\mathbf{C}\)};
	\node[right of=YC] (V1C*) {\(V_1\)};
	\node[right of=V1C*, circle, fill=gray!50, draw] (XC*) {\(X\)};
	\node[above of=XC*, xshift=10mm, yshift=-5mm, rectangle, draw] (V2C*) {\(V_2\)};
	\node[right of=XC*, xshift=20mm, circle, fill=gray!50, draw] (YC*) {\(Y\)};
	\node[above of=YC*, xshift=10mm, yshift=-5mm, rectangle, draw] (V4C*) {\(V_4\)};
	
	\draw[->, thick] (V1C*) to (XC*);
	\draw[->, thick] (V2C*) to (XC*);
	\draw[->, thick] (V2C*) to (YC*);
	\draw[->, thick] (V4C*) to (YC*);
	\draw[->, line width=0.7mm] (XC*) to (YC*);	
	
	\node[yshift=-55mm, anchor=north west, rectangle, draw, minimum size=3mm] (D) {\textbf{D}};
	\node[right of=D, xshift=-15mm, anchor=north west] (DD) {\(\mathcal{D}_\mathbf{D}\)};
	\node[below of=D] (V1D) {\(V_1\)};
	\node[right of=V1D, circle, fill=gray!50, draw] (XD) {\(X\)};
	\node[above of=XD, yshift=-5mm, rectangle, draw] (V2D) {\(V_2\)};
	\node[below of=XD, yshift=5mm] (V3D) {\(V_3\)};
	\node[right of=XD, circle, draw, scale=0.9] (V5D) {\(V_5\)};
	\node[above of=V5D, yshift=-5mm, rectangle, draw] (V4D) {\(V_4\)};
	\node[below of=V5D, yshift=5mm, circle, draw, scale=0.9] (V6D) {\(V_6\)};
	\node[right of=V5D, circle, fill=gray!50, draw] (YD) {\(Y\)};
	\node[above of=YD, yshift=-5mm, rectangle, draw] (V7D) {\(V_7\)};
	\node[below of=YD, yshift=5mm, circle, draw, scale=0.9] (V8D) {\(V_8\)};
	
	\draw[->, thick] (V1D) to (XD);
	\draw[->, thick] (V2D) to (XD);
	\draw[->, thick] (XD) to (V3D);
	\draw[->, thick] (V2D) to (YD);
	\draw[->, thick] (V4D) to (V5D);
	\draw[->, thick] (V5D) to (V6D);
	\draw[->, thick] (V7D) to (YD);
	\draw[->, thick] (YD) to (V8D);
	\draw[->, line width=0.7mm] (XD) to (V5D);
	\draw[->, line width=0.7mm] (V5D) to (YD);
	
	\node[right of=DD, xshift=60mm] (DD*) {\(\mathcal{D}^{XY}_\mathbf{D}\)};
	\node[right of=YD] (V1D*) {\(V_1\)};
	\node[right of=V1D*, circle, fill=gray!50, draw] (XD*) {\(X\)};
	\node[above of=XD*, yshift=-5mm, rectangle, draw] (V2D*) {\(V_2\)};
	\node[below of=XD*, yshift=5mm] (V3D*) {\(V_3\)};
	\node[right of=V2D*, rectangle, draw] (V4D*) {\(V_4\)};
	\node[right of=XD*, xshift=20mm, circle, fill=gray!50, draw] (YD*) {\(Y\)};
	\node[above of=YD*, yshift=-5mm, rectangle, draw] (V7D*) {\(V_7\)};
	
	\draw[->, thick] (V1D*) to (XD*);
	\draw[->, thick] (V2D*) to (XD*);
	\draw[->, thick] (XD*) to (V3D*);
	\draw[->, thick] (V2D*) to (YD*);
	\draw[->, thick] (V4D*) to (YD*);
	\draw[->, thick] (V7D*) to (YD*);
	\draw[->, line width=0.7mm] (XD*) to (YD*);
	
	\node[yshift=-100mm, anchor=north west, rectangle, draw, minimum size=3mm] (E) {\textbf{E}};
	\node[right of=E, xshift=-15mm, anchor=north west] (DE) {\(\mathcal{D}_\mathbf{E}\)};
	\node[below of=DE, xshift=10mm, yshift=7mm, circle, fill=gray!50, draw] (XE) {\(X\)};
	\node[right of=XE, circle, draw, scale=0.85] (VE) {\(V_E\)};
	\node[right of=VE, yshift=10mm, circle, fill=gray!50, draw, scale=0.9] (Y1E) {\(Y_1\)};
	\node[right of=VE, yshift=-10mm, circle, fill=gray!50, draw, scale=0.9] (Y2E) {\(Y_2\)};
	
	\draw[->, line width=0.7mm] (XE) to (VE);
	\draw[->, line width=0.7mm] (VE) to (Y1E);
	\draw[->, line width=0.7mm] (VE) to (Y2E);
	
	\node[right of=DE, xshift=60mm] (DE*) {\(\mathcal{D}^{X\mathbf{Y}}_\mathbf{E}\)};
	\node[below of=DE*, xshift=10mm, yshift=7mm, circle, fill=gray!50, draw] (XE*) {\(X\)};
	\node[right of=XE*, xshift=20mm, yshift=10mm, circle, fill=gray!50, draw, scale=0.9] (Y1E*) {\(Y_1\)};
	\node[right of=XE*, xshift=20mm, yshift=-10mm, circle, fill=gray!50, draw, scale=0.9] (Y2E*) {\(Y_2\)};
	
	\draw[->, line width=0.7mm] (XE*) to (Y1E*);
	\draw[->, line width=0.7mm] (XE*) to (Y2E*);
	\draw[<->, thick, bend left=30] (Y1E*) to (Y2E*);
	
	\node[yshift=-135mm, anchor=north west, rectangle, draw, minimum size=3mm] (F) {\textbf{F}};
	\node[right of=F, xshift=-15mm, anchor=north west] (DF) {\(\mathcal{D}_\mathbf{F}\)};
	\node[below of=F, xshift=10mm, yshift=5mm, rectangle, draw] (V1F) {\(V_1\)};
	\node[right of=V1F, circle, fill=gray!50, draw, xshift=-5mm, scale=0.85] (X1F) {\(X_1\)};
	\node[right of=X1F, xshift=-5mm, circle, draw, scale=0.95] (V2F) {\(V_2\)};
	\node[right of=V2F, circle, fill=gray!50, draw, xshift=-5mm, scale=0.85] (X2F) {\(X_2\)};
	\node[right of=X2F, circle, fill=gray!50, draw, xshift=-5mm] (YF) {\(Y\)};
	
	\draw[->, thick] (V1F) to (X1F);
	\draw[->, line width=0.7mm] (X1F) to (V2F);
	\draw[->, thick] (V2F) to (X2F);
	\draw[->, line width=0.7mm] (X2F) to (YF);
	\draw[->, thick] (X1F) to [bend left=55] (X2F);
	\draw[->, line width=0.7mm] (V2F) to [bend right=55] (YF);
	
	\node[right of=DF, xshift=60mm] (DF*) {\(\mathcal{D}^{\mathbf{X}Y}_\mathbf{F}\)};
	\node[right of=YF, xshift=-2mm, rectangle, draw] (V1F*) {\(V_1\)};
	\node[right of=V1F*, circle, fill=gray!50, draw, xshift=-5mm, scale=0.85] (X1F*) {\(X_1\)};
	\node[right of=X1F*, circle, fill=gray!50, draw, xshift=5mm, scale=0.85] (X2F*) {\(X_2\)};
	\node[right of=X2F*, circle, fill=gray!50, draw, xshift=-5mm] (YF*) {\(Y\)};
	
	\draw[->, thick] (V1F*) to (X1F*);
	\draw[->, thick] (X1F*) to (X2F*);
	\draw[->, line width=0.7mm] (X2F*) to (YF*);
	\draw[->, thick] (V1F*) to [bend right=45] (YF*);
	\draw[->, line width=0.7mm] (X1F*) to [bend right] (YF*);
	\draw[->, thick] (V1F*) to [bend left=50] (X2F*);
	\draw[<->, thick] (X2F*) to [bend left=75] (YF*);
	
	\end{tikzpicture}
	\caption{Example DAGs with their forbidden projections. The forbidden nodes are shown as circles, nodes in the \(\mathbf{O}\)-set are shown as boxes. The bold arrows show the causal paths from \(\mathbf{X}\) to \(\mathbf{Y}\). In panels \textbf{A} and \textbf{B}, the original DAGs and their forbidden projections are identical. In panel \textbf{E}, the empty set is a valid adjustment set and also the \(\mathbf{O}\)-set. In panel \textbf{F}, the bi-directed edge between \(X_2\) and \(Y\) indicates that the effect of \(\mathbf{X}=\{X_1,X_2\}\) on \(Y\) is not identified via adjustment.}
	\label{fig:Examples}
\end{figure}
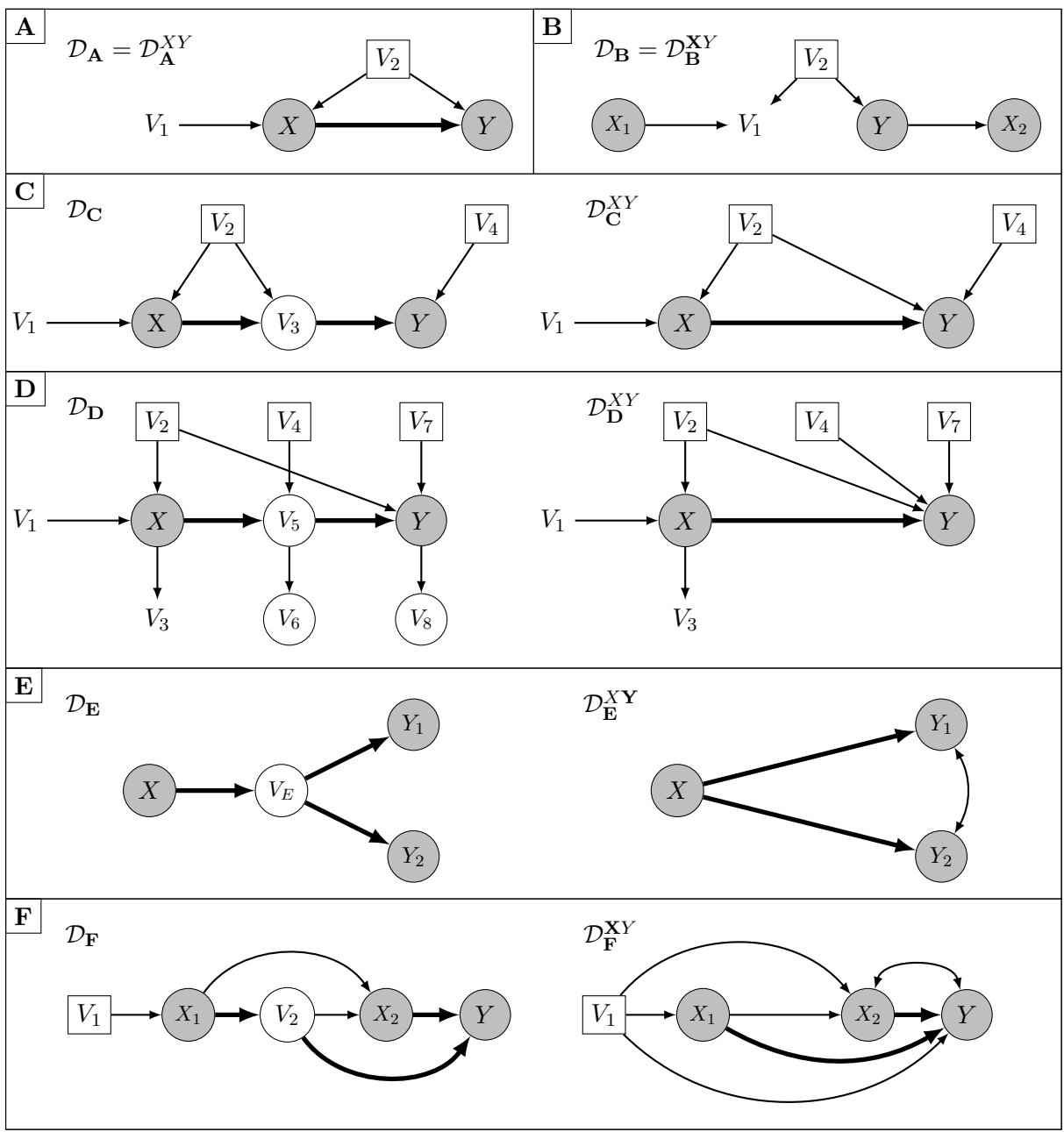

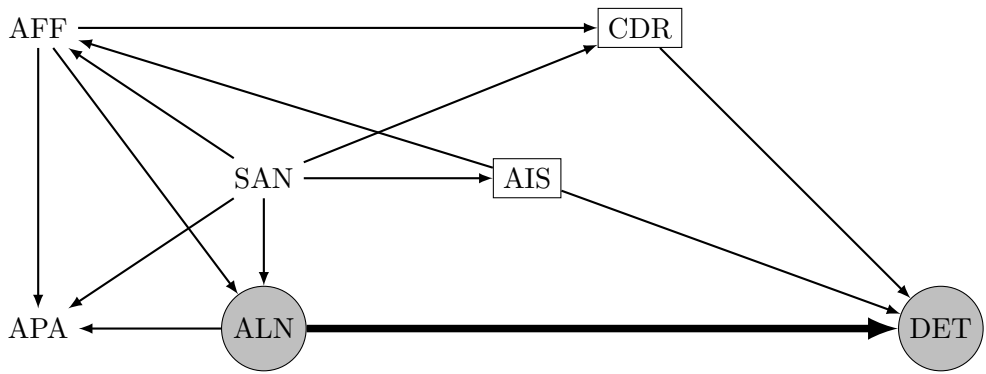
\begin{figure}[ht]
	\begin{center}
	\begin{tikzpicture}[node distance=20mm, >=latex]
	\node (AFF) {AFF};
	\node[below of=AFF, yshift=-20mm] (APA) {APA};
	\node[right of=APA, xshift=10mm, circle, fill=gray!50, draw, inner sep=0pt, minimum size=32pt] (ALN) {ALN};
	\node[right of=AFF, xshift=60mm, rectangle, draw] (CDR) {CDR};
	\node[right of=ALN, xshift=70mm, circle, fill=gray!50, draw, inner sep=0pt, minimum size=32pt] (DET) {DET};
	\node[above of=ALN] (SAN) {SAN};
	\node[right of=SAN, xshift=15mm, rectangle, draw] (AIS) {AIS};
	\draw[->, line width=1mm] (ALN) to (DET);
	\draw[->, thick] (AFF) to (ALN);
	\draw[->, thick] (SAN) to (AFF);
	\draw[->, thick] (SAN) to (AIS);
	\draw[->, thick] (SAN) to (APA);
	\draw[->, thick] (SAN) to (ALN);
	\draw[->, thick] (SAN) to (CDR);
	\draw[->, thick] (CDR) to (DET);
	\draw[->, thick] (AFF) to (APA);
	\draw[->, thick] (AIS) to (AFF);
	\draw[->, thick] (AFF) to (CDR);
	\draw[->, thick] (ALN) to (APA);
	\draw[->, thick] (AIS) to (DET);
	
	\end{tikzpicture}
	\end{center}
	\caption{Forbidden projection of the DAG in Figure~\ref{fig:ExO} with respect to \(\mathbf{X}=\{\text{ALN}\}\) and \(\mathbf{Y}=\{\text{DET}\}\). The forbidden nodes are shown as circles, nodes in the \(\mathbf{O}\)-set (parents of DET) are shown as boxes. The bold arrow shows the causal path from ALN to DET.}
	\label{fig:ExOproj}
\end{figure} 

First, the forbidden projection can be used to check whether a valid adjustment set exists relative to given sets of nodes \(\mathbf{X}\) and \(\mathbf{Y}\):

\begin{proposition}
	\label{nobiDAG}
	Let \(\mathbf{X}\) and \(\mathbf{Y}\) be disjoint node sets in a causal DAG \(\mathcal{D}\) such that \(\mathbf{Y}\subseteq\mathrm{de}(\mathbf{X},\mathcal{D})\). Then a valid adjustment set relative to \((\mathbf{X}\),\(\mathbf{Y})\) in \(\mathcal{D}\) exists if and only if there is no bi-directed edge between any \(X\in\mathbf{X}\) and \(Y\in\mathbf{Y}\) in \(\mathcal{D}^{\mathbf{XY}}\).
\end{proposition}

In Figure \ref{fig:Examples}, valid adjustment sets with respect to \(\mathbf{X}\) and \(\mathbf{Y}\) exist in all panels except for panel \(\mathbf{F}\). The effect of \(\mathbf{X}=\{X_1,X_2\}\) on \(Y\) in panel \(\mathbf{F}\) is, however, identified e.g.\ by the more general G-formula \citep{Robins1986, DawidDidelez2010}, the algorithm in \cite{TianPearl2003}, or the methods in \cite{NandyMaathuisRichardson2017}. See \cite{GuoPerkovic2020} and \citetalias{RotnitzkySmucler2019} for results on efficient adjustment in the linear and non-parametric case, respectively. The bi-directed edge between \(Y_1\) and \(Y_2\) in panel \(\mathbf{E}\) has no relevance in defining or determining a valid adjustment set.

For singleton \(Y\) such that a valid adjustment set with respect to \((\mathbf{X},Y)\) exists, the forbidden projection is particularly easy to interpret, as it is itself a causal DAG.
\begin{proposition}
	\label{XYDAG}
	Let \(\mathbf{X}\) and \(\{Y\}\) be disjoint node sets in a causal DAG \(\mathcal{D}\) such that \(Y\in\mathrm{de}(\mathbf{X},\mathcal{D})\). Then \(\mathcal{D}^{\mathbf{X}Y}\) is a causal DAG if and only if there exists a valid adjustment set relative to \((\mathbf{X},Y)\) in \(\mathcal{D}\).
\end{proposition}

Further, an adjustment set that is valid in the original graph is also valid in the forbidden projection and vice versa:
\begin{proposition}
	\label{VAS}
	Let \(\mathbf{X}\), \(\mathbf{Y}\) and \(\mathbf{Z}\) be disjoint node sets in a causal DAG \(\mathcal{D}\). Then $\mathbf{Z}$ is a valid adjustment set relative to \((\mathbf{X},\mathbf{Y})\) in $\mathcal{D}$ if and only if $\mathbf{Z}$ is also a valid adjustment set relative to \((\mathbf{X},\mathbf{Y})\) in $\mathcal{D}^{\mathbf{XY}}$.
\end{proposition}

Using the forbidden projection, we now define the \(\mathbf{O}^*\)-set and prove that it is equal to the \(\mathbf{O}\)-set.

\begin{definition}[\(\mathbf{O}^*\)-set]
	\label{Oset2}
	Let \(\mathbf{X}\) and \(\mathbf{Y}\) be disjoint node sets in a DAG \(\mathcal{D}\). We define \(\mathbf{O}^*(\mathbf{X},\mathbf{Y},\mathcal{D})\) as:
	\[\mathbf{O}^*(\mathbf{X},\mathbf{Y},\mathcal{D})=\mathrm{pa}(\mathbf{Y},\mathcal{D}^{\mathbf{XY}})\setminus(\mathbf{X}\cup\mathbf{Y}).\]
\end{definition}

In words, the \(\mathbf{O}^*\)-set is the set of parents of \(\mathbf{Y}\) in the forbidden projection \(\mathcal{D}^{\mathbf{XY}}\), excluding treatment nodes and outcome nodes. The next proposition states our key result.

\begin{proposition}
	\label{OOequal}	
	Let \(\mathbf{X}\) and \(\mathbf{Y}\) be disjoint subsets of the node set \(\mathbf{V}\) of a DAG \(\mathcal{D}\) such that \(\mathbf{Y}\subseteq\mathrm{de}(\mathbf{X}, \mathcal{D})\). Then \(\mathbf{O}(\mathbf{X},\mathbf{Y},\mathcal{D})=\mathbf{O}^*(\mathbf{X},\mathbf{Y},\mathcal{D})\).
\end{proposition}

It now follows trivially that the statements about the \(\mathbf{O}\)-set in Proposition \ref{eff} are true for the \(\mathbf{O}^*\)-set as well.

Again, \(\mathbf{Y}\subseteq\mathrm{de}(\mathbf{X},\mathcal{D})\) in Proposition \ref{OOequal} is not a severe restriction, because if \(\mathbf{Y}\not\subseteq\mathrm{de}(\mathbf{X},\mathcal{D})\), we can instead consider the effect on \(\mathbf{Y}\cap\mathrm{de}(\mathbf{X},\mathcal{D})\), as we know that the effect on \(\mathbf{Y}\setminus\mathrm{de}(\mathbf{X},\mathcal{D})\) is zero.

Figure~\ref{fig:ExOproj} shows the forbidden projection with respect to ALN and DET of the DAG in Figure~\ref{fig:ExO}. The \(\mathbf{O}\)-set \{AIS, CDR\} (in boxes) is the parent set of DET. All other valid adjustment sets are less efficient, for example the afore-mentioned sets \{AFF, SAN\}, \{AIS, CDR, AFF\} and \{AFF, APA, AIS, CDR, SAN\}. Due to Proposition \ref{VAS}, the validity of all of these sets can be confirmed by using the generalised adjustment criterion stated above on either the original DAG (Figure~\ref{fig:ExO}) or its forbidden projection (Figure~\ref{fig:ExOproj}). See Figure~\ref{fig:Examples} for further examples.

To summarise, the forbidden projection can be used as follows: First, check in the original graph if \(\mathbf{Y}\subseteq\mathrm{de}(\mathbf{X},\mathcal{D})\). Next, construct the forbidden projection \(\mathcal{G}^{\mathbf{XY}}\) and check for bi-directed edges. If there is a bi-directed edge between a node in \(\mathbf{X}\) and a node in \(\mathbf{Y}\), then the causal effect of interest is not identified via adjustment (Proposition \ref{nobiDAG}). Else, \(\mathcal{G}^{\mathbf{XY}}\) contains all information necessary to determine a valid adjustment set (Proposition \ref{VAS}), and in particular the \(\mathbf{O}\)-set, which is then the set of parents of \(\mathbf{Y}\) (Definition \ref{Oset2}, Proposition \ref{OOequal}). If \(\mathbf{Y}\) contains only one node, then \(\mathcal{G}^{\mathbf{XY}}\) is a causal DAG itself and hence straightforward to interpret (Proposition \ref{XYDAG}).

\section{Optimal Adjustment in the IDA Algorithm}
\label{sec:IDA}

In Sections \ref{sec:review} and \ref{sec:new_def}, we considered optimal adjustment in DAGs, and in Appendix~\ref{appC} we generalised the results to amenable maxPDAGs, which include amenable CPDAGs. As a reminder, a maxPDAG is said to be amenable relative to \((\mathbf{X},\mathbf{Y})\) if every proper possibly causal path from \(\mathbf{X}\) to \(\mathbf{Y}\) starts with a directed edge out of \(\mathbf{X}\). In this section, we consider non-amenable CPDAGs and maxPDAGs.

CPDAGs and maxPDAGs are of interest because they are the output of popular causal search algorithms, i.e.\ algorithms that attempt to learn a graph from data. Under the linear model with Gaussian error terms, which we focus on in this section, it is generally not possible to learn a unique DAG. Even under the additional assumptions of causal sufficiency and faithfulness (see e.g.\ \citealp{Spirtesetal2000}), one can at best learn a Markov equivalence class of DAGs, uniquely represented by a CPDAG (see e.g.\ \citealp{AnderssonMadiganPerlman1997}). Given additional knowledge of some causal relationships between variables, access to interventional data, or other model restrictions, one can obtain a refinement of this class, uniquely represented by a maxPDAG  \citep{Meek1995,PerkovicKalischMaathuis2017}. For a CPDAG or maxPDAG \(\mathcal{G}\), we use \([\mathcal{G}]\) to denote the set of DAGs that it represents.  The interpretation of edges in a CPDAG or maxPDAG \(\mathcal{G}\) is as follows: A directed edge \(A\rightarrow B\) means that this edge is present in all DAGs in \([\mathcal{G}]\). An undirected edge \(A-B\) means that \(A\) and \(B\) are adjacent in every DAG in \([\mathcal{G}]\) and there is at least one DAG in \([\mathcal{G}]\) with \(A\rightarrow B\) and at least one with \(A\leftarrow B\).

We suppose in this section that we are interested in a univariate exposure \(X\) and a univariate outcome \(Y\). For a given CPDAG or maxPDAG \(\mathcal{G}\), the true causal effect of \(X\) on \(Y\) may differ across the DAGs in \([\mathcal{G}]\). In particular, \cite{Perkovic2019} (Proposition 4.2) showed that assuming \(Y\not\in\mathrm{pa}(X,\mathcal{G})\), the true causal effect of \(X\) on \(Y\) differs across DAGs in \([\mathcal{G}]\) if and only if \(\mathcal{G}\) is non-amenable relative to \((X,Y)\), i.e.\ there is a possibly causal path from \(X\) to \(Y\) that starts with an undirected edge. Hence, when \(\mathcal{G}\) is non-amenable relative to \((X,Y)\), we can at best determine a multiset of possible causal effects \((\tau_{yx}(\mathcal{D}))_{\mathcal{D} \in [\mathcal{G}]}\), one for each DAG in \([\mathcal{G}]\). (A multiset \((\tau_{yx}(\mathcal{D}))_{\mathcal{D} \in [\mathcal{G}]}\) may contain the same entry multiple times, e.g.\ if \([\mathcal{G}]\) contains five DAGs, of which three imply an effect of 0 and two imply an effect of 1.2, then \((\tau_{yx}(\mathcal{D}))_{\mathcal{D} \in [\mathcal{G}]}=\{0,0,0,1.2,1.2\}\).) While obviously less informative than a single number, this multiset of possible causal effects may still yield useful statistics. The minimum absolute value, for example, is a lower bound for the size of the causal effect. However, enumerating all DAGs in \([\mathcal{G}]\) is computationally very expensive even for moderately sized \(\mathcal{G}\) when there are many undirected edges.

\citet{MaathuisKalischBuhlmann2009} proposed to reduce the complexity of this problem as follows. Consider two DAGs \(\mathcal{D}, \mathcal{D'} \in [\mathcal{G}]\) such that \(\mathrm{pa}(X,\mathcal{D})=\mathrm{pa}(X,\mathcal{D'})=\mathbf{P}\) and \(Y\not\in\mathbf{P}\). As the parents of \(X\) form a valid adjustment set (\citealp{Pearl2009},  p.\,72f.), \(\tau_{yx}(\mathcal{D})=\tau_{yx}(\mathcal{D'})=\tau_{yx}(\mathbf{P})\), where \(\tau_{yx}(\mathbf{P})\) denotes the coefficient of \(X\) in the linear regression of \(Y\) on \(X\) and \(\mathbf{P}\), i.e.\ \(\beta_{yx.\mathbf{p}}\). Let \(\mathbb{P}=\{\mathrm{pa}(X,\mathcal{D})\mid\mathcal{D}\in[\mathcal{G}]\}\) denote the set of all possible parent sets of \(X\) compatible with \(\mathcal{G}\). Then \((\tau_{yx}(\mathbf{P}))_{\mathbf{P} \in \mathbb{P}}\) contains the same distinct values as \((\tau_{yx}(\mathcal{D}))_{\mathcal{D} \in [\mathcal{G}]}\), while \(|\mathbb{P}| \le |[\mathcal{G}]|\). \citet{MaathuisKalischBuhlmann2009} showed that it is possible to determine \(\mathbb{P}\) locally from the CPDAG \(\mathcal{G}\) without enumerating all DAGs in \(\mathcal{G}\). They hence proposed a simple local procedure for calculating \((\hat{\tau}_{yx}(\mathbf{P}))_{\mathbf{P} \in \mathbb{P}}\), which is called `local IDA' (local \underline{I}ntervention Calculus when the \underline{D}AG is \underline{A}bsent). \citet{PerkovicKalischMaathuis2017} proposed a semi-local generalisation to maxPDAGs (`semi-local IDA'). 

The semi-local IDA algorithm for a maxPDAG is given in Algorithm~\ref{AlgL}. Let \(\mathrm{sib}(X,\mathcal{G})\) denote the set of nodes sharing an undirected edge with \(X\) in \(\mathcal{G}\). Semi-local IDA loops over all subsets \(\mathbf{S} \subseteq \mathrm{sib}(X,\mathcal{G})\). It first constructs a graph \(\mathcal{G}'\) such that \(\mathrm{pa}(X,\mathcal{G}')=\mathbf{P}=\mathrm{pa}(X,\mathcal{G})\cup\mathbf{S}\). Here the complexity reduction becomes apparent: only the edges adjacent to \(X\) need to be oriented. To verify whether the added orientations are compatible with the original graph \(\mathcal{G}\), the algorithm attempts to extend the graph to a maxPDAG by applying Meek's orientation rules (ConstructMaxPDAG algorithm; \citealp{Meek1995, PerkovicKalischMaathuis2017}; see Figure~\ref{fig:Meek} in Appendix~\ref{appA}). This step is semi-local as edges not adjacent to \(X\) need to be oriented. If successful, \(\hat{\beta}_{yx.\mathbf{p}}\) is added as a possible causal effect estimate, where \(\mathbf{P} = \mathrm{pa}(X,\mathcal{G'}) = \mathbf{S} \cup \mathrm{pa}(X,\mathcal{G})\).

\cite{NandyMaathuisRichardson2017} further generalised semi-local IDA to sets \(\mathbf{X}\) and \(\mathbf{Y}\). However, this procedure does not use regression adjustment for possible causal effect estimation and is therefore not directly related to our results.

\begin{algorithm}
	\caption{Local or semi-local IDA \citep{MaathuisKalischBuhlmann2009,PerkovicKalischMaathuis2017}.\newline When the input is a CPDAG, line 7 can be simplified and the algorithm becomes fully local.}
	\begin{algorithmic}[1]
		\REQUIRE CPDAG or maxPDAG \(\mathcal{G}\) with node set \(\mathbf{V}=\{V_1,\dots,V_p,X,Y\}\), i.i.d.\ observations for \(V_1,\cdots,V_p,X,Y\) 
		\ENSURE multiset of estimates \(\widehat{\Theta}\)
		\STATE \(\widehat{\Theta}\leftarrow\emptyset\)
		\STATE \(\mathrm{sib}(X,\mathcal{G})\leftarrow\{V \in \mathbf{V}: X-V \text{ in } \mathcal{G}\}\)
		\FORALL {\(\mathbf{S}\subseteq\mathrm{sib}(X,\mathcal{G})\)}
		\STATE \(\mathbf{LocalBg}\leftarrow\emptyset\)
		\STATE for all \(S\in\mathbf{S}\), add \(\{S\rightarrow X\}\) to \(\mathbf{LocalBg}\)
		\STATE for all \(S\in\mathrm{sib}(X,\mathcal{G})\setminus\mathbf{S}\), add \(\{S\leftarrow X\}\) to \(\mathbf{LocalBg}\)
		\STATE \(\mathcal{G}'\leftarrow\text{ConstructMaxPDAG}(\mathcal{G},\mathbf{LocalBg})\)
		\IF {\(\mathcal{G}'\ne\text{``FAIL"}\)}
		\IF {\(Y \notin \mathrm{pa}(X,\mathcal{G}')\)}
		\STATE regress \(Y\) on \(X\cup\mathrm{pa}(X,\mathcal{G}')\) and add the estimated coefficient of \(X\) to \(\widehat{\Theta}\)
		\ELSE 
		\STATE add 0 to \(\widehat{\Theta}\)
		\ENDIF
		\ENDIF
		\ENDFOR
		\RETURN \(\widehat{\Theta}\)
	\end{algorithmic}
	\label{AlgL}
\end{algorithm}

\begin{algorithm}
\caption{Optimal IDA.}
\begin{algorithmic}[1]
	\REQUIRE CPDAG or maxPDAG \(\mathcal{G}\) with node set \(\mathbf{V}=\{V_1,\dots,V_p,X,Y\}\), i.i.d.\ observations for \(V_1,\cdots,V_p,X,Y\) 
	\ENSURE multiset of estimates \(\widehat{\Theta}\)
	\STATE \(\widehat{\Theta}\leftarrow\emptyset\)
	\STATE \(\mathrm{sib}(X,\mathcal{G})\leftarrow\{V \in \mathbf{V}: X-V \text{ in } \mathcal{G}\}\)
	\FORALL {\(\mathbf{S}\subseteq\mathrm{sib}(X,\mathcal{G})\)}
	\STATE \(\mathbf{LocalBg}\leftarrow\emptyset\)
	\STATE for all \(S\in\mathbf{S}\), add \(\{S\rightarrow X\}\) to \(\mathbf{LocalBg}\)
	\STATE for all \(S\in\mathrm{sib}(X,\mathcal{G})\setminus\mathbf{S}\), add \(\{S\leftarrow X\}\) to \(\mathbf{LocalBg}\)
	\STATE \(\mathcal{G}'\leftarrow\text{ConstructMaxPDAG}(\mathcal{G},\mathbf{LocalBg})\)
	\IF {\(\mathcal{G}'\ne\text{``FAIL"}\)}
	\IF {\(Y\in\mathrm{possde}(X,\mathcal{G}')\)}
	\STATE regress \(Y\) on \(X\cup\mathbf{O}(X,Y,\mathcal{G}')\) and add the estimated coefficient of \(X\) to \(\widehat{\Theta}\)
	\ELSE
	\STATE add 0 to \(\widehat{\Theta}\)
	\ENDIF
	\ENDIF
	\ENDFOR
	\RETURN \(\widehat{\Theta}\)
\end{algorithmic}
\label{AlgO}
\end{algorithm}

\subsection{Optimal IDA}
\citetalias{HenckelPerkovicMaathuis2019} established that the parents of \(X\), as used for adjustment by semi-local IDA, form one of the least efficient valid adjustment sets. It therefore seems a good idea to replace \(\mathrm{pa}(X,\mathcal{D})\) by the \(\mathbf{O}\)-set within the IDA algorithm to improve estimation precision. The key question is, however, whether the possible \(\mathbf{O}\)-sets can still be determined semi-locally. More formally, our aim is to estimate the multiset \((\tau_{yx}(\mathbf{O}))_{\mathbf{O} \in \mathbb{O}}\), \(\mathbb{O}= \{\mathbf{O}(X,Y,\mathcal{D}) \,|\, \mathcal{D} \in [\mathcal{G}]\}\), where with a slight abuse of notation we define \(\tau_{yx}(\mathbf{O})=0\) if \(Y\not\in\mathrm{possde}(X,\mathcal{D})\). As before, for two DAGs \(\mathcal{D}\) and \(\mathcal{D'}\) with the same valid \(\mathbf{O}\)-set \(\mathbf{O}(X,Y,\mathcal{D})=\mathbf{O}(X,Y,\mathcal{D'})=\mathbf{O}\), we have \(\tau_{yx}(\mathcal{D})=\tau_{yx}(\mathcal{D'})=\tau_{yx}(\mathbf{O})\).

At first glance, it appears impossible to determine \(\mathbb{O}\) locally or semi-locally, as by Definitions \ref{Oset1} and \ref{Oset2}  the causal nodes, their parents and the forbidden nodes, or the forbidden projection, are required to find the \(\mathbf{O}\)-set. However, it turns out that \(\mathbb{O}\) can be determined semi-locally almost in the same manner as \(\mathbb{P}\). This is because once the directions of all edges involving \(X\) are given, i.e.\ for given \(\mathbf{P}\), application of Meek's rules reveals all descendants of \(X\) and, in consequence, all causal nodes, their parents and the forbidden nodes (cf.\ Lemma \ref{mylemma} in Appendix~\ref{appC}). Hence, via Meek's rules there exists a correspondence between possible parent sets and possible \(\mathbf{O}\)-sets. We therefore propose Algorithm~\ref{AlgO}, which we call optimal IDA. It is implemented in the \texttt{R} package \texttt{pcalg} \citep{Kalischetal2012, Kalischetal2019}.

Algorithm~\ref{AlgO} does not specify whether \(\mathbf{O}(X,Y,\mathcal{G'})\) is determined from \(\mathcal{G}'\) or from the forbidden projection. We expect this choice to be of limited relevance to the algorithm's runtime. In our implementation, we determine \(\mathbf{O}(X,Y,\mathcal{G'})\) directly from \(\mathcal{G}'\). Note also that different possible parent sets can correspond to the same \(\mathbf{O}\)-set. Hence, optimal IDA could be modified to collect all sets in \(\mathbb{O}\) first, remove duplicates, and only then estimate regression coefficients.

In the following, we first state formally what can be said about the efficiency of the estimates output by optimal IDA, showing that it is worthwhile to replace the parents of \(X\) by the \(\mathbf{O}\)-set. Subsequently we compare the computational burden of the two algorithms.

\begin{proposition}
	\label{propIDA}
	Let $X$ and $Y$ be nodes in a causal CPDAG or maxPDAG $\mathcal{G} = (\mathbf{V},\mathbf{E})$, such that $\mathbf{V}$ follows a causal linear model compatible with $\mathcal{G}$ with Gaussian errors. Let $\widehat{\Theta}^{\mathbf{P}}$ and $\widehat{\Theta}^{\mathbf{O}}$ be the multisets returned by semi-local IDA and optimal IDA, respectively, applied to \(X\), \(Y\) and \(\mathcal{G}\), with the subsets of \(\mathrm{sib}(X,\mathcal{G})\) considered in the same order for both. Then, for $i \in\{1 \dots, k\}$, with $k = |\widehat{\Theta}^{\mathbf{P}}| = |\widehat{\Theta}^{\mathbf{O}}|$,
	\begin{enumerate}
		\item $\mathbb{E}[\widehat{\Theta}^{\mathbf{P}}_i] = \mathbb{E}[\widehat{\Theta}^{\mathbf{O}}_i]$ 
		and 
		\item $a.var(\widehat{\Theta}^{\mathbf{P}}_i) \geq a.var(\widehat{\Theta}^{\mathbf{O}}_i)$. 
	\end{enumerate}
\end{proposition}

The proof is given in Appendix~\ref{appD}. Note that if we do not assume Gaussianity in Proposition \ref{propIDA}, then \(a.var(\hat{\beta}_{yx.\mathbf{o}}) \leq a.var(\hat{\beta}_{yx.\mathbf{z}})\) can only be guaranteed if (i) \(\mathbf{Z}\) is a valid adjustment set in the true DAG, and (ii) \(\mathbf{O}\) is the \(\mathbf{O}\)-set of the true DAG. This is because in a causal linear model with non-Gaussian errors, a variable is only required to be linear in its parents, and is not necessarily linear given another node set (cf. \citealp{NandyMaathuisRichardson2017}). However, if we are willing to assume that all errors in the underlying causal model are non-Gaussian, alternative causal search approaches exist which output a DAG instead of an equivalence class, e.g.\ algorithms such as LiNGAM \citep{Shimizuetal2006}.
\begin{remark}
	(1)	In terms of the computational burden, semi-local and optimal IDA are very similar for maxPDAGs. The key difference is that optimal IDA adjusts for the \(\mathbf{O}\)-set instead of the parent set of \(X\) (line 10), where the \(\mathbf{O}\)-set is straightforward to determine from \(\mathcal{G'}\). However, optimal IDA crucially relies on the construction of the maxPDAG in line 7 to determine the \(\mathbf{O}\)-set, while in semi-local IDA this step can be replaced by a simple local query when the input is known to be a CPDAG. Hence, for the special case of a CPDAG, semi-local IDA can be made fully local by simplifying line 7, whereas optimal IDA cannot.
	
	(2)	A further minor difference between semi-local and optimal IDA is the \texttt{if}-statement in line 9. Semi-local IDA only checks whether \(Y \notin \mathrm{pa}(X,\mathcal{G}')\), whereas optimal IDA checks the stronger condition \(Y \in \mathrm{possde}(X,\mathcal{G'})\). Both conditions ensure that the considered adjustment sets \(\mathrm{pa}(X,\mathcal{G}')\) and \(\mathbf{O}(X,Y,\mathcal{G'})\), respectively, are valid adjustment sets. Moreover, if \(Y\not\in\mathrm{possde}(X,\mathcal{G}')\), then \(\tau_{yx}(\mathcal{D})=0\) for any \(\mathcal{D} \in [\mathcal{G'}]\). The \(0\) estimate of optimal IDA in this case is therefore the most efficient estimate. Alternatively, we could also insist on \(Y \in \mathrm{possde}(X,\mathcal{G'})\) in semi-local IDA and return \(0\) otherwise. As discussed in the appendix of \cite{MaathuisKalischBuhlmann2009}, this is only recommended if the input graph is thought to be reliable, but can lead to the amplification of errors if the input graph is not accurate.
\end{remark}

\begin{remark}
	\label{IDAse}
	Proposition \ref{propIDA} concerns the asymptotic variance when the true CPDAG or a true maxPDAG is given. When the graph is estimated on the same data as used for IDA, the naive standard errors from the adjusted linear regressions are invalid. Although considerable progress has been made in the area of post-selection inference (e.g.\ \citealp{Berketal2013, BelloniChernozhukovHansen2014,RinaldoWassermanGSell2019}), no method has been proposed specifically for estimating standard errors of causal effect estimates after causal search.
\end{remark}

It is straightforward to extend optimal IDA to situations where  \(\mathbf{X}\) and \(\mathbf{Y}\) are sets.
However, as noted earlier, in this case joint causal effect estimation via regression adjustment is not always possible. Optimal IDA will then not return an estimate. The estimation procedures used by joint IDA \citep{NandyMaathuisRichardson2017} provide an alternative.

\subsection{Illustration}
We now illustrate optimal IDA (Algorithm~\ref{AlgO}) using a toy example. Consider the CPDAG \(\mathcal{G}\) shown in Figure~ \ref{fig:ExIDA}(a) and suppose we are interested in the causal effect of \(X\) on \(Y\). Clearly, \(\mathcal{G}\) is not amenable relative to \((X,Y)\) and thus it is sensible to apply optimal IDA.

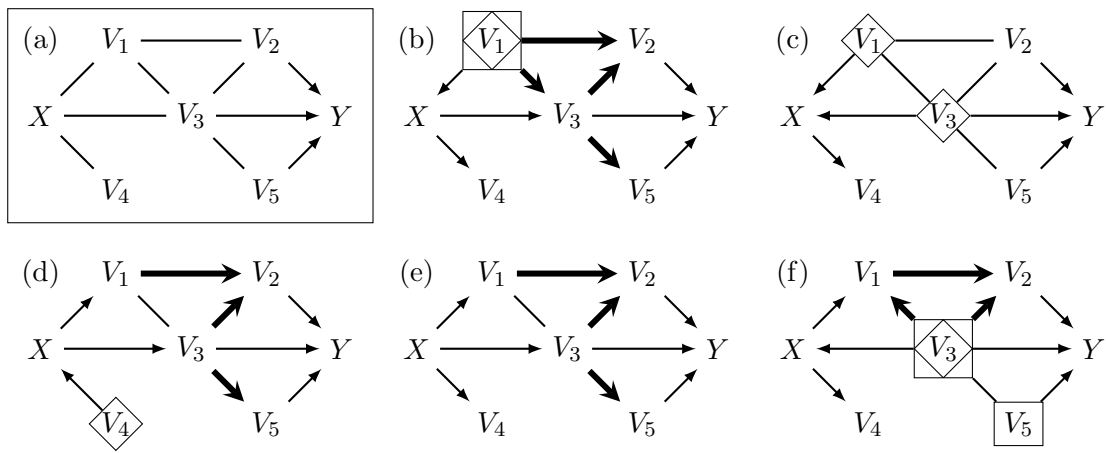
\begin{figure}[p]
	\begin{center}
	\begin{tikzpicture}[node distance=20mm, >=latex]
	\node (X) {\(X\)};
	\node[above of=X, yshift=-10mm] (a) {(a)};
	\node[right of=X] (V3) {\(V_3\)};
	\node[right of=V3] (Y) {\(Y\)};
	\node[right of=X, xshift=-10mm, yshift=10mm] (V1) {\(V_1\)};
	\node[right of=X, xshift=-10mm, yshift=-10mm] (V4) {\(V_4\)};
	\node[right of=V3, xshift=-10mm, yshift=10mm] (V2) {\(V_2\)};
	\node[right of=V3, xshift=-10mm, yshift=-10mm] (V5) {\(V_5\)};
	\draw[-, thick] (X) to (V1);
	\draw[-, thick] (X) to (V4);
	\draw[-, thick] (X) to (V3);
	\draw[-, thick] (V1) to (V2);
	\draw[-, thick] (V1) to (V3);
	\draw[-, thick] (V2) to (V3);
	\draw[-, thick] (V3) to (V5);
	\draw[->, thick] (V2) to (Y);
	\draw[->, thick] (V3) to (Y);
	\draw[->, thick] (V5) to (Y);
	
	\node[draw, fit=(X) (V1) (V5) (Y)] {};
	
	\node[right of=Y, xshift=-10mm] (Xe) {\(X\)};
	\node[above of=Xe, yshift=-10mm] (b) {(b)};
	\node[right of=Xe] (V3e) {\(V_3\)};
	\node[right of=V3e] (Ye) {\(Y\)};
	\node[both, draw, minimum size=22pt, right of=Xe, xshift=-10mm, yshift=10mm] (V1e) {\(V_1\)};
	\node[right of=Xe, xshift=-10mm, yshift=-10mm] (V4e) {\(V_4\)};
	\node[right of=V3e, xshift=-10mm, yshift=10mm] (V2e) {\(V_2\)};
	\node[right of=V3e, xshift=-10mm, yshift=-10mm] (V5e) {\(V_5\)};
	\draw[latex-, thick] (Xe) to (V1e);
	\draw[->, thick] (Xe) to (V4e);
	\draw[->, thick] (Xe) to (V3e);
	\draw[-stealth, line width=0.8mm] (V1e) to (V2e);
	\draw[-stealth, line width=0.8mm] (V1e) to (V3e);
	\draw[stealth-, line width=0.8mm] (V2e) to (V3e);
	\draw[-stealth, line width=0.8mm] (V3e) to (V5e);
	\draw[->, thick] (V2e) to (Ye);
	\draw[->, thick] (V3e) to (Ye);
	\draw[->, thick] (V5e) to (Ye);
	
	\node[right of=Ye, xshift=-10mm] (Xd) {\(X\)};
	\node[above of=Xd, yshift=-10mm] (c) {(c)};
	\node[diamond, draw, inner sep=0pt, minimum size=1pt, right of=Xd] (V3d) {\(V_3\)};
	\node[right of=V3d] (Yd) {\(Y\)};
	\node[diamond, draw, inner sep=0pt, minimum size=1pt, right of=Xd, xshift=-10mm, yshift=10mm] (V1d) {\(V_1\)};
	\node[right of=Xd, xshift=-10mm, yshift=-10mm] (V4d) {\(V_4\)};
	\node[right of=V3d, xshift=-10mm, yshift=10mm] (V2d) {\(V_2\)};
	\node[right of=V3d, xshift=-10mm, yshift=-10mm] (V5d) {\(V_5\)};
	\draw[<-, thick] (Xd) to (V1d);
	\draw[->, thick] (Xd) to (V4d);
	\draw[<-, thick] (Xd) to (V3d);
	\draw[-, thick] (V1d) to (V2d);
	\draw[-, thick] (V1d) to (V3d);
	\draw[-, thick] (V2d) to (V3d);
	\draw[-, thick] (V3d) to (V5d);
	\draw[->, thick] (V2d) to (Yd);
	\draw[->, thick] (V3d) to (Yd);
	\draw[->, thick] (V5d) to (Yd);
	
	\node[below of=X, yshift=-11mm] (Xc) {\(X\)};
	\node[above of=Xc, yshift=-10mm] (d) {(d)};
	\node[right of=Xc] (V3c) {\(V_3\)};
	\node[right of=V3c] (Yc) {\(Y\)};
	\node[right of=Xc, xshift=-10mm, yshift=10mm] (V1c) {\(V_1\)};
	\node[diamond, draw, inner sep=0pt, minimum size=1pt, right of=Xc, xshift=-10mm, yshift=-10mm] (V4c) {\(V_4\)};
	\node[right of=V3c, xshift=-10mm, yshift=10mm] (V2c) {\(V_2\)};
	\node[right of=V3c, xshift=-10mm, yshift=-10mm] (V5c) {\(V_5\)};
	\draw[->, thick] (Xc) to (V1c);
	\draw[<-, thick] (Xc) to (V4c);
	\draw[->, thick] (Xc) to (V3c);
	\draw[-stealth, line width=0.8mm] (V1c) to (V2c);
	\draw[-, thick] (V1c) to (V3c);
	\draw[stealth-, line width=0.8mm] (V2c) to (V3c);
	\draw[-stealth, line width=0.8mm] (V3c) to (V5c);
	\draw[->, thick] (V2c) to (Yc);
	\draw[->, thick] (V3c) to (Yc);
	\draw[->, thick] (V5c) to (Yc);
	
	\node[right of=Yc, xshift=-10mm] (Xb) {\(X\)};
	\node[above of=Xb, yshift=-10mm] (e) {(e)};
	\node[right of=Xb] (V3b) {\(V_3\)};
	\node[right of=V3b] (Yb) {\(Y\)};
	\node[right of=Xb, xshift=-10mm, yshift=10mm] (V1b) {\(V_1\)};
	\node[right of=Xb, xshift=-10mm, yshift=-10mm] (V4b) {\(V_4\)};
	\node[right of=V3b, xshift=-10mm, yshift=10mm] (V2b) {\(V_2\)};
	\node[right of=V3b, xshift=-10mm, yshift=-10mm] (V5b) {\(V_5\)};
	\draw[->, thick] (Xb) to (V1b);
	\draw[->, thick] (Xb) to (V4b);
	\draw[->, thick] (Xb) to (V3b);
	\draw[-stealth, line width=0.8mm] (V1b) to (V2b);
	\draw[-, thick] (V1b) to (V3b);
	\draw[stealth-, line width=0.8mm] (V2b) to (V3b);
	\draw[-stealth, line width=0.8mm] (V3b) to (V5b);
	\draw[->, thick] (V2b) to (Yb);
	\draw[->, thick] (V3b) to (Yb);
	\draw[->, thick] (V5b) to (Yb);
	
	\node[right of=Yb, xshift=-10mm] (Xf) {\(X\)};
	\node[above of=Xf, yshift=-10mm] (f) {(f)};
	\node[both, draw, minimum size=22pt, right of=Xf] (V3f) {\(V_3\)};
	\node[right of=V3f] (Yf) {\(Y\)};
	\node[right of=Xf, xshift=-10mm, yshift=10mm] (V1f) {\(V_1\)};
	\node[right of=Xf, xshift=-10mm, yshift=-10mm] (V4f) {\(V_4\)};
	\node[right of=V3f, xshift=-10mm, yshift=10mm] (V2f) {\(V_2\)};
	\node[rectangle, draw, right of=V3f, xshift=-10mm, yshift=-10mm] (V5f) {\(V_5\)};
	\draw[->, thick] (Xf) to (V1f);
	\draw[->, thick] (Xf) to (V4f);
	\draw[<-, thick] (Xf) to (V3f);
	\draw[-stealth, line width=0.8mm] (V1f) to (V2f);
	\draw[stealth-, line width=0.8mm] (V1f) to (V3f);
	\draw[stealth-, line width=0.8mm] (V2f) to (V3f);
	\draw[-, thick] (V3f) to (V5f);
	\draw[->, thick] (V2f) to (Yf);
	\draw[->, thick] (V3f) to (Yf);
	\draw[->, thick] (V5f) to (Yf);
	
	\end{tikzpicture}
	\end{center}
	\caption{A CPDAG \(\mathcal{G}\) (a) and the five maxPDAGs (b-f) corresponding to the five valid orientations of the neighbourhood of \(X\). The bold edges have been obtained by applying Meek's rules. For each maxPDAG \(\mathcal{G'}\), the boxes \(\square\) indicate \(\mathbf{O}(X,Y,\mathcal{G'})\), while the diamonds \(\Diamond\) indicate \(\mathrm{pa}(X,\mathcal{G})\). In (c), optimal IDA returns 0, as there is no possibly causal path from \(X\) to \(Y\).}
	\label{fig:ExIDA}
\end{figure}

\begin{figure}[p]
\begin{center}
	\includegraphics[height=6.8cm, trim=0cm 0.7cm 0 2.5cm, clip]{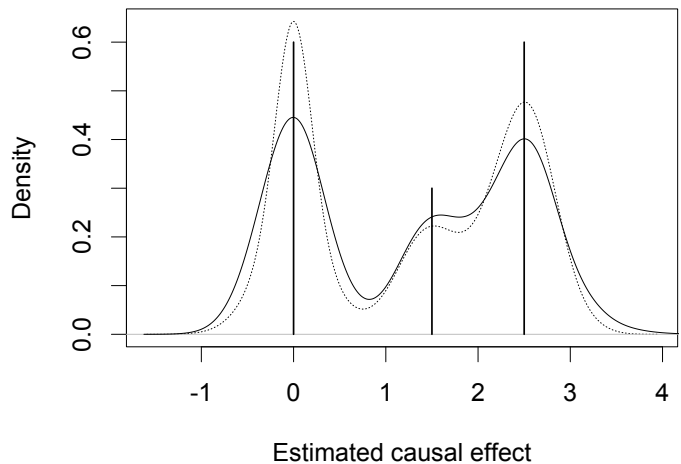}
\end{center}
\caption{IDA density plot in the style of \cite{MaathuisKalischBuhlmann2009}. Shown are density curves for the estimated possible causal effects returned by local IDA (solid) and optimal IDA (dotted). The true possible causal effects are 0, 1.5 and 2.5 (vertical lines; height indicates relative frequency: 0 and 2.5 each occur in two of the five maxPDAGs in Figure~\ref{fig:ExIDA}).}
\label{fig:density}
\end{figure}

The set \(sib(X,\mathcal{G})\) contains 3 nodes, hence there are 8 potential orientations of the undirected edges with endpoint \(X\). From these 8, 3 imply new v-structures and are thus not compatible with \(\mathcal{G}\). The other 5 can be extended to the maxPDAGs shown in Figure~\ref{fig:ExIDA}(b-f), where the bold arrows indicate orientations derived by Meek's rules (see Figure~\ref{fig:Meek} in Appendix~\ref{appA}). For example, in \ref{fig:ExIDA}(b) it follows from \(V_1\rightarrow X\rightarrow V_3\) that \(V_1\rightarrow V_3\) by Meek's Rule 2. By Rule 1, it then follows that \(V_3\rightarrow V_5\). The compatibility check and the application of Meek's rules are carried out in line 7 of optimal IDA.

Next, optimal IDA checks for each maxPDAG \(\mathcal{G}'\), whether \(Y\in\mathrm{possde}(X,\mathcal{G'})\). Here, this is the case for all maxPDAGs except \ref{fig:ExIDA}(c). For the other four graphs, \(\mathbf{O}=\mathbf{O}(X,Y, \mathcal{G'})\) is determined and used to compute \(\hat{\beta}_{yx.\mathbf{o}}\). We indicate \(\mathbf{O}(X,Y, \mathcal{G'})\) by boxes in the Figures \ref{fig:ExIDA}(b) and \ref{fig:ExIDA}(d)-(f). For (c), an effect estimate of zero is returned.

For comparison, the diamonds in Figures~\ref{fig:ExIDA}(b-f) show the adjustment sets in local IDA (Algorithm~\ref{AlgL}), i.e.\ \(\mathrm{pa}(X,\mathcal{G'})\). In (b) and (e), \(\mathrm{pa}(X,\mathcal{G'})=\mathbf{O}(X,Y,\mathcal{G'})\). In (c), optimal IDA returns zero (Algorithm~\ref{AlgO}, line 12), while local IDA returns \(\hat{\beta}_{yx.\mathbf{p}}\) with \(\mathbf{P}=\{V_1,V_3\}\), which converges to \(\beta_{yx.\mathbf{p}}=0\). The main advantage of optimal IDA becomes apparent in cases (d) and (f): In (d), \(\mathbf{O}(X,Y,\mathcal{G'})=\emptyset\), whereas \(\mathrm{pa}(X,\mathcal{G'})=\{V_4\}\) which is guaranteed to reduce efficiency. In (f), \(\mathrm{pa}(X,\mathcal{G'})=\{V_3\}\) and \(\mathbf{O}(X,Y,\mathcal{G'})=\{V_3,V_5\}\), where the latter improves efficiency.

For further illustration, we carried out a small simulation study in which we generated data according to a causal linear model compatible with Figure~\ref{fig:ExIDA}(b). 1\,000 datasets with 40 observations each were generated and given as input to local IDA and optimal IDA, together with the CPDAG in Figure~\ref{fig:ExIDA}(a). The true possible causal effects are 0, 1.5 and 2.5, visualised as vertical lines in Figure~\ref{fig:density}. The plot shows smoothed density curves for the estimates returned by local IDA (solid) and optimal IDA (dotted). The density plot for optimal IDA is clearly narrower around the values 0 and 2.5. The difference between the algorithms is even more pronounced for graphs with more nodes and longer paths (not shown). The \texttt{R}-code \citep{R} for reproducing Figure~\ref{fig:density} is available in the Online Supplement.

\subsection{Simulation}
In order to compare the performance of optimal versus local IDA in finite sample settings, we carried out a more extensive simulation study. The design was chosen to reflect a typical situation where IDA is used, i.e.\ interest lies in the causal effect of \(X\) on \(Y\) in a (known or estimated) CPDAG \(\mathcal{G}\) that is non-amenable relative to \((X,Y)\). Non-amenability implies that the multiset \((\tau_{xy}(\mathcal{D}))_{\mathcal{D} \in [\mathcal{G}]}\) of possible causal effects of \(X\) on \(Y\) compatible with \(\mathcal{G}\) contains more than one distinct value (for almost all parameters values of the causal linear model) (\citealp{Perkovic2019}, Proposition 4.2). A useful summary of \((\tau_{xy}(\mathcal{D}))_{\mathcal{D} \in [\mathcal{G}]}\) is the minimum absolute value, \(\min(\mathrm{abs}((\tau_{xy}(\mathcal{D}))_{\mathcal{D} \in [\mathcal{G}]}))\), because when this value is non-zero, we know that \(X\) has \textit{some} effect on \(Y\). The aim of our simulation study was to compare how well \(\min(\mathrm{abs}((\tau_{xy}(\mathcal{D}))_{\mathcal{D} \in [\mathcal{G}]}))\) is estimated by optimal versus local IDA, in terms of the Monte-Carlo mean squared error (MSE).

We investigated 24 scenarios by considering all combinations of the following parameters: number of nodes \(p\in\{10,20,50,100\}\), expected number of neighbours per node \(d\in\{2,3,4\}\), and sample size \(n\in\{100,1\,000\}\). In each scenario, the following was repeated 1\,000 times (\texttt{R} code for reproducing the simulation study is available  in the Online Supplement):

A DAG \(\mathcal{D}\), with CPDAG \(\mathcal{G}\), with \(p\) nodes and \(d\) expected neighbours per node was randomly chosen such that \(\mathcal{G}\) was non-amenable relative to two randomly chosen nodes \((X,Y)\) and such that \(\min(\mathrm{abs}((\tau_{xy}(\mathcal{D}))_{\mathcal{D} \in [\mathcal{G}]}))\) was non-zero. (Note that the DAG with its unique `true' causal effect was simulated for convenience only. Conceptually, we drew directly from the space of CPDAGs, which is why we consider the whole multiset of possible effects to be `the truth'.) The following was then repeated 100 times: A dataset with \(n\) observations was generated from a linear causal model on \(\mathcal{D}\) where the non-zero coefficients were randomly chosen from a uniform distribution on \([-1,-0.1]\cup[0.1,1]\). Greedy equivalence search \citep{Chickering2002} was applied to the data, yielding an estimated CPDAG \(\mathcal{G}^*\). Optimal and local IDA were both applied to the true CPDAG \(\mathcal{G}\) and the estimated CPDAG \(\mathcal{G}^*\). The four output multisets of estimates were summarised by their minimum absolute values. These were compared on the basis of their Monte-Carlo MSE, i.e.\ the squared difference between the estimated minimum absolute value and the true minimum absolute value, averaged over the 100 repetitions. Specifically, we calculated the MSEs for the estimated minima using optimal IDA versus local IDA by computing the relative MSE (RMSE), MSE(optimal IDA)/MSE(local IDA). This was done separately for \(\mathcal{G}\) and for \(\mathcal{G}^*\), and denoted \(r\) and \(r^*\), respectively. An RMSE of less than one indicates that optimal IDA is more precise than local IDA in estimating the minimum of the multiset of causal effects. In light of Remark~\ref{IDAse}, we did not consider estimated standard errors.

In addition to the above 24 scenarios, we investigated the relative performance of optimal IDA in a scenario where the graph is sparse (\(d=1\)) and the sample size is moderate (\(n=100\)). We considered eight setting where the number of nodes was between \(p=10\) and \(p=1\,000\). Such high-dimensional scenarios occur, for instance, with gene expression data. As greedy equivalence search is slow for large graphs, we reduced the number of replications from 1\,000 to 100 and the number of datasets per graph from 100 to 10.

\begin{figure}[h]
	\begin{center}
	\includegraphics[height=8cm, trim=0 0.5cm 0 0, clip]{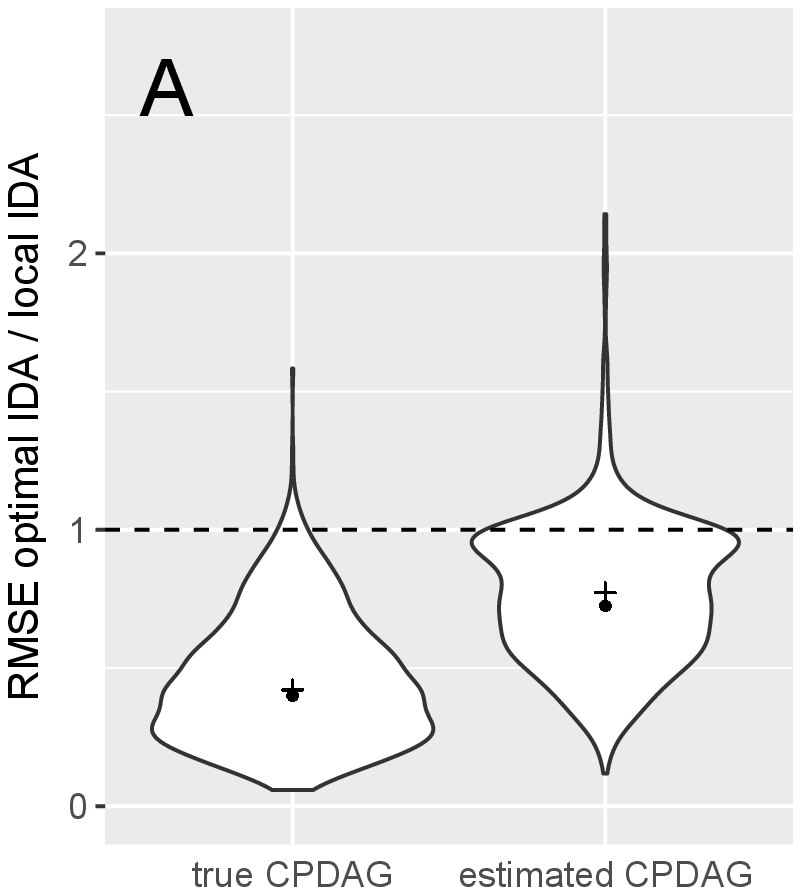}
	\includegraphics[height=8cm, trim=-1cm 0.5cm 0 0, clip]{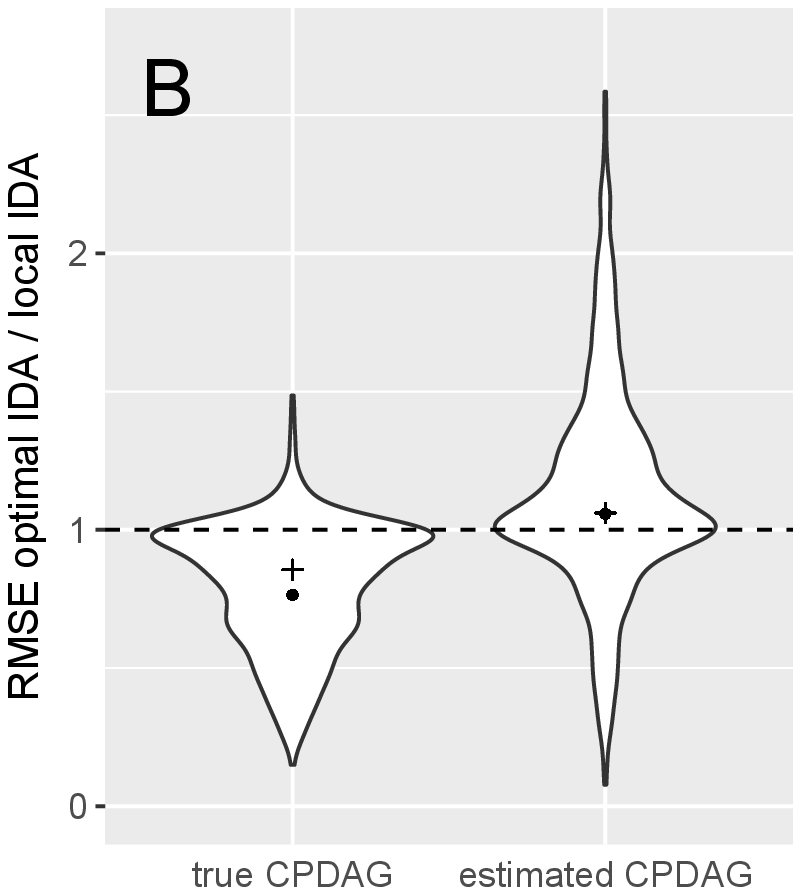}
	\end{center}
	\caption{Violin plots of the relative mean squared errors (RMSEs)  \(r\) and \(r^*\) for the true and estimated CPDAGs, respectively. Scenario \textbf{\textsf{A}}: \(p=100\) nodes, \(d=4\) expected neighbours per node, sample size \(n=1\,000\). Scenario \textbf{\textsf{B}}: \(p=10\), \(d=4\) and \(n=100\). The dots mark the geometric means, the plus signs the medians.}
	\label{fig:violin_main22}
\end{figure}

Figure~\ref{fig:violin_main22} shows violin plots of the RMSEs \(r\) and \(r^*\) over the 1\,000 repetitions, together with the geometric mean and the median. Two scenarios are shown: The one where optimal IDA showed the best overall performance (scenario \textsf{\textbf{A}}, \(p=100, d=4, n=1\,000\)), and the worst one (scenario \textsf{\textbf{B}}, \(p=10, d=4, n=100\)) of all the simulation settings considered. The geometric means and medians for all scenarios are summarised in Tables \ref{tab1} and \ref{tab2}; the complete set of violin plots is shown in Appendix~\ref{appE}. Optimal IDA clearly outperformed local IDA, in terms of the geometric mean and median of the RMSE, in all scenarios when applied to the true CPDAG. When the CPDAG was estimated using greedy equivalence search, optimal IDA was still superior in the majority of scenarios, but \(r^*\) was notably larger than \(r\) in all scenarios, i.e.\ the relative performance of optimal IDA was worse with an estimated CPDAG than with a known CPDAG. As an estimated graph inevitably contains some errors regarding the presence and direction of edges, this result may indicate that estimation adjusting for  the \(\mathbf{O}\)-set suffers more from such errors than adjusting for the set of parents of \(X\).

\begin{table}[t]
	\begin{center}
	\begin{small}
		\begin{tabular}{lcccccc}
			&&&&&&\\
			\toprule
			&  & \(n=100\) &  &  & \(n=1\,000\) & \\
			\cmidrule(rl){2-4}\cmidrule(rl){5-7}
			& \(d=2\) & \(d=3\) & \(d=4\) & \(d=2\) & \(d=3\) & \(d=4\)\\
			\midrule
			\(p=10\) & \(0.70\) \((0.76)\) & \(0.72\) \((0.79)\) & \(0.76\) \((0.86)\) & \(0.69\) \((0.78)\) & \(0.71\) \((0.78)\) & \(0.75\) \((0.86)\)\\
			\(p=20\) & \(0.64\) \((0.69)\) & \(0.63\) \((0.68)\) & \(0.61\) \((0.66)\) & \(0.63\) \((0.68)\) & \(0.60\) \((0.65)\) & \(0.59\) \((0.65)\)\\
			\(p=50\) & \(0.60\) \((0.64)\) & \(0.54\) \((0.57)\) & \(0.51\) \((0.55)\) & \(0.55\) \((0.58)\) & \(0.50\) \((0.54)\) & \(0.46\) \((0.49)\)\\
			\(p=100\) & \(0.57\) \((0.61)\) & \(0.50\) \((0.52)\) & \(0.44\) \((0.46)\) & \(0.54\) \((0.58)\) & \(0.44\) \((0.46)\) & \(0.40\) \((0.42)\)\\
			\bottomrule
		\end{tabular}
	\end{small}
	\end{center}
\caption{Geometric means (in parentheses: medians) of the relative mean squared errors (RMSEs) \(r\) over 1\,000 repetitions for scenarios with different numbers of nodes (\(p\)), expected number of neighbours per node (\(d\)), and sample sizes (\(n\)). Optimal and local IDA were applied to the true CPDAG \(\mathcal{G}\).}	\label{tab1}
\end{table}
\begin{table}[t]
	\begin{center}
	\begin{small}
		\begin{tabular}{lcccccc}
			&&&&&&\\
			\toprule
			&  & \(n=100\) &  &  & \(n=1\,000\) & \\
			\cmidrule(rl){2-4}\cmidrule(rl){5-7}
			& \(d=2\) & \(d=3\) & \(d=4\) & \(d=2\) & \(d=3\) & \(d=4\)\\
			\midrule
			\(p=10\) & \(1.06\) \((1.01)\) & \(1.06\) \((1.04)\) & \(1.06\) \((1.06)\) & \(0.95\) \((0.99)\) & \(0.97\) \((1.00)\) & \(1.01\) \((1.00)\) \\
			\(p=20\) & \(0.99\) \((1.00)\) & \(0.99\) \((1.00)\) & \(0.96\) \((1.00)\) & \(0.88\) \((0.96)\) & \(0.89\) \((0.97)\) & \(0.94\) \((0.99)\)\\
			\(p=50\) & \(0.94\) \((0.98)\) & \(0.89\) \((0.93)\) & \(0.89\) \((0.93)\) & \(0.81\) \((0.90)\) & \(0.79\) \((0.85)\) & \(0.78\) \((0.86)\)\\
			\(p=100\) & \(0.97\) \((1.00)\) & \(0.94\) \((0.97)\) & \(0.90\) \((0.94)\) & \(0.81\) \((0.91)\) & \(0.73\) \((0.80)\) & \(0.72\) \((0.77)\)\\
			\bottomrule
		\end{tabular}
	\end{small}
	\end{center}
\caption{Geometric means (in parentheses: medians) of the relative mean squared errors (RMSEs)  \(r^*\) over 1\,000 repetitions for scenarios with different numbers of nodes (\(p\)), expected number of neighbours per node (\(d\)), and sample sizes (\(n\)). Optimal and local IDA were applied to the estimated CPDAG \(\mathcal{G}^*\).}	\label{tab2}
\end{table}

Small \(n\) and small \(p\) do not entail much advantage of using optimal IDA: In graphs with only a few nodes, the \(\mathbf{O}\)-set and the set of parents of \(X\) are often similar or even coincide, so that the gain in efficiency when using the \(\mathbf{O}\)-set is less pronounced. A smaller sample size leads to more errors in the estimated graph, which affects estimation of the \(\mathbf{O}\)-set more than estimation of the set of parents of \(X\), as we conjectured above. However, optimal IDA seems to have a slight advantage for larger \(d\) when \(p\) is also larger.

\begin{figure}[p]
	\begin{center}
		\includegraphics[width=15cm, trim=0 0 0 0, clip]{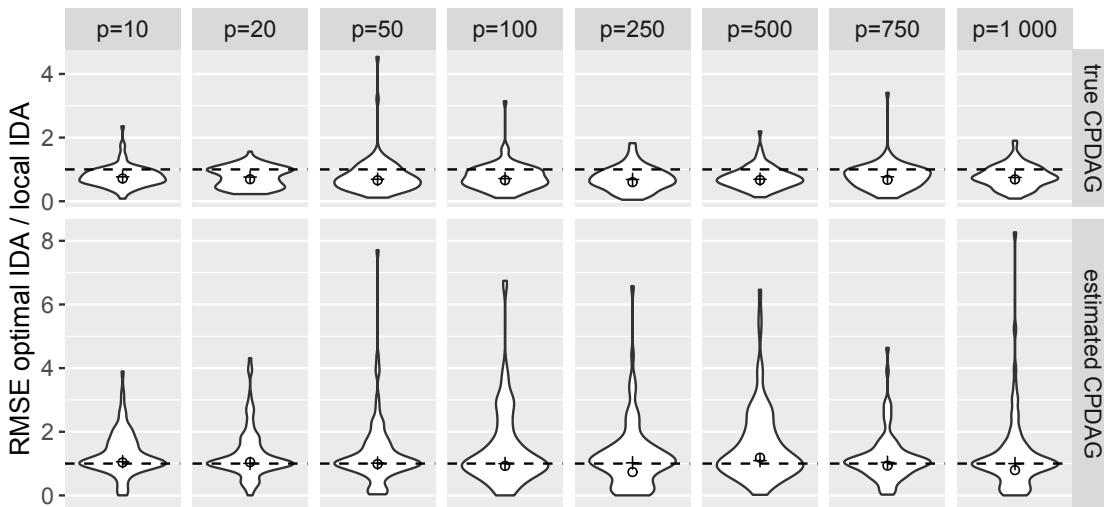}
	\end{center}
	\caption{Violin plots of the relative mean squared errors (RMSEs) \(r\) and \(r^*\) for the true and estimated CPDAGs, respectively. All graphs have \(d=1\) expected neighbour per node and graphs were estimated from \(n=100\) observations; the number of nodes \(p\) varies. The dots mark the geometric means, the plus signs the medians.}
	\label{fig:violin_sparse}
\end{figure}

\begin{table}[p]
	\begin{center}
	\begin{small}
		\begin{tabular}{lcccc}
			&&&&\\
			\toprule
			& \(p=10\) & \(p=20\) & \(p=50\) & \(p=100\)\\
			\midrule
			true CPDAG & \(0.71\) \((0.77)\) & \(0.69\) \((0.76)\) & \(0.66\) \((0.69)\) & \(0.66\) \((0.71)\) \\
			estimated CPDAG & \(1.04\) \((1.06)\) & \(1.04\) \((1.01)\) & \(0.99\) \((1.02)\) & \(0.93\) \((1.00)\) \\ 
			\toprule
			 & \(p=250\) & \(p=500\) & \(p=750\) & \(n=1\,000\)\\
			 \midrule
			 true CPDAG &  \(0.60\) \((0.68)\) & \(0.67\) \((0.69)\) & \(0.67\) \((0.77)\) & \(0.69\) \((0.75)\)\\
			 estimated CPDAG & \(0.73\) \((1.02)\) & \(1.16\) \((1.09)\) & \(0.94\) \((1.04)\) & \(0.79\) \((1.00)\)\\
			\bottomrule
		\end{tabular}
	\end{small}
	\end{center}
\caption{Geometric means (in parentheses: medians) of the relative mean squared errors (RMSEs) \(r\) and \(r^*\) for the true and estimated CPDAGs, respectively, over 100 repetitions for scenarios with different numbers of nodes (\(p\)), \(d=1\) expected neighbours per node, and sample size \(n=100\).}	\label{tab3}
\end{table}

The additional results for the sparse graphs are shown in Figure~\ref{fig:violin_sparse} and Table~\ref{tab3}. Optimal IDA outperformed local IDA regardless of the number of nodes \(p\) when the CPDAG was known. When the CPDAG was not known, the median RMSE was about 1 for all \(p\). The geometric mean varied around 1 with no obvious pattern, which may be due to the small number of replications. The results suggest that in the high-dimensional setting, the primary difficulty is learning the graph, limiting the advantage of optimal IDA over local IDA.

In summary, based on the simulation results, we recommend using optimal IDA when there is high confidence in the estimated graph. The advantage over local IDA will be most pronounced when the number of nodes is at least 20, or better 50 or more.

\section{The \(\mathbf{O}\)-Set and Non-Graphical Variable Selection}
\label{sec:data}

We now assume that neither the causal DAG \(\mathcal{D}\) nor a CPDAG or maxPDAG is known to us, therefore we wish to select a valid adjustment set in a non-graphical manner. We restrict our discussion to the case where we have a univariate treatment \(X\) and outcome \(Y\) of interest. 

In multiple regression analyses, it is common to apply variable selection procedures, e.g.\ backward selection, to find a set of relevant predictors for an outcome $Y$. In high-dimensional settings, regularisation methods that combine selection and estimation, such as the Lasso or the Elastic Net, are commonly used \citep{Tibshirani1996, ZouHastie2005}. While variable selection for prediction is in general a different task than finding an efficient or optimal adjustment set for causal effect estimation, we will discuss next under what assumptions and modifications these tasks coincide. For a general overview of the relation between variable and confounder selection see \cite{WitteDidelez2019}. A basic assumption for the validity of a selected adjustment set is that the set \(\mathbf{Z}\) from which we select the variables must itself be a valid adjustment set, as defined in Section ~\ref{sec:review}. A number of selection procedures can then be used to determine different types of valid adjustment sets as subsets of \(\mathbf{Z}\), e.g.\ a minimal valid adjustment set \citep{WitteDidelez2019, deLunaWaernbaumRichardson2011}.

Consider first Algorithm~\ref{BRS}, which shows the template for backward regression selection (see e.g.\ \citealp{KleinbaumKupper1978, MontgomeryPeckVining2012}) with the above basic assumption added at the outset.
Under the linear model assumptions with Gaussian errors, \(Y\ind Z_i\mid(\mathbf{Z}'_{-i}, X)\) can be tested by comparing the models with regressors \(\mathbf{Z}'_{-i}\cup\{X\}\) versus \(\mathbf{Z}'\cup\{X\}\), using a t-test with null hypothesis \(\beta_{yz_i.x\mathbf{z}'_{-i}}=0\). `Pval' in line 6 is a function that outputs the p-value of a test for the null hypothesis specified in the argument. The maximum p-value is compared in line 9 to a threshold \(\alpha\). For a given \(\alpha\), Algorithm~\ref{BRS} implements the classical `p-value method' (see e.g.\ \citealp{GreenlandPearce2015}). Denote by \(F_{\chi^2_1}(.)\) the distribution function of the \(\chi^2\) distribution with one degree of freedom. For a given sample size \(n\), Algorithm~\ref{BRS} with \(\alpha=1-F_{\chi^2_1}(2)\) or \(\alpha=1-F_{\chi^2_1}(\log(n))\) is equivalent to backward selection using the AIC or BIC, respectively (e.g.\ \citealp{Murtaugh2014, Derryberryetal2018}; although the motivation for using them stems from frameworks other than independence testing, see \citealp{Akaike1974} and \citealp{Schwarz1978}). Algorithm~\ref{BRS} can easily be adapted to work with a measure of conditional independence other than the p-value of the t-test. For example, two  non-parametric implementations of Algorithm~\ref{BRS} were proposed by \cite{LiCookNachtsheim2005}.

If the true independence relations are known, Algorithm~\ref{BRS} can be condensed to its oracle version, Algorithm~\ref{BRS2}. Comparing p-values is then redundant, and every \(Z_i\) needs to be visited only once, as it follows from general properties of conditional independence that the ordering \(Z_1, Z_2, \dots, Z_p\) does not matter, provided the joint probability of all variables is strictly positive. Essentially, Algorithm~\ref{BRS2} eliminates variables until only the `direct predictors' of \(Y\) are left, i.e.\ those variables with non-zero coefficients in the oracle regression of \(Y\) on \(X\) and \(\mathbf{Z}\).

Algorithm~\ref{BRS2} is the non-graphical version of the pruning algorithm introduced in \citetalias{HenckelPerkovicMaathuis2019} which uses d-separation relationships to prune a valid adjustment set to a subset such that the resultant effect estimator has a smaller asymptotic variance. Assume now that an underlying graph exists. The following Proposition \ref{new3.6} formalises how the \(\mathbf{O}\)-set can be viewed as the target set of backward variable selection algorithms and follows from Proposition 3.6 of \citetalias{HenckelPerkovicMaathuis2019} and Theorem 1 in \citetalias{RotnitzkySmucler2019}.

\begin{algorithm}
	\caption{Backward regression selection.}
	\begin{algorithmic}[1]
		\REQUIRE i.i.d.\ observations for variables \(X\), \(Y\) and \(\mathbf{Z}\), such that \(\mathbf{Z}\) is a valid adjustment set relative to \((X,Y)\)
		\STATE \(\mathbf{Z}'\leftarrow\mathbf{Z}\)
		\STATE \(\mathrm{Pmax}\leftarrow 1\)
		\WHILE {\(\mathrm{Pmax}>\alpha\)}
		\STATE \(\mathrm{Plist}\leftarrow\) empty list of length \(|\mathbf{Z}'|\)
		\FORALL {\(i\) in 1 to \(|\mathbf{Z}'|\)}
		\STATE \(\mathrm{Plist}[i]\leftarrow\mathrm{Pval}(Y\ind Z_i\mid(X,\mathbf{Z}'_{-i}))\)
		\ENDFOR
		\STATE \(\mathrm{Pmax}\leftarrow\max(\mathrm{Plist})\)
		\IF {\(\mathrm{Pmax}>\alpha\)}
		\STATE \(\mathbf{Z}'\leftarrow\mathbf{Z}'_{-\mathrm{argmax}(\mathrm{Plist})}\)
		\ENDIF
		\ENDWHILE
		\RETURN \(\mathbf{Z}'\)
	\end{algorithmic}
	\label{BRS}
\end{algorithm}

\begin{algorithm}
	\caption{Oracle backward regression selection.}
	\begin{algorithmic}[1]
		\REQUIRE independence relations between variables \(X\), \(Y\) and \(\mathbf{Z}\), such that \(\mathbf{Z}\) is a valid adjustment set relative to \((X,Y)\)
		\STATE \(\mathbf{Z}'\leftarrow\mathbf{Z}\)
		\FORALL {\(i\) in 1 to \(|\mathbf{Z}'|\)}
		\IF {\(Y\ind Z_i\mid(X,\mathbf{Z}'_{-i})\)}
		\STATE \(\mathbf{Z}'\leftarrow\mathbf{Z}'_{-i}\)
		\ENDIF
		\ENDFOR
		\RETURN \(\mathbf{Z}'\)
	\end{algorithmic}
	\label{BRS2}
\end{algorithm}

\begin{proposition}
	\label{new3.6}
	Let \(X\) and \(Y\) be nodes in a causal DAG, CPDAG or maxPDAG \(\mathcal{G}\) with node set \(\mathbf{V}\) and let \(\mathbf{V}\) follow a causal model with a joint density faithful to \(\mathcal{G}\). Let \(\mathbf{Z}\) be a valid adjustment set relative to \((X,Y)\) in \(\mathcal{G}\) and let \(\mathbf{Z}'\) be the output of Algorithm~\ref{BRS2} when applied to \(\mathbf{Z}\).
	\begin{enumerate}
		\item[(a)] \(\mathbf{Z}'\) is a valid adjustment set and does not depend on the order in which the variables in \(\mathbf{Z}\) are considered in Algorithm~\ref{BRS2}.
		\item[(b)] If \(\mathbf{O}(X,Y,\mathcal{G}) \subseteq \mathbf{Z}\), then \(\mathbf{Z}'=\mathbf{O}(X,Y,\mathcal{G})\).
		\item[(c)] If \(\mathbf{V}\) follows a causal linear model, then \(a.var(\hat{\beta}_{yx.\mathbf{z}'})\le a.var(\hat{\beta}_{yx.\mathbf{z}})\).
		\item[(d)] For every pair of values \(x, x'\in\mathcal{X}\), \(a.var(\hat{\Delta}_{yxx'.\mathbf{z}'}) \le a.var(\hat{\Delta}_{yxx'.\mathbf{z}})\).
		\end{enumerate}
\end{proposition}

As mentioned earlier, Lasso estimation can also be regarded as variable selection, even though its original motivation and common usage mostly concerns prediction. Under specific assumptions, the Lasso asymptotically selects all and only all the `direct predictors' of \(Y\) with probability 1 \citep{ZhaoYu2006, Lounici2008}. Thus, although Lasso uses a different principle than backward selection, it follows from Proposition \ref{new3.6} that when the starting set \(\mathbf{Z}\) is a valid adjustment set and a superset of the \(\mathbf{O}\)-set, the \(\mathbf{O}\)-set can also be viewed as the target set of the Lasso.

We emphasise again that the \(\mathbf{O}\)-set cannot be determined in a purely data-driven way. Neither the assumption that \(\mathbf{Z}\) is valid nor \(\mathbf{O}(X,Y,\mathcal{G})\subseteq\mathbf{Z}\) can be verified empirically. Hence, prior causal knowledge is essential before any variable selection algorithm can be applied \citep{WitteDidelez2019}. In contrast to (semi-)local or optimal IDA, however, selection of an adjustment set based on Algorithm~\ref{BRS2} allows some latent structures, as long as the assumption that \(\mathbf{Z}\) is a valid adjustment set continues to hold. This may be of advantage when only a subset of the variables have been measured.

The guarantees of Proposition \ref{new3.6} for Algorithm~\ref{BRS2} do not translate to the finite sample version Algorithm~\ref{BRS}. Regression selection in finite samples is known to have several weaknesses (see e.g.\ \citealp{Harrell2010}). Some issues are that the output may only be a local optimum, and that valid post-selection inference is difficult \citep{LeebPoetscher2008}. There is, however, a growing literature on post-selection inference both in the context of OLS-based approaches (e.g. \citealp{Berketal2013,  RinaldoWassermanGSell2019}) and of Lasso-based approaches (e.g.\ \citealp{Lockhartetal2014, Leeetal2016}). For causal effect estimation in a non-graphical context, post-selection inference has been considered by \cite{BelloniChernozhukovHansen2014}, \cite{DukesVansteelandt2020}, \cite{DukesVansteelandt2020b}, \cite{Chernozhukovetal2018} and others.

\section{Conclusions}

In this paper, we provided insight into the construction and properties of the \(\mathbf{O}\)-set introduced by \citetalias{HenckelPerkovicMaathuis2019}. We showed that the \(\mathbf{O}\)-set equals the set of parents of \(\mathbf{Y}\) in the latent projection over the forbidden nodes (Proposition \ref{OOequal}). This lends formal support to the intuition that adjusting for all direct causes of \(\mathbf{Y}\) minimises the residual variance and hence improves precision when estimating the causal effect of \(\mathbf{X}\) on \(\mathbf{Y}\).

The forbidden projection is a useful tool in its own right when the aim is to estimate a causal effect via adjustment. It displays all variables of interest, while the forbidden variables, which must not be adjusted for, are marginalised out. The forbidden projection thus reduces the complexity of the causal graph while preserving all information relevant for choosing an adjustment set (see Propositions \ref{VAS} and \ref{XYDAG}). 

We further proposed a new modification of the IDA algorithm, called optimal IDA, which outputs multisets of estimates of possible causal effects by adjusting for the possible \(\mathbf{O}\)-sets. We showed that this increases estimation precision also in cases  where the causal structure is a-priori unknown and needs to be estimated. Moreover, this extends the applicability of optimal adjustment to non-amenable CPDAGs/maxPDAGs. Optimal IDA has been implemented in the \texttt{R} package \texttt{pcalg}. While causal search methods in general have some well-known shortcomings, IDA has proved to be a valuable tool for instance for screening purposes in large datasets \citep{Leetal2013,Engelmannetal2015,LuoHuangCao2018}. The `optimal' version can further improve its performance.

Finally, we detailed the prerequisites and assumptions under which non-graphical algorithms for backward variable selection can be viewed as aiming at selecting the \(\mathbf{O}\)-set. Essentially, we need to assume that the set of variables to select from consists of all nodes in the forbidden projection, or a suitable subset thereof. The algorithm then determines the parents / direct causes of \(Y\) based on detected conditional independencies. If the input contains forbidden nodes, however, or lacks certain confounders, the algorithm might select an invalid adjustment set. To avoid the latter, sufficient prior knowledge on the set of variables corresponding to forbidden nodes is therefore a key prerequisite when automated variable selection is to be used for causal inference. While this prerequisite may not require full knowledge of the underlying causal DAG, it is important to recognise that such prior knowledge cannot be established in a purely data-driven way \citep{WitteDidelez2019}.

Much research on variable selection in causal graphs has focussed on finding small or minimal adjustment sets \citep{deLunaWaernbaumRichardson2011,TextorLiskiewicz2011,KnueppelStang2010}. Small adjustment sets are useful during study planning, for instance when data collection is expensive and costs are to be minimised. Moreover, they entail desirable statistical properties e.g.\ for matching estimators, because suitable matches are more easily found when matching on a few variables only. 
In general, the \(\mathbf{O}\)-set is not minimal, but instead entails optimality of causal effect estimation by regression adjustment in linear causal models and non-parametric settings.  Simulation results further indicate that the optimality of the \(\mathbf{O}\)-set extends to other parametric settings and estimation methods, e.g.\ estimation of the marginal odds ratio via standardised logistic regression \citep{WitteDidelez2019}. Combining the benefits of small and optimal adjustment sets, \citetalias{RotnitzkySmucler2019} show that the optimal minimal set, i.e.\ the set among all minimal adjustment sets yielding the most precise estimation in the class of regular asymptotically linear estimators, must be a subset of the \(\mathbf{O}\)-set, underlining its relevance and importance.

We note that adjustment is only one of several possible ways of identifying causal effects. While adjusting for the \textbf{O}-set is asymptotically more efficient than adjusting for any other valid adjustment set, it is possible that an even smaller asymptotic variance can be obtained by using an alternative identification strategy, e.g.\ the front-door strategy \citep{Pearl2009}. This is further investigated in \citetalias{RotnitzkySmucler2019} and in \cite{GuoPerkovic2020}.

Finally, we would like to discuss some avenues for future research. First, given the results by \citetalias{RotnitzkySmucler2019}, a natural question is whether a non-parametric version of optimal IDA is feasible. Those aspects of IDA that relate to finding different possible valid adjustment sets are obviously not limited to the causal linear model, and estimators such as in \citetalias{RotnitzkySmucler2019} could also be employed for any given \(X\), \(Y\) and adjustment set. The simplifications for graph search algorithms and IDA under linearity, however, are considerable. For instance, greedy equivalence search with a Gaussian score has been shown to be consistent \citep{Chickering2002}, and has the advantage of always returning a CPDAG. Non-parametric graph search algorithms exist, but often come with large computational burdens and/or a low power to detect edges \citep{ShahPeters2020, Ramsey2014}. Further, under a causal linear model, the causal effect of \(X\) on \(Y\) is a single value, see Section~\ref{sec:review}; and the marginal and conditional causal effects for different valid adjustment sets are all identical. For the non-parametric case, instead, the causal effect of \(X\) on \(Y\) is an unspecified function, and issues of non-collapsibility might also come into play. Solving these conceptual hurdles for non-parametric optimal IDA remains an open question.

Second, we assumed throughout this paper that all variables are observed. HPM19 have shown that in the presence of hidden variables, an asymptotically optimal set may not exist. \cite{SmuclerSapienzaRotnitzky2020}, however, gave a sufficient condition under which an optimal adjustment set exists when the underlying DAG includes hidden variables, and showed that an optimal minimal adjustment set always exists. A necessary and sufficient condition for the existence of an optimal adjustment set, however, has not yet been formulated.


\acks{We thank the reviewers and the editor for their valuable comments and suggestions. We gratefully acknowledge financial support by the German Research Foundation (DFG---Project DI 2372/1-1).}


\newpage

\appendix
\section[Appendix A: Terminology]{Terminology}
\label{appA}

The following terminology is used throughout this paper. It is consistent with, and extends \citetalias{HenckelPerkovicMaathuis2019} where needed.

\textbf{Graphs.} A graph \(\mathcal{G}=(\mathbf{V},\mathbf{E})\) consists of a node set \(\mathbf{V}\) and a set of edges \(\mathbf{E}\). We consider three types of edges: \textit{directed} (\(\rightarrow\)), \textit{bi-directed} (\(\leftrightarrow\)) and \textit{undirected} (\(-\)). There can be more than one edge between a given pair of nodes. We only consider loop-free graphs, i.e.\ an edge between a node and itself is not allowed. A loop-free graph where there is at most one edge between a given pair of nodes is called a \textit{simple graph}. Two nodes joined by at least one edge are called \textit{endpoints of the edge} and \textit{adjacent}. A directed edge \(A\rightarrow B\) is said to be \textit{out of} \(A\) and \textit{into} \(B\). A graph \(\mathcal{G}'=(\mathbf{V}',\mathbf{E}')\) is the \textit{induced subgraph} of \(\mathcal{G}=(\mathbf{V},\mathbf{E})\) with respect to \(\mathbf{V}'\) if \(\mathbf{V}'\subseteq\mathbf{V}\) and \(\mathbf{E}'\) includes all edges in \(\mathbf{E}\) that are between nodes in \(\mathbf{V}'\).

\textbf{Paths.} A \textit{path} is a sequence of nodes and edges \((V_0,e_1,V_1,\dots,e_K,V_K)\), \(K\ge 1\), such that every node occurs only once and for \(k=1,\dots,K\), \(e_k\) has endpoints \(V_{k-1}\) and \(V_k\). In a simple graph, the path \((V_0,e_1,V_1,\dots,e_K,V_K)\) can unambiguously be identified by the sequence of nodes \((V_0,V_1,\dots,V_K)\) alone. \(V_0\) and \(V_K\) are called \textit{endpoints of the path} \((V_0,e_1,V_1,\dots,e_K,V_K)\) and the path is said to be \textit{between} \(V_0\) \textit{and} \(V_K\) or \textit{from} \(V_0\) \textit{to} \(V_K\), irrespective of the direction of the edges. For sets of nodes \(\mathbf{A}\) and \(\mathbf{B}\), a path is said to be between \(\mathbf{A}\) and \(\mathbf{B}\) or from \(\mathbf{A}\) to \(\mathbf{B}\) if its first node is in \(\mathbf{A}\) and the last node is in \(\mathbf{B}\). A path from \(\mathbf{A}\) to \(\mathbf{B}\) is \textit{proper} if only its first node \(V_0\) is in \(\mathbf{A}\). Let \(p=(V_0,e_1,V_1,\dots,e_K,V_K)\) and \(k=1,\dots,K\). Then an edge \(V_k\leftarrow V_{k+1}\) on \(p\) is said to point towards \(V_0,\dots,V_k\), while an edge \(V_k\rightarrow V_{k+1}\) on \(p\) is said to point towards \(V_{k+1},\dots,V_K\). A path is \textit{directed} from \(V_0\) to \(V_K\) if all edges in the sequence are directed and point towards \(V_K\). A path \(p\) is \textit{possibly directed} from \(V_0\) to \(V_K\) if all edges on \(p\) are either directed or undirected and there are no \(i,j\), \(1\le i<j\le K\), such that \(V_i\leftarrow V_j\) (cf.\ \cite{PerkovicKalischMaathuis2017}; this definition of a possibly directed path is non-standard as \(V_i\) and \(V_j\) are not necessarily adjacent nodes on the path, which is required for maxPDAGs later). We define the concatenation of two paths \(p=(V_0,e_1,V_1,\dots,e_K,V_K)\) and \(q=(V_K,e_{K+1},V_{K+1},\dots,e_{K+L},V_{K+L})\) as \(p \oplus q=(V_0,e_1,V_1,\dots,e_{K+L},V_{K+L})\), where we require that the nodes \(V_0, \dots, V_{K+L}\) are distinct.

\textbf{Ancestry.} If there is a directed path from \(A\) to \(B\), or if \(A=B\), then \(A\) is an \textit{ancestor} of \(B\) and \(B\) is a \textit{descendant} of \(A\). If there is a possibly directed path from \(A\) to \(B\), or if \(A=B\), then \(A\) is a \textit{possible ancestor} of \(B\) and \(B\) is a \textit{possible descendant} of \(A\). If there is an edge \(A\rightarrow B\), then \(A\) is a \textit{parent} of \(B\) and \(B\) is a \textit{child} of \(A\). If there is an edge \(A-B\), \(A\) and \(B\) are \textit{siblings}. Note that in our terminology, a node is a (possible) ancestor and (possible) descendant of itself, but not a parent/child/sibling of itself. For a node \(V\) in a simple graph \(\mathcal{G}\), we denote the set of all ancestors, possible ancestors, descendants, possible descendants, parents, children and siblings of \(V\) in \(\mathcal{G}\) as \(\mathrm{an}(V, \mathcal{G})\), \(\mathrm{possan}(V, \mathcal{G})\), \(\mathrm{de}(V, \mathcal{G})\), \(\mathrm{possde}(V, \mathcal{G})\), \(\mathrm{pa}(V, \mathcal{G})\), \(\mathrm{ch}(V, \mathcal{G})\), \(\mathrm{sib}(V, \mathcal{G})\), respectively. For a set of nodes \(\mathbf{W}\), the set \(\mathrm{an}(\mathbf{W}, \mathcal{G})\) is defined as \(\bigcup_{W\in\mathbf{W}}\mathrm{an}(W,\mathcal{G})\), with analogous definitions for \(\mathrm{possan}(\mathbf{W}, \mathcal{G})\), \(\mathrm{de}(\mathbf{W}, \mathcal{G})\), \(\mathrm{possde}(\mathbf{W}, \mathcal{G})\), \(\mathrm{pa}(\mathbf{W}, \mathcal{G})\), \(\mathrm{ch}(\mathbf{W}, \mathcal{G})\) and \(\mathrm{sib}(\mathbf{W}, \mathcal{G})\).

\textbf{Colliders, definite-status paths and v-structures.} A non-endpoint node \(V\) is a \textit{collider} on a path \(p\) if both edges adjoining \(V\) on \(p\) have arrowheads at \(V\), i.e.\ \(\rightarrow V\leftarrow\), \(\leftrightarrow V\leftarrow\), \(\rightarrow V\leftrightarrow\), \(\leftrightarrow V\leftrightarrow\). A non-endpoint node \(V\) is a \textit{non-collider} on a path \(p\) if at least one of the edges adjoining \(V\) on \(p\) is out of \(V\), i.e.\ \(\rightarrow V\rightarrow\), \(-V\rightarrow\), \(\leftrightarrow V\rightarrow\), \(\leftarrow V\rightarrow\), \(\leftarrow V\leftrightarrow\), \(\leftarrow V-\), \(\leftarrow V\leftarrow\), or if both edges adjoining \(V\) on \(p\) are undirected edges and the two nodes adjacent to \(V\) on \(p\) are not adjacent to each other. A \textit{definite-status path} is a path on which every non-endpoint is either a collider or a non-collider. In a DAG or an ADMG, all paths are of definite status. Three nodes \(A\), \(B\) and \(C\) form a \textit{v-structure} in a graph \(\mathcal{G}\) if \(A\rightarrow B\leftarrow C\) is the induced subgraph \(\mathcal{G}'\) on \( \{A,  B, C\} \).

\textbf{ADMGs, DAGs and PDAGs.} A directed path from \(A\) to \(B\), together with an edge \(A\leftarrow B\) forms a \textit{directed cycle}. A graph with only directed and bi-directed edges and without directed cycles is called an \textit{acyclic directed mixed graph} (ADMG). A simple graph with only directed edges and without directed cycles is called a \textit{directed acyclic graph} (DAG). A simple graph with only directed and undirected edges containing no directed cycles is called a \textit{partially directed acyclic graph} (PDAG).

\textbf{Blocking and separation.} \citep{Richardson2003,MaathuisColombo2015,Pearl2009} A definite-status path \(p\) in an ADMG or PDAG \(\mathcal{G}\) is \textit{blocked} by a node set \(\mathbf{C}\) if (i) \(p\) contains a non-collider in \(\mathbf{C}\) or (ii) \(p\) contains a collider that is not in \(\mathrm{an}(\mathbf{C},\mathcal{G})\). Otherwise the path \(p\) is \textit{open} given \(\mathbf{C}\). Node sets \(\mathbf{A}\) and \(\mathbf{B}\) are said to be \textit{m-separated} given a set \(\mathbf{C}\) if every path between an \(A\in\mathbf{A}\) and a \(B\in\mathbf{B}\) is blocked by \(\mathbf{C}\). We then write \(\mathbf{A}\perp_\mathcal{G}\mathbf{B}\mid\mathbf{C}\). In DAGs, m-separation is called d-separation.

\textbf{Markov equivalence and CPDAGs.} \citep{AnderssonMadiganPerlman1997} The \textit{(Markov) equivalence class} of a DAG \(\mathcal{D}\) is the set of DAGs that imply the same d-separation relationships as \(\mathcal{D}\). These are all DAGs with the same adjacencies and v-structures \cite{VermaPearl1990}. Markov equivalence classes can be represented as \textit{completed partially directed acyclic graphs} (CPDAGs), which are simple graphs with directed or undirected edges, without directed cycles and with certain restrictions regarding the patterns of edges that can occur. The equivalence class represented by a CPDAG \(\mathcal{G}\) is denoted by \([\mathcal{G}]\). A directed edge \(A\rightarrow B\) in \(\mathcal{G}\) means that this edge is present in all DAGs in the equivalence class \([\mathcal{G}]\). An undirected edge \(A-B\) in \(\mathcal{G}\) means that \(A\) and \(B\) are adjacent in every DAG in \([\mathcal{G}]\) and there is at least one DAG in \([\mathcal{G}]\) with \(A\rightarrow B\) and at least one with \(A\leftarrow B\). 

\textbf{Meek's rules and maxPDAGs.} \citep{PerkovicKalischMaathuis2017} Certain subsets of equivalence classes of DAGs can be represented by maximally oriented PDAGs (maxPDAGs), which are PDAGs with edge orientations that are closed under the orientation rules in Figure~\ref{fig:Meek} (\textit{Meek's rules}, \cite{Meek1995}). The set of DAGs represented by a maxPDAG \(\mathcal{G}\) is denoted by \([\mathcal{G}]\). The edges in maxPDAGs have the same interpretation as in CPDAGs. DAGs and CPDAGs are special cases of maxPDAGs.

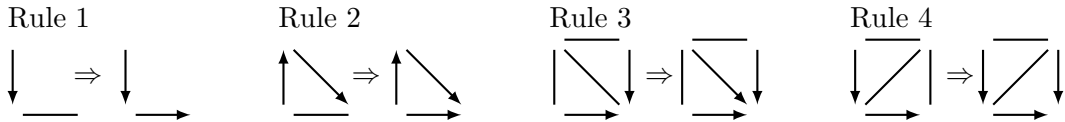
\begin{figure}[h]
	\begin{center}
	\begin{tikzpicture}[node distance=10mm, >=latex]
	\node[] (A1) {};
	\node[above of=A1, xshift=-2mm, yshift=-7mm, anchor=west] (R1) {Rule 1};
	\node[below of=A1] (C1) {};
	\node[right of=C1] (D1) {};
	\node[right of=A1, yshift=-5mm] (Pfeil1) {\(\Rightarrow\)};
	\node[right of=D1, xshift=-5mm] (C1*) {};
	\node[above of=C1*] (A1*) {};
	\node[right of=C1*] (D1*) {};
	\draw[->, thick] (A1) to (C1);
	\draw[-, thick] (C1) to (D1);
	\draw[->, thick] (A1*) to (C1*);
	\draw[->, thick] (C1*) to (D1*);
	
	\node[right of=A1, xshift=26mm] (A2) {};
	\node[above of=A2, xshift=-2mm, yshift=-7mm, anchor=west] (R2) {Rule 2};
	\node[below of=A2] (C2) {};
	\node[right of=C2] (D2) {};
	\node[right of=A2, xshift=1mm, yshift=-5mm] (Pfeil2) {\(\Rightarrow\)};
	\node[right of=D2, xshift=-5mm] (C2*) {};
	\node[above of=C2*] (A2*) {};
	\node[right of=C2*] (D2*) {};
	\draw[->, thick] (C2) to (A2);
	\draw[->, thick] (A2) to (D2);
	\draw[-, thick] (C2) to (D2);
	\draw[->, thick] (C2*) to (A2*);
	\draw[->, thick] (A2*) to (D2*);
	\draw[->, thick] (C2*) to (D2*);
	
	\node[right of=A2, xshift=26mm] (A3) {};
	\node[above of=A3, xshift=-2mm, yshift=-7mm, anchor=west] (R3) {Rule 3};
	\node[right of=A3] (B3) {};
	\node[below of=A3] (C3) {};
	\node[right of=C3] (D3) {};
	\node[right of=A3, xshift=4mm, yshift=-5mm] (Pfeil3) {\(\Rightarrow\)};
	\node[right of=D3, xshift=-3mm] (C3*) {};
	\node[above of=C3*] (A3*) {};
	\node[right of=A3*] (B3*) {};
	\node[right of=C3*] (D3*) {};
	\draw[-, thick] (A3) to (B3);
	\draw[-, thick] (A3) to (C3);
	\draw[-, thick] (A3) to (D3);
	\draw[->, thick] (B3) to (D3);
	\draw[->, thick] (C3) to (D3);
	\draw[-, thick] (A3*) to (B3*);
	\draw[-, thick] (A3*) to (C3*);
	\draw[->, thick] (A3*) to (D3*);
	\draw[->, thick] (B3*) to (D3*);
	\draw[->, thick] (C3*) to (D3*);
	
	\node[right of=A3, xshift=30mm] (A4) {};
	\node[above of=A4, xshift=-2mm, yshift=-7mm, anchor=west] (R4) {Rule 4};
	\node[right of=A4] (B4) {};
	\node[below of=A4] (C4) {};
	\node[right of=C4] (D4) {};
	\node[right of=A4, xshift=4mm, yshift=-5mm] (Pfeil4) {\(\Rightarrow\)};
	\node[right of=D4, xshift=-3mm] (C4*) {};
	\node[above of=C4*] (A4*) {};
	\node[right of=A4*] (B4*) {};
	\node[right of=C4*] (D4*) {};
	\draw[-, thick] (A4) to (B4);
	\draw[->, thick] (A4) to (C4);
	\draw[-, thick] (B4) to (C4);
	\draw[-, thick] (B4) to (D4);
	\draw[->, thick] (C4) to (D4);
	\draw[-, thick] (A4*) to (B4*);
	\draw[->, thick] (A4*) to (C4*);
	\draw[-, thick] (B4*) to (C4*);
	\draw[->, thick] (B4*) to (D4*);
	\draw[->, thick] (C4*) to (D4*);
	
	\end{tikzpicture}
	\end{center}
	\caption{Meek's orientation rules. Let \(\mathcal{G}\) be a simple graph with only directed and undirected edges and without directed cycles. If the graph on the left is an induced subgraph of \(\mathcal{G}\), then orient the undirected edges in \(\mathcal{G}\) according to the graph on the right \citep{Meek1995}. The rules prevent directed cycles and new v-structures.}
	\label{fig:Meek}
\end{figure}

\textbf{Partial topological ordering.} Let \(\mathcal{D}\) be a DAG with node set \(\mathbf{V}\) and let \(\mathbf{V}_1,\dots,\mathbf{V}_p\) be a partition of \(\mathbf{V}\). Then \(\mathbf{V}_1<\dots <\mathbf{V}_p\) is a \textit{partial topological ordering} of \(\mathbf{V}\) if for every \(i>j\), there are no directed edges from \(\mathbf{V}_i\) to \(\mathbf{V}_j\).

\textbf{Independence and faithfulness.} For sets of random variables \(\mathbf{X}\), \(\mathbf{Y}\) and \(\mathbf{Z}\), if \(\mathbf{X}\) and \(\mathbf{Y}\) are conditionally independent given \(\mathbf{Z}\), we write \(\mathbf{X}\ind\mathbf{Y}\mid\mathbf{Z}\). A joint density \(f(\mathbf{v})\) over a set of random variables \(\mathbf{V}\) 
is \textit{Markov} with respect to a DAG \(\mathcal{D}\) with node set \(\mathbf{V}\) if for disjoint \(\mathbf{X},\mathbf{Y},\mathbf{Z}\subseteq\mathbf{V}\), \(\mathbf{X}\perp_\mathcal{D}\mathbf{Y}\mid\mathbf{Z} \Rightarrow\mathbf{X}\ind\mathbf{Y}\mid\mathbf{Z}\); the density \(f(\mathbf{v})\)  is \textit{faithful} to \(\mathcal{D}\) if also \(\mathbf{X}\ind\mathbf{Y}\mid\mathbf{Z}\Rightarrow\mathbf{X}\perp_\mathcal{D}\mathbf{Y}\mid\mathbf{Z}\).

\textbf{Causal DAGs, CPDAGs, maxPDAGs and ADMGs.} Intuitively, a \textit{causal DAG} is a DAG where  an edge \(A\rightarrow B\) means that \(A\) is a direct cause of \(B\) (relative to the variables included). This can be formalised using the intervention operator, denoted by \(do(\cdot)\) in \cite{Pearl2009}. For random variables \(\mathbf{V}\) and \(\mathbf{X}\subseteq\mathbf{V}\), the \textit{post-intervention density} \(f(\mathbf{v}\mid do(\mathbf{x}'))\) is the joint density of \(\mathbf{V}\) in a (hypothetical) experiment that fixes \(\mathbf{X}\) to \(\mathbf{x}'\) for everyone in the population by an external intervention. A joint density \(f(\mathbf{v})\) is \textit{compatible with a causal DAG} \(\mathcal{D}=(\mathbf{V},\mathbf{E})\) if for all \(\mathbf{X}\subseteq\mathbf{V}\), the post-intervention density  \(f(\mathbf{v}\mid do(\mathbf{x}'))\) can be written as 
\[
f(\mathbf{v}\mid do(\mathbf{x}'))=\mathbf{1}(\mathbf{x}=\mathbf{x}')
\prod_{V\in\mathbf{V}\setminus\mathbf{X}}f(v\mid \mathrm{pa}(V,\mathcal{D})),
\]
where \(\mathbf{1}(\mathbf{x}=\mathbf{x}')\) is the indicator function that is 1 if \(\mathbf{x}=\mathbf{x'}\) and 0 otherwise. This is known as the truncated factorisation formula \citep{Spirtesetal2000, Pearl2009}. A CPDAG or maxPDAG \(\mathcal{G}\) is called a \textit{causal CPDAG} or \textit{ causal maxPDAG} if \([\mathcal{G}]\) contains a causal DAG. A \textit{causal ADMG} is an ADMG that has been obtained by subjecting a causal DAG to a latent projection, see Definition \ref{lproj}.

\textbf{(Possibly) causal nodes and forbidden nodes.} See Section~\ref{sec:review}.

\textbf{Valid adjustment sets and amenability.} Let \(\mathbf{X}\), \(\mathbf{Y}\) and \(\mathbf{Z}\) be disjoint sets of random variables, where \(\mathbf{Z}\) is possibly empty. Then \(\mathbf{Z}\) is a \textit{valid adjustment set} relative to \((\mathbf{X},\mathbf{Y})\) if we have 
\begin{equation}\label{eq:VAS}f(\mathbf{y}\mid do(\mathbf{x}))=\begin{cases}
f(\mathbf{y}\mid\mathbf{x})&\textit{if }\mathbf{Z}=\emptyset,\\
\int_{\mathbf{z}}f(\mathbf{y}\mid\mathbf{x},\mathbf{z})f(\mathbf{z})d\mathbf{z}&\text{otherwise.}
\end{cases}\end{equation}
Relative to a causal DAG, CPDAG, maxPDAG or ADMG \(\mathcal{G}=(\mathbf{V},\mathbf{E})\), a valid adjustment set is defined as follows: Let \(\mathbf{X}\), \(\mathbf{Y}\) and \(\mathbf{Z}\) be disjoint subsets of \(\mathbf{V}\), where \(\mathbf{Z}\) is possibly empty. Then \(\mathbf{Z}\) is a \textit{valid adjustment set} relative to \((\mathbf{X},\mathbf{Y})\) in \(\mathcal{G}\) if equation (\ref{eq:VAS}) holds for every joint density \(f(\mathbf{v})\) compatible with \(\mathcal{G}\) \citep{Perkovicetal2018}. Further, \(\mathcal{G}\) is said to be \textit{amenable} for adjustment relative to \((\mathbf{X},\mathbf{Y})\) if every proper possibly causal path from \(\mathbf{X}\) to \(\mathbf{Y}\) starts with a directed edge out of \(\mathbf{X}\) (\citealp{Perkovicetal2018}).

\textbf{Generalised adjustment criterion.} \citep{PerkovicKalischMaathuis2017,Perkovicetal2018,ShpitserVanderWeeleRobins2010} Let \(\mathbf{X}\), \(\mathbf{Y}\) and \(\mathbf{Z}\) be disjoint node sets in a causal DAG, CPDAG, maxPDAG or ADMG \(\mathcal{G}\). Then \(\mathbf{Z}\) is a valid adjustment set relative to \((\mathbf{X},\mathbf{Y})\) in \(\mathcal{G}\) if and only if the following three conditions hold:
\begin{enumerate}
	\item[(a)] \(\mathcal{G}\) is amenable relative to \((\mathbf{X},\mathbf{Y})\),
	\item[(b)] \(\mathbf{Z}\cap\mathrm{forb}(\mathbf{X},\mathbf{Y},\mathcal{G})=\emptyset\),
	\item[(c)] all proper non-causal definite-status paths from \(\mathbf{X}\) to \(\mathbf{Y}\) are blocked by \(\mathbf{Z}\).
\end{enumerate}

\textbf{Causal linear model.} Let \(\mathcal{D}\) be a causal DAG with node set \(\mathbf{V}=(V_1,\cdots,V_p)\). Then \(\mathbf{V}\) is said to follow a \textit{causal linear model} compatible with \(\mathcal{D}\) if the distribution of each \(V_i\in\mathbf{V}\) can be described by an equation of the form\[V_i=\sum_{V_j\in\mathrm{pa}(V_i,\mathcal{D})}\alpha_{ij}V_j+\epsilon_{v_i}\]
with \(\alpha_{ij}\in\mathbb{R}\) and \(\epsilon_{v_i}\) a random variable with mean 0 and finite variance such that \(\epsilon_{v_1},\dots,\epsilon_{v_p}\) are jointly independent. For a causal CPDAG or maxPDAG \(\mathcal{G}\), \(\mathbf{V}\) is said to follow a causal linear model compatible with \(\mathcal{G}\) if \(\mathbf{V}\) follows a causal linear model compatible with a DAG in \([\mathcal{G}]\).

\textbf{Partial variance notation.} Consider a random variable \(S\) and a random vector \(\mathbf{T}\). We denote the covariance matrix of \(\mathbf{T}\) by \(\Sigma_{\mathbf{tt}}\) and the row vector of covariances between \(S\) and \(\mathbf{T}\) by \(\Sigma_{s\mathbf{t}}\). The partial variance of \(S\) given \(\mathbf{T}\) is defined as \(\sigma_{ss.\mathbf{t}}=\mathrm{Var}(S)-\Sigma_{s\mathbf{t}}\Sigma_{\mathbf{tt}}^{-1}\Sigma_{s\mathbf{t}}^{-1}\).

\textbf{Asymptotic variance.} Consider a sequence of estimators \((\hat{\beta}_n)_{n\in\mathbb{N}}\) such that \(\sqrt{n}(\hat{\beta}_n-\beta)\) converges in distribution to \(\mathcal{N}(0,v)\). We call \(v\) the asymptotic variance of \(\hat{\beta}\) and write \(a.var(\hat{\beta})=v\).

\section[Appendix B: Proofs for Section 3]{Proofs for Section 3}
\label{appB}
In this appendix, we prove our claims about the forbidden projection made in Section~\ref{sec:new_def}.

\setcounter{theorem}{5}
\begin{proposition}
	Let \(\mathbf{X}\) and \(\mathbf{Y}\) be disjoint node sets in a causal DAG \(\mathcal{D}\) such that \(\mathbf{Y}\subseteq\mathrm{de}(\mathbf{X},\mathcal{D})\). Then a valid adjustment set relative to \((\mathbf{X}\),\(\mathbf{Y})\) in \(\mathcal{D}\) exists if and only if there is no bi-directed edge between any \(X\in\mathbf{X}\) and \(Y\in\mathbf{Y}\) in \(\mathcal{D}^{\mathbf{XY}}\).
\end{proposition}
\begin{proof}
	We show that a valid adjustment set relative to \((\mathbf{X}\),\(\mathbf{Y})\) in \(\mathcal{D}\) \textit{cannot} exist if and only if there is a bi-directed edge between a \(X\in\mathbf{X}\) and a \(Y\in\mathbf{Y}\) in the forbidden projection \(\mathcal{D}^{\mathbf{XY}}\).
	
	Assume first that there is a bi-directed edge in \(\mathcal{D}^{\mathbf{XY}}\) between some \(X\in\mathbf{X}\) and a \(Y\in\mathbf{Y}\). Then according to Definition \ref{lproj} of the latent projection there is a path in \(\mathcal{D}\) between \(X\) and \(Y\) on which all nodes are non-colliders and contained in the forbidden set. This constitutes a non-causal path that cannot be blocked by any sets of nodes that are not forbidden. Hence no valid adjustment set relative to (\(\mathbf{X}\),\(\mathbf{Y}\)) and \(\mathcal{D}\) exists.
	
	Assume now that there is no valid adjustment set relative to (\(\mathbf{X}\),\(\mathbf{Y}\)) in \(\mathcal{D}\). Then Lemma \ref{EmaCorollary} implies \(\mathbf{X}\cap\mathrm{de}(\mathrm{cn}(\mathbf{X},\mathbf{Y},\mathcal{D}),\mathcal{D})\ne\emptyset\). Let \(X^*\in \mathbf{X}\cap\mathrm{de}(\mathrm{cn}(\mathbf{X},\mathbf{Y},\mathcal{D}),\mathcal{D})\). Then there must exist a node \(C^*\in\mathrm{cn}(\mathbf{X},\mathbf{Y},\mathcal{D})\) and a node \(Y^*\in\mathbf{Y}\) such that there is a path of the form \( X^*\leftarrow\dots\leftarrow C^*\rightarrow\dots\rightarrow Y^*\) where	all non-endpoints are non-colliders on the path and in the forbidden set. It follows from Definition \ref{lproj} of the latent projection that \(\mathcal{D}^{\mathbf{XY}}\) contains a bi-directed edge \(X^*\leftrightarrow Y^*\).
\end{proof}

\setcounter{theorem}{14}
\begin{lemma}[Corollary 27 in \citealp{Perkovicetal2018}]
	\label{EmaCorollary}
	Let \(\mathbf{X}\) and \(\mathbf{Y}\) be disjoint node sets in a causal DAG \(\mathcal{D}\) such that \(\mathbf{Y}\subseteq\mathrm{de}(\mathbf{X},\mathcal{D})\). Then a valid adjustment set relative to \((\mathbf{X}\),\(\mathbf{Y})\) in \(\mathcal{D}\) exists if and only if \(\mathbf{X}\cap\mathrm{de}(\mathrm{cn}(\mathbf{X},\mathbf{Y},\mathcal{D}),\mathcal{D})=\emptyset\).
\end{lemma}

\setcounter{theorem}{6}
\begin{proposition}
	Let \(\mathbf{X}\) and \(\{Y\}\) be disjoint node sets in a causal DAG \(\mathcal{D}\) such that \(Y\in\mathrm{de}(\mathbf{X},\mathcal{D})\). Then \(\mathcal{D}^{\mathbf{X}Y}\) is a causal DAG if and only if there exists a valid adjustment set relative to \((\mathbf{X},Y)\) in \(\mathcal{D}\).
\end{proposition}
\begin{proof}
	First assume that \(\mathcal{D}^{\mathbf{X}Y}\) is a causal DAG. Then by Proposition~\ref{nobiDAG}, a valid adjustment set relative to \((\mathbf{X},Y)\) in \(\mathcal{D}\) exists.
	
	Now assume that a valid adjustment set relative to \((\mathbf{X},Y)\) in \(\mathcal{D}\) exists. We show that (1) \(\mathcal{D}^{\mathbf{X}Y}\) is a DAG syntactically, i.e.\ a directed graph without cycles, (2) semantically, applying the \textit{d}-separation criterion to sets of nodes in \(\mathcal{D}^{\mathbf{X}Y}\) yields the same separations as applying the \textit{d}-separation criterion to the same node sets in \(\mathcal{D}\), and (3) \(\mathcal{D}^{\mathbf{X}Y}\) is a causal DAG for \(\left(\mathbf{V}\setminus\mathrm{forb}(\mathbf{X},Y,\mathcal{D})\right)\cup\mathbf{X}\cup\{Y\}\).
	
	\textit{(1)} As we assume that a valid adjustment set relative to \((\mathbf{X},Y)\) in \(\mathcal{D}\) exists, it follows from Proposition \ref{nobiDAG} together with Lemma \ref{intoY} that \(\mathcal{D}^{\mathbf{X}Y}\) does not contain bi-directed edges. Acyclicity of latent projections is guaranteed by property 1 of Definition \ref{lproj} of the latent projection: every directed edge in \(\mathcal{D}(\mathbf{W})\) corresponds to a directed path in \(\mathcal{D}\), hence if \(\mathcal{D}^{\mathbf{X}Y}\) had a directed cycle then so would \(\mathcal{D}\). It follows that \(\mathcal{D}^{\mathbf{X}Y}\) is acyclic.	
	
	\textit{(2)} The m-separations in a latent projection \(\mathcal{D}(\mathbf{W})\) correspond to the d-separations between nodes in \(\mathbf{W}\) in the original DAG \(\mathcal{D}\) \citep{Richardsonetal2017}. In our case, \(\mathcal{D}(\mathbf{W})=\mathcal{D}^{\mathbf{X}Y}\) is itself a DAG syntactically, and for DAGs m-separation and d-separation are equivalent.
	
	\textit{(3)} Since \(\mathcal{D}\) is a causal DAG for the random variables \(\mathbf{V}\), the truncated factorisation derived from \(\mathcal{D}\) holds for all interventions \(do(\mathbf{T}=\mathbf{t}')\) with \(\mathbf{T}\subseteq\mathbf{V}\):
	\begin{equation}
	\label{eq:trunc}
	f(\mathbf{v}\mid do(\mathbf{t}'))=\mathbf{1}(\mathbf{t}=\mathbf{t}')\prod_{V\in\mathbf{V}\setminus\mathbf{T}}f(v\mid\mathrm{pa}(V,\mathcal{D})).
	\end{equation}
	We need to show that the truncated factorisation implied by \(\mathcal{D}^{\mathbf{X}Y}\) holds for the joint marginal distribution of \((\mathbf{V}\setminus\mathrm{forb}(\mathbf{X},Y,\mathcal{D}))\cup\mathbf{X}\cup\{Y\}\). We distinguish two cases. In the first case, \(Y\not\in\mathrm{de}(\mathbf{X}, \mathcal{D})\). This case is trivial, as the forbidden set is then empty  and \(\mathcal{D}=\mathcal{D}^{\mathbf{X}Y}\). For the second case, \(Y\in\mathrm{de}(\mathbf{X}, \mathcal{D})\) we define the following sets: \(\mathbf{A}=\left(\mathbf{V}\setminus\mathrm{forb}(\mathbf{X},Y,\mathcal{D})\right)\cup\mathbf{X}\) is the node set of \(\mathcal{D}^{\mathbf{X}Y}\) without \(Y\), \(\mathbf{C}=\mathrm{cn}(\mathbf{X},Y,\mathcal{D})\setminus(\mathbf{X}\cup\{Y\})\) is the set of forbidden nodes that are ancestors of \(Y\), excluding \(\mathbf{X}\) and \(Y\), and \(\mathbf{M}=\mathrm{forb}(\mathbf{X},Y,\mathcal{D})\setminus(\mathbf{X}\cup\{Y\}\cup\mathbf{C})\) is the forbidden set excluding \(\mathbf{X}\), \(Y\) and \(\mathbf{C}\), so that \(\mathbf{C}\cup\mathbf{M}=\mathrm{forb}(\mathbf{X},Y,\mathcal{D})\setminus(\mathbf{X}\cup\{Y\})\) is the set over which we marginalise when subjecting \(\mathcal{D}\) to the forbidden projection. Then the following partial topological ordering holds: \(\mathbf{A}<\mathbf{C}<Y<\mathbf{M}\). (Note that \(Y\) cannot have descendants in \(\mathbf{X}\), as otherwise no valid adjustment set would exist by Lemma \ref{EmaCorollary}.)
	
	We can now rewrite equation (\ref{eq:trunc}) as
	\begin{multline*}
	f(\mathbf{v}\mid do(\mathbf{t}'))=\mathbf{1}(\mathbf{t}=\mathbf{t}')\prod_{A\in\mathbf{A}\setminus\mathbf{T}}f(a\mid\mathrm{pa}(A,\mathcal{D}))\prod_{C\in\mathbf{C}\setminus\mathbf{T}}f(c\mid\mathrm{pa}(C,\mathcal{D}))\\f(y\mid\mathrm{pa}(Y,\mathcal{D}))^{\mathbf{1}(Y\not\in\mathbf{T})}\prod_{M\in\mathbf{M}\setminus\mathbf{T}}f(m\mid\mathrm{pa}(M,\mathcal{D})).\end{multline*}
	
	Consider now interventions only in nodes \(\mathbf{T}\subseteq\left(\mathbf{V}\setminus\mathrm{forb}(\mathbf{X},Y,\mathcal{D})\right)\cup\mathbf{X}\cup\{Y\}\), then \(\mathbf{C}\setminus\mathbf{T}=\mathbf{C}\) and \(\mathbf{M}\setminus\mathbf{T}=\mathbf{M}\). Upon marginalising the above intervention distribution over \(\mathbf{M}\) the last term in the product vanishes but the remaining terms do not change:
	\begin{multline*}
	f(\mathbf{a},y,\mathbf{c}\mid do(\mathbf{t}'))=\mathbf{1}(\mathbf{t}=\mathbf{t}')\prod_{A\in\mathbf{A}\setminus\mathbf{T}}f(a\mid\mathrm{pa}(A,\mathcal{D}))\prod_{C\in\mathbf{C}}f(c\mid\mathrm{pa}(C,\mathcal{D}))\\f(y\mid\mathrm{pa}(Y,\mathcal{D}))^{\mathbf{1}(Y\not\in\mathbf{T})}.
	\end{multline*}
	Further marginalising over \(\mathbf{C}\), the partial topological order guarantees that the variables in \(\mathbf{A}\) do not have parents in \(\mathbf{C}\).  
	This yields
	\begin{multline*}
	f(\mathbf{a},y\mid do(\mathbf{t}'))=\mathbf{1}(\mathbf{t}=\mathbf{t}')\prod_{A\in\mathbf{A}\setminus\mathbf{T}}f(a\mid\mathrm{pa}(A,\mathcal{D}))\int_{\mathbf{c}}\prod_{C\in\mathbf{C}}f(c\mid\mathrm{pa}(C,\mathcal{D}))\\f(y\mid\mathrm{pa}(Y,\mathcal{D}))^{\mathbf{1}(Y\not\in\mathbf{T})}d\mathbf{c}.
	\end{multline*}
	A variable is conditionally independent of its non-descendants given its parents. All variables in \(\mathbf{A}\cup\mathbf{C}\) are non-descendants of \(Y\), hence \(Y\ind\mathbf{A}\cup\mathbf{C}\mid\mathrm{pa}(Y,\mathcal{D})\) and \(f(y\mid\mathrm{pa}(Y,\mathcal{D}))=f(y\mid\mathrm{pa}(Y,\mathcal{D})\cup\mathbf{a}\cup\mathbf{c})=f(y\mid\mathbf{a}\cup\mathbf{c})\). The second equality holds because the parents of \(Y\), if there are any, form a subset of \(\mathbf{A}\cup\mathbf{C}\). 
	Similarly, all variables in \(\mathbf{A}\) are non-descendants of all variables in \(\mathbf{C}\), hence \(f(c\mid\mathrm{pa}(C,\mathcal{D}))=f(c\mid\mathrm{pa}(C,\mathcal{D})\cup\mathbf{a})\). Further, all parents of variables in \(\mathbf{C}\) are in \(\mathbf{A}\cup\mathbf{C}\), hence \(\prod_{C\in\mathbf{C}}f(c\mid\mathrm{pa}(C,\mathcal{D})\cup\mathbf{a})=f(\mathbf{c}\mid\mathbf{a})\). We obtain
	\begin{align*}
	f(\mathbf{a},y\mid do(\mathbf{t}'))&=\mathbf{1}(\mathbf{t}=\mathbf{t'})\prod_{A\in\mathbf{A}\setminus\mathbf{T}}f(a\mid\mathrm{pa}(A,\mathcal{D}))\int_{\mathbf{c}}f(\mathbf{c}\mid\mathbf{a})f(y\mid\mathbf{a}\cup\mathbf{c})^{\mathbf{1}(Y\not\in\mathbf{T})}d\mathbf{c}\\
	&=\mathbf{1}(\mathbf{t}=\mathbf{t}')\prod_{A\in\mathbf{A}\setminus\mathbf{T}}f(a\mid\mathrm{pa}(A,\mathcal{D}))\int_{\mathbf{c}}f(\mathbf{c},y\mid\mathbf{a})^{\mathbf{1}(Y\not\in\mathbf{T})}f(\mathbf{c}\mid\mathbf{a})^{\mathbf{1}(Y\in\mathbf{T})}d\mathbf{c}\\
	&=\mathbf{1}(\mathbf{t}=\mathbf{t'})\prod_{A\in\mathbf{A}\setminus\mathbf{T}}f(a\mid\mathrm{pa}(A,\mathcal{D}))f(y\mid\mathbf{a})^{\mathbf{1}(Y\not\in\mathbf{T})}.
	\end{align*}
	Two things remain to be shown. First, for every \(A\in\mathbf{A}\), \(\mathrm{pa}(A,\mathcal{D})=\mathrm{pa}(A,\mathcal{D}^{\mathbf{X}Y})\) because \(A\) does not have parents in the node set over which we marginalised in the projection. Second, \(f(y\mid\mathbf{a})=f(y\mid\mathrm{pa}(Y,\mathcal{D}^{\mathbf{X}Y}))\), which follows from the fact that all conditional independencies between variables in \(\mathcal{D}^{\mathbf{X}Y}\) can be read off \(\mathcal{D}^{\mathbf{X}Y}\) using the \textit{d}-separation criterion, as we showed in part (2) of this proof. Hence, we have
	\begin{align*}
	f(\mathbf{a},y\mid do(\mathbf{t}'))&=\mathbf{1}(\mathbf{t}=\mathbf{t}')\prod_{A\in\mathbf{A}\setminus\mathbf{T}}f(a\mid\mathrm{pa}(A,\mathcal{D}^{\mathbf{X}Y}))f(y\mid\mathrm{pa}(Y,\mathcal{D}^{\mathbf{X}Y}))^{\mathbf{1}(Y\not\in\mathbf{T})}\\
	&=\mathbf{1}(\mathbf{t}=\mathbf{t}')\prod_{V\in (\mathbf{A}\cup\{Y\}) \setminus\mathbf{T}}f(v\mid\mathrm{pa}(V,\mathcal{D}^{\mathbf{X}Y})),
	\end{align*}
	which is exactly the truncated factorisation formula implied by \(\mathcal{D}^{\mathbf{X}Y}\). Hence, \(\mathcal{D}^{\mathbf{X}Y}\) is a causal DAG for the random variables \(\mathbf{A}\cup\{Y\}=(\mathbf{V}\setminus\mathrm{forb}(\mathbf{X},Y,\mathcal{D}))\cup\mathbf{X}\cup\{Y\}\).
\end{proof}

\setcounter{theorem}{15}
\begin{lemma}
	\label{intoY}
	Let \(\mathcal{D}\) be a DAG with node set \(\mathbf{V}\) and let \(\mathbf{X}\subset\mathbf{V}\) and \(Y\in\mathbf{V}\setminus\mathbf{X}\) such that a valid adjustment set relative to \((\mathbf{X},Y)\) in \(\mathcal{D}\) exists. Then any edge that is in \(\mathcal{D}^{\mathbf{X}Y}\) but not in \(\mathcal{D}\) is a directed edge into \(Y\).
\end{lemma}
\begin{proof}
	We only consider the case where \(Y\in\mathrm{de}(\mathbf{X},\mathcal{D})\), as otherwise \(\mathcal{D}=\mathcal{D}^{\mathbf{X}Y}\) and our statement follows trivially. Define \(\mathbf{F}=\mathrm{forb}(\mathbf{X},Y,\mathcal{D})\setminus(\mathbf{X}\cup\{Y\})\).
	
	By Definition \ref{fproj}, an edge present in \(\mathcal{D}^{\mathbf{X}Y}\) but not in \(\mathcal{D}\) only occurs if \(\mathcal{D}\) contains a node \(W_j\in\mathbf{V}\setminus\mathbf{F}\) that has an ancestor in \(\mathbf{F}\). We show that the only node in \(\mathbf{V}\setminus\mathbf{F}\) that can have an ancestor in \(\mathbf{F}\) is \(Y\).
	
	Consider first a \(W\in\mathbf{V}\setminus\mathrm{forb}(\mathbf{X},Y,\mathcal{D})\). \(W\) does not have ancestors in \(\mathbf{F}\), as otherwise \(W\) would be a forbidden node itself. Consider next a node \(X\in\mathbf{X}\). \(X\) does not have ancestors in \(\mathbf{F}\) either, as every node in \(\mathbf{F}\) is a descendant of \(\mathrm{cn}(\mathbf{X},\mathbf{Y},\mathcal{D})\), but we assume that a valid adjustment set exists relative to \((\mathbf{X},Y)\) in \(\mathcal{D}\), implying \(\mathbf{X}\cap\mathrm{de}(\mathrm{cn}(\mathbf{X},\mathbf{Y},\mathcal{D}),\mathcal{D})=\emptyset\) by Lemma \ref{EmaCorollary}. Hence, \(Y\) is the only node in \(\mathbf{V}\setminus\mathbf{F}\) that can have an ancestor in \(\mathbf{F}\).
\end{proof}

\setcounter{theorem}{7}
\begin{proposition}
	Let \(\mathbf{X}, \mathbf{Y}\) and \(\mathbf{Z}\) be disjoint node sets in a causal DAG \(\mathcal{D}\). Then  $\mathbf{Z}$ is a valid adjustment set relative to \((\mathbf{X},\mathbf{Y})\) in $\mathcal{D}$ if and only if $\mathbf{Z}$ is also a valid adjustment set relative to \((\mathbf{X},\mathbf{Y})\) in $\mathcal{D}^{\mathbf{XY}}$.
\end{proposition}
\begin{proof}
	Throughout, let \(\mathbf{F}=\mathrm{forb}(\mathbf{X},\mathbf{Y},\mathcal{D})\setminus(\mathbf{X}\cup\mathbf{Y})\).
	
	We first suppose that $\mathbf{Z}$ is a valid adjustment set in $\mathcal{D}$ and show that this implies that it is also a valid adjustment set in $\mathcal{D}^{\mathbf{XY}}$.
	Hence, \(\mathbf{Z} \cap \fb{\mathcal{D}} = \emptyset\) and \(\mathbf{Z} \cap \mathbf{Y} = \emptyset\), so that every node in \(\mathbf{Z}\) is also a node of	 \(\mathcal{D}^{\mathbf{XY}}\). Further, \(\fb{\mathcal{D}^{\mathbf{XY}}}\subseteq \mathbf{X} \cup \mathbf{Y}\) and hence \(\mathbf{Z} \cap \fb{\mathcal{D}^{\mathbf{XY}}}=\emptyset\).
	Amenability trivially holds in both $\mathcal{D}$ and $\mathcal{D}^{\mathbf{XY}}$ by assumption.
	
	It  remains to show that every proper non-causal path from $\mathbf{X}$ to $\mathbf{Y}$ in $\mathcal{D}^{\mathbf{XY}}$ is blocked by $\mathbf{Z}$, which we do by contradiction. So suppose that $\mathbf{Z}$ is a valid adjustment set relative to $(\mathbf{X},\mathbf{Y})$ in $\mathcal{D}$ and that there exists a proper non-causal path $p=(V_0,e_1,V_1,\dots,e_K,V_K)$ from $\mathbf{X}$ to $\mathbf{Y}$ in $\mathcal{D}^{\mathbf{XY}}$ that is open given $\mathbf{Z}$. We denote $\mathbf{Y}^{F} = \mathbf{Y} \cap \fb{\mathcal{D}^{\mathbf{XY}}}= \fb{\mathcal{D}^{\mathbf{XY}}} \setminus \mathbf{X}$ and note that $\de(\mathbf{Y}^{F},\mathcal{D}^{\mathbf{XY}}) \subseteq \mathbf{Y}^{F}$.
	
	Let $V_L \in \mathbf{Y}$ be the first node on $p$ that is in \(\mathbf{Y}\) and consider the path segment $p'=(V_0,e_1,V_1,\dots,e_L,V_L)$. Suppose that $L<K$ and that $V_L \in \mathbf{Y^F}$. If $p'$ is causal, then $p$ must either be causal or contain a collider in $\mathbf{Y}^F$, contradicting our assumption that it is open given $\mathbf{Z}$. If $V_L \in \mathbf{Y} \setminus \mathbf{Y}^F$ then $p'$ cannot be causal. Hence, we can suppose that $L=K$ or replace $p$ with $p'$ without loss of generality.
	
	Consider now the case that $V_K \in \mathbf{Y} \setminus \mathbf{Y}^F$. This implies that all nodes on $p$ except $V_0$ are not in $\fb{\mathcal{D}^{\mathbf{XY}}}$. Since $\de(\mathrm{forb}(\mathbf{X},\mathbf{Y},\mathcal{G})\setminus\mathbf{X},\mathcal{D}) \subseteq \fb{\mathcal{D}}$ and by definition of latent projections, this implies that $p$ is also a path in $\mathcal{D}$. As for any node $V$ in $\mathcal{D}^{\mathbf{XY}}$, $\de(V,\mathcal{D}) \subseteq \de(V,\mathcal{D}^{\mathbf{XY}})\cup\mathbf{F}$, and $\mathbf{Z} \cap \mathbf{F}=\emptyset$, it follows that $p$ is also open given $\mathbf{Z}$ in $\mathcal{D}$.
	
	Consider now the case that  $V_K \in \mathbf{Y}^F$. The path $p$ cannot be a one-edge path, as the two possible such paths would require the existence of paths in $\mathcal{D}$ implying that no valid adjustment sets relative to $(\mathbf{X},\mathbf{Y})$ exist in $\mathcal{D}$. By the fact that $\de(\mathbf{Y}^{F},\mathcal{D}) \subseteq \mathbf{Y}^{F}$, the last edge of $p$ must be of the form $p''=V_{K-1}\rightarrow V_K$. By the same argument and the definition of the forbidden projection, the segment $p'=(V_0,e_1,V_1,\dots,V_{K-1})$ is also a path in $\mathcal{D}$, which by definition of $p$ and $p''$ must be non-causal. The path $p''$ corresponds to a causal path $q''$ in $\mathcal{D}$, such that all nodes except for $V_{K-1}$ on $q''$ are forbidden. 
	
	The path $q=p' \oplus q''$ is a proper non-causal path from $\mathbf{X}$ to $\mathbf{Y}$ in $\mathcal{D}$;  we now show that it is open given $\mathbf{Z}$. Since $\de(V,\mathcal{D}) \subseteq \de(V,\mathcal{D}^{\mathbf{XY}}) \cup \mathbf{F} $, for any node $V \notin \mathbf{F}$ in $\mathcal{D}$, it follows from the fact that $p'$ is open given $\mathbf{Z}$ in $\mathcal{D}^{\mathbf{XY}}$ that it is also open given $\mathbf{Z}$ in $\mathcal{D}$. Since $\mathbf{F} \cap \mathbf{Z} =\emptyset$, $q''$ is also open given $\mathbf{Z}$. The node $V_{K-1} $ is a non-collider on $p$ and hence by the assumption that $p$ is open given $\mathbf{Z}$ it follows that $V_K \notin \mathbf{Z}$. Since $V_{K-1}$ is also a non-collider on $q$ it follows that $q$ is open given $\mathbf{Z}$ in $\mathcal{D}$.	
	
	We now turn to the second part of the proof showing that if a set $\mathbf{Z}$ containing no nodes in $\mathbf{F}$ is not a valid adjustment relative to $(\mathbf{X},\mathbf{Y})$ in $\mathcal{D}$ then it is also not a valid adjustment set relative to $(\mathbf{X},\mathbf{Y})$ in $\mathcal{D}^{\mathbf{XY}}$.
	
	Suppose that $\mathbf{Z} \cap \fb{\mathcal{D}} \neq \emptyset$. Since  $\mathbf{Z} \cap \mathbf{F} = \emptyset$ it follows that $\mathbf{Z} \cap (\mathbf{X} \cup \mathbf{Y})\neq \emptyset$; but this clearly implies that  $\mathbf{Z}$ cannot be a valid adjustment set in $\mathcal{D}^{\mathbf{XY}}$.
	
	Suppose now that $\mathbf{Z} \cap (\fb{\mathcal{D}}\cup\mathbf{Y}) = \emptyset$ and that $\mathbf{Z}$ is not a valid adjustment set relative to $(\mathbf{X},\mathbf{Y})$ in $\mathcal{D}$. This implies the existence of a proper non-causal path $p=(V_0,e_1,V_1,\dots,e_K,V_K)$ from $\mathbf{X}$ to $\mathbf{Y}$ in $\mathcal{D}$ that is open given $\mathbf{Z}$. 
	
	Consider first the case that $p$ contains no nodes in $\fb{\mathcal{D}}$. Then $p$ also exists in $\mathcal{D}^{\mathbf{XY}}$ and by the fact that $\de(V,\mathcal{D}^{\mathbf{XY}}) = \de(V,\mathcal{D}) \setminus \mathbf{F}$ it follows that $p$ is also open given $\mathbf{Z}$ in $\mathcal{D}^{\mathbf{XY}}$. Suppose now that $p$ contains at least one node in $\fb{\mathcal{D}}$. Since $p$ is open given $\mathbf{Z}$, it cannot contain a collider in \(\mathrm{forb}(\mathbf{X},\mathbf{Y},\mathcal{D})\). If we suppose that all nodes on $p$ are in $\fb{\mathcal{D}}$, then its existence implies that no valid adjustment exists in $\mathcal{D}$, while the corresponding edge in the forbidden projection would imply the same for $\mathcal{D}^{\mathbf{XY}}$. Hence, we can suppose that $p$ contains at least one node not in $\fb{\mathcal{D}}$ and let $V_L$ be the last such node. Let $p'=(V_0,e_1,V_1,\dots,e_L,V_L)$ and $p''=(V_L,e_{L+1},V_1,\dots,e_K,V_K)$. By construction, $V_{L+1} \in \fb{\mathcal{D}}$ and since $\fb{\mathcal{D}} \setminus \mathbf{X} \subseteq \fb{\mathcal{D}}$, it follows that $p''$ is causal. Thus the forbidden projection will map $p''$ to the path $q''=V_L \rightarrow V_K$. This also implies that $p'$ is non-causal. 
	
	Suppose first that $V_0 \in \mathbf{X} \cap \de(\fb{\mathcal{D}} \setminus \mathbf{X},\mathcal{D})$. Then no valid adjustment set exists in $\mathcal{D}$. Further,  there must be a bi-directed edge from $\mathbf{X}$ to $\mathbf{Y}$ in $\mathcal{D}^{\mathbf{XY}}$ and hence that no valid adjustment set exists in $\mathcal{D}^{\mathbf{XY}}$ either. 
	We can hence suppose that $V_0 \notin \fb{\mathcal{D}} \setminus \mathbf{X}$. This implies that all nodes on $p'$ except $V_0$ are not in $\fb{\mathcal{D}}$. Since this implies that no node on $p'$ is in $\de(\mathbf{F},\mathcal{D})$ it follows that $p'$ is also a path in $\mathcal{D}^{\mathbf{XY}}$. The path $q=p'\oplus q''$ is a proper non-causal path from $\mathbf{X}$ to $\mathbf{Y}$ in $\mathcal{D}^{\mathbf{XY}}$. By the usual argument $p'$ is also open given $\mathbf{Z}$ in $\mathcal{D}^{\mathbf{XY}}$ and trivially, this is also true for $q''$. Further, $V_L \notin \mathbf{Z}$ is a non-collider on $q''$ and hence, $q$ is open given $\mathbf{Z}$ in $\mathcal{D}^{\mathbf{XY}}$.
\end{proof}

\setcounter{theorem}{9}
\begin{proposition}
	Let \(\mathbf{X}\) and \(\mathbf{Y}\) be disjoint subsets of the node set \(\mathbf{V}\) of a DAG \(\mathcal{D}\) such that \(\mathbf{Y}\subseteq\mathrm{de}(\mathbf{X}, \mathcal{D})\). Then \(\mathbf{O}(\mathbf{X},\mathbf{Y},\mathcal{D})=\mathbf{O}^*(\mathbf{X},\mathbf{Y},\mathcal{D})\).
\end{proposition}
\begin{proof}
	We first show that \(\mathbf{O}(\mathbf{X},\mathbf{Y},\mathcal{D})\subseteq\mathbf{O}^*(\mathbf{X},\mathbf{Y},\mathcal{D})\). Let \(Z\in\mathbf{O}(\mathbf{X},\mathbf{Y},\mathcal{D})\). By Definition \ref{Oset1}, \(\mathbf{O}(\mathbf{X},\mathbf{Y},\mathcal{D})\cap\mathrm{forb}(\mathbf{X},\mathbf{Y},\mathcal{D})=\emptyset\) and hence \(Z\) is a node in \(\mathcal{D}^{\mathbf{XY}}\). Furthermore, since \(\mathbf{O}(\mathbf{X},\mathbf{Y},\mathcal{D})\subseteq\mathrm{pa}(\mathrm{cn}(\mathbf{X},\mathbf{Y},\mathcal{D}),\mathcal{D})\), and \(\mathrm{cn}(\mathbf{X},\mathbf{Y},\mathcal{D})\subseteq\mathrm{forb}(\mathbf{X},\mathbf{Y},\mathcal{D})\), there is a node \(Y\in\mathbf{Y}\) such that \(\mathcal{D}\) contains a directed path \(Z\rightarrow\dots\rightarrow Y\) on which all non-endpoint nodes are in \(\mathrm{forb}(\mathbf{X},\mathbf{Y},\mathcal{D})\). Due to property 1 of Definition \ref{lproj}, this corresponds to an edge \(Z\rightarrow Y\) in \(\mathcal{D}^{\mathbf{XY}}\), hence \(Z\in\mathbf{O}^*(\mathbf{X},\mathbf{Y},\mathcal{D})\).
	
	Next, we show that \(\mathbf{O}^*(\mathbf{X},\mathbf{Y},\mathcal{D})\subseteq\mathbf{O}(\mathbf{X},\mathbf{Y},\mathcal{D})\). Let \(Z^*\in\mathbf{O}^*(\mathbf{X},\mathbf{Y},\mathcal{D})\). By Definition \ref{fproj}, this implies that \(Z^*\in\mathbf{V}\setminus(\mathrm{forb}(\mathbf{X},\mathbf{Y},\mathcal{D})\cup\mathbf{X}\cup\mathbf{Y})\). Moreover, by Definition \ref{Oset2}, there is an edge \(Z^*\rightarrow Y^*\) in \(\mathcal{D}^{\mathbf{XY}}\) with \(Y^*\in\mathbf{Y}\). In \(\mathcal{D}\), this corresponds to a directed path \(Z^*\rightarrow\dots\rightarrow Y^*\) on which all non-endpoint nodes are in \(\mathrm{forb}(\mathbf{X},\mathbf{Y},\mathcal{D})\), and \(Z^*\not\in\mathrm{forb}(\mathbf{X},\mathbf{Y},\mathcal{D})\). Denote the path by \(p\). There are two cases: In the first case, \(p\) has no non-endpoint nodes, i.e.\ \(\mathcal{D}\) contains the edge \(Z^*\rightarrow Y^*\). Since we assume \(\mathbf{Y}\subseteq\mathrm{de}(\mathbf{X},\mathcal{D})\), \(Y^*\) must be in \(\mathrm{cn}(\mathbf{X},\mathbf{Y},\mathcal{D})\), hence \(Z^*\in\mathbf{O}(\mathbf{X},\mathbf{Y},\mathcal{D})\). In the second case, \(p\) has at least one non-endpoint node. This means that \(Z^*\in\mathrm{pa}(W,\mathcal{D})\), where  \(W\in\mathrm{forb}(\mathbf{X},\mathbf{Y},\mathcal{D})\setminus(\mathbf{X}\cup\mathbf{Y})\) and \(W\in\mathrm{an}(Y^*,\mathcal{D})\). Since in a DAG, all forbidden nodes are descendants of \(\mathbf{X}\), we also have \(W\in\mathrm{de}(\mathbf{X},\mathcal{D})\), and hence \(W\in\mathrm{cn}(\mathbf{X},\mathbf{Y}, \mathcal{D})\). It follows that \(Z^*\in\mathbf{O}(\mathbf{X},\mathbf{Y},\mathcal{D})\).
\end{proof}

\section[Appendix C: Generalisation of the Forbidden Projection and the \(\mathbf{O}^*\)-Set to Amenable MaxPDAGs]{Generalisation of the Forbidden Projection and the \(\mathbf{O}^*\)-Set to Amenable MaxPDAGs}
\label{appC}

In this appendix, we generalise the forbidden projection (Definition \ref{fproj}) and the \(\mathbf{O}^*\)-set (Definition \ref{Oset2}) to amenable maxPDAGs and show that Propositions similar to \ref{nobiDAG}, \ref{VAS}, \ref{XYDAG} and \ref{OOequal} still hold for the more general definitions.

The latent projection in general (Definition \ref{lproj}) cannot be generalised to (amenable) maxPDAGs as marginalising does not generally result in an ADMG. As an example, consider the maxPDAG \(W_1-L\rightarrow W_2\) with latent node \(L\). It is not clear how the projection should be constructed in this case: \(W_1\rightarrow W_2\) would give the wrong impression that \(W_1\) is an ancestor of \(W_2\) (instead of a \textit{possible} ancestor), while \(W_1-W_2\) would imply that \(W_2\) is a possible ancestor of \(W_1\). As we will show in the following propositions, however, the latent projection can be meaningfully generalised to amenable maxPDAGs when projecting over the special case of a forbidden set.

\setcounter{theorem}{16}
\begin{definition}[Forbidden projection for amenable maxPDAGs]
	\label{gfprof}
	Let \(\mathcal{G}\) be a maxPDAG with node set \(\mathbf{V}\),  and let \(\mathbf{X}\subset\mathbf{V}\) and \(Y\in\mathbf{V}\setminus\mathbf{X}\) such that \(\mathcal{G}\) is amenable relative to \((\mathbf{X},Y)\). Define \(\mathbf{F}=\mathrm{forb}(\mathbf{X},Y,\mathcal{G})\setminus(\mathbf{X}\cup\{Y\})\). The \emph{forbidden projection} \(\mathcal{G}^{\mathbf{X}Y}\) of \(\mathcal{G}\) is a graph with node set \(\mathbf{V}\setminus\mathbf{F}\) and edges as follows: For distinct nodes \(W_i,W_j\in\mathbf{V}\setminus\mathbf{F}\),
	\begin{enumerate}
		\item \(\mathcal{G}^{\mathbf{X}Y}\) contains a directed edge \(W_i\rightarrow W_j\) if and only if \(\mathcal{G}\) contains a directed path \(W_i\rightarrow\dots\rightarrow W_j\) on which all non-endpoint nodes are in \(\mathbf{F}\),
		\item \(\mathcal{G}^{\mathbf{X}Y}\) contains a bi-directed edge \(W_i\leftrightarrow W_j\) if and only if \(\mathcal{G}\) contains a path, with at least one non-endpoint node, of the form \(W_i\leftarrow\dots\rightarrow W_j\) on which all non-endpoints are non-colliders and in \(\mathbf{F}\),
		\item \(\mathcal{G}^{\mathbf{X}Y}\) contains an undirected edge \(W_i-W_j\) if and only if \(\mathcal{G}\) contains \(W_i-W_j\).
	\end{enumerate}
\end{definition}

Note that we restrict the definition to singleton \(Y\). This is because with a set \(\mathbf{Y}\), we run into similar construction/interpretation problems as described above. Consider, for example, an amenable maxPDAG with node set \(\{X,F,Y_1,Y_2\}\) and edges \(X\rightarrow Y_1-F\rightarrow Y_2\) as well as \(X\rightarrow F\). None of \(Y_1\rightarrow Y_2\), \(Y_1\leftrightarrow Y_2\) or \(Y_1-Y_2\) are correct representations of the marginal distribution.

Before generalising the \(\mathbf{O}^*\)-set, we now describe the properties of the forbidden projection for maxPDAGs. A key property is that if a valid adjustment set exists, the forbidden projection of a maxPDAG is itself a maxPDAG (Proposition \ref{keyprop}). This is analogous to Proposition \ref{XYDAG} for DAGs. Proposition \ref{nobiPDAG}, in analogy to Proposition \ref{nobiDAG}, states that if \(\mathbf{X}\) has a causal effect on \(Y\), then the forbidden projection can be used to check whether a valid adjustment set exists. In Proposition \ref{VASPDAG}, we show that a set \(\mathbf{Z}\) is a valid adjustment set in the forbidden projection if and only if it is a valid adjustment set in the original graph, which is analogous to Proposition \ref{VAS}.

We begin with a lemma that will allow us to use \(\mathrm{possde}(\mathbf{X},\mathcal{G})\) and \(\mathrm{de}(\mathbf{X},\mathcal{G})\) interchangeably when \(\mathcal{G}\) is an amenable maxPDAG.

\begin{lemma}
	\label{mylemma}
	Let \(p=(V_1,V_2,\dots,V_K)\) be a possibly directed path in a maxPDAG \(\mathcal{G}\) such that no node on \(p\) shares an undirected edge with \(V_1\). Then a subsequence of \(p\) forms a directed path from \(V_1\) to \(V_K\) in \(\mathcal{G}\).
\end{lemma}
\begin{proof}
	We show this by induction. By assumption, \((V_1,V_2)\) is a subsequence of \(p\) and forms a directed path from \(V_1\) to \(V_2\). Now assume that a subsequence of \(p\) forms a directed path from \(V_1\) to \(V_{k-1}\), for \(2<k\le K\). Denote this subsequence by \((V_1=W_1,W_2,\dots,W_Q=V_{k-1})\). Clearly, if \(V_{k-1}\rightarrow V_k\) then \((V_1=W_1,W_2,\dots,W_Q=V_{k-1},V_k)\) is a subsequence of \(p\) and forms a directed path from \(V_1\) to \(V_k\) in \(\mathcal{G}\), which is what we wanted to show. If, on the other hand, if \(V_{k-1}-V_k\)  then there are four cases, three of which lead to a contradiction:
	
	(1) The induced subgraph of \(\mathcal{G}\) on \(\{W_{Q-1}, W_Q=W_{k-1}, V_k\}\) is Graph 1 in Figure~\ref{fig:mylemma}. Then \(\mathcal{G}\) is not closed under Meek's Rule 1 (see Figure~\ref{fig:Meek}), which is a contradiction.
	
	(2) The induced subgraph of \(\mathcal{G}\) on \(\{W_{Q-1}, W_Q=W_{k-1}, V_k\}\) is Graph 2 in Figure~\ref{fig:mylemma}. Then \(\mathcal{G}\) is not closed under Meek's Rule 2, which is a contradiction.
	
	(3) The induced subgraph of \(\mathcal{G}\) on \(\{W_{Q-1}, W_Q=W_{k-1}, V_k\}\) is Graph 3 in Figure~\ref{fig:mylemma}. This implies the induced subgraph of \(\mathcal{G}\) on \(\{W_{Q-2}, W_{Q-1}, V_k\}\) would also be graph 3 (with the same reasons as above excluding graphs 1 and 2). Repeating the argument for \(W_{Q-3}, W_{Q-4},\dots\) implies an undirected edge between \(W_1=V_1\) and \(V_k\), which is a contradiction.
	
	(4) The induced subgraph of \(\mathcal{G}\) on \(\{W_{Q-1}, W_Q=W_{k-1}, V_k\}\) is Graph 4 in Figure~\ref{fig:mylemma}. Then \((V_1=W_1,W_2,\dots,W_{Q-1},V_{k})\) is a subsequence of \(p\) and forms a directed path from \(V_1\) to \(V_k\), which is what we wanted to show.
\end{proof}

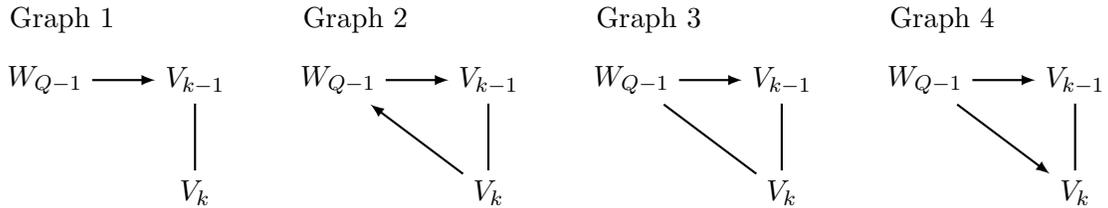
\begin{figure}[h]
	\begin{center}
	\begin{tikzpicture}[node distance=15mm, >=latex]
	\node[] (W1) {\(W_{Q-1}\)};
	\node[above of=W1, xshift=-6mm, yshift=-7mm, anchor=west] (R1) {Graph 1};
	\node[right of=W1, xshift=5mm] (V1) {\(V_{k-1}\)};
	\node[below of=V1] (Vk1) {\(V_k\)};
	\draw[->, thick] (W1) to (V1);
	\draw[-, thick] (V1) to (Vk1);
	
	\node[right of=W1, xshift=24mm] (W2) {\(W_{Q-1}\)};
	\node[above of=W2, xshift=-6mm, yshift=-7mm, anchor=west] (R2) {Graph 2};
	\node[right of=W2, xshift=5mm] (V2) {\(V_{k-1}\)};
	\node[below of=V2] (Vk2) {\(V_k\)};
	\draw[->, thick] (W2) to (V2);
	\draw[-, thick] (V2) to (Vk2);
	\draw[->, thick] (Vk2) to (W2);
	
	\node[right of=W2, xshift=24mm] (W3) {\(W_{Q-1}\)};
	\node[above of=W3, xshift=-6mm, yshift=-7mm, anchor=west] (R3) {Graph 3};
	\node[right of=W3, xshift=5mm] (V3) {\(V_{k-1}\)};
	\node[below of=V3] (Vk3) {\(V_k\)};
	\draw[->, thick] (W3) to (V3);
	\draw[-, thick] (V3) to (Vk3);
	\draw[-, thick] (Vk3) to (W3);
	
	\node[right of=W3, xshift=24mm] (W4) {\(W_{Q-1}\)};
	\node[above of=W4, xshift=-6mm, yshift=-7mm, anchor=west] (R4) {Graph 4};
	\node[right of=W4, xshift=5mm] (V4) {\(V_{k-1}\)};
	\node[below of=V4] (Vk4) {\(V_k\)};
	\draw[->, thick] (W4) to (V4);
	\draw[-, thick] (V4) to (Vk4);
	\draw[<-, thick] (Vk4) to (W4);
	
	\end{tikzpicture}
	\end{center}
	\caption{Graphs for the proof of Lemma \ref{mylemma}.}
	\label{fig:mylemma}
\end{figure}

\begin{proposition}
	\label{nobiPDAG}
	Let \(\mathcal{G}\) be a causal maxPDAG with node set \(\mathbf{V}\), and let \(\mathbf{X}\subset\mathbf{V}\) and \(Y\in\mathbf{V}\setminus\mathbf{X}\) such that \(Y\in\mathrm{possde}(\mathbf{X},\mathcal{G})\) and \(\mathcal{G}\) is amenable relative to \((\mathbf{X},Y)\). Then a valid adjustment set relative to \((\mathbf{X},Y)\) in \(\mathcal{G}\) exists if and only if there is no bi-directed edge between any nodes in the forbidden projection \(\mathcal{G}^{\mathbf{X}Y}\).
\end{proposition}
\begin{proof}
	By Lemma \ref{mylemma}, \(Y\in\mathrm{possde}(\mathbf{X},\mathcal{G})\) implies \(Y\in\mathrm{de}(\mathbf{X},\mathcal{G})\). The proof is now analogous to the proof of Proposition \ref{nobiDAG}, with Lemma \ref{EmaCorollary} replaced by Lemma \ref{generalisedEmaCorollary}.
\end{proof}

\begin{lemma}
	\label{generalisedEmaCorollary}
	Let \(\mathcal{G}\) be a causal maxPDAG with node set \(\mathbf{V}\), and let \(\mathbf{X}\subset\mathbf{V}\) and \(Y\in\mathbf{V}\setminus\mathbf{X}\) such that \(Y\in\mathrm{de}(\mathbf{X},\mathcal{G})\) and \(\mathcal{G}\) is amenable relative to \((\mathbf{X},Y)\). Then a valid adjustment set relative to \((\mathbf{X},Y)\) in \(\mathcal{G}\) exists if and only if \(\mathbf{X}\cap\mathrm{de}(\mathrm{cn}(\mathbf{X},Y,\mathcal{G}),\mathcal{G})=\emptyset\).
\end{lemma}
\begin{proof}
	We show that \textit{no} valid adjustment set relative to \((\mathbf{X},Y)\) in \(\mathcal{G}\) exists if and only if \(\mathbf{X}\cap\mathrm{de}(\mathrm{cn}(\mathbf{X},Y,\mathcal{G}),\mathcal{G})\ne\emptyset\). The proof is similar to the proof of Corollary 27 in \cite{Perkovicetal2018}.
	
	Assume first that no valid adjustment set relative to \((\mathbf{X},Y)\) in \(\mathcal{G}\) exists, then by Lemma \ref{adjust}, there is a proper non-causal definite-status path from some \(X\in\mathbf{X}\) to \(Y\) that is open given \(\mathrm{adjust}(\mathbf{X},Y,\mathcal{G})=\mathrm{possan}(\mathbf{X}\cup\{Y\},\mathcal{G})\setminus(\mathbf{X}\cup\mathrm{forb}(\mathbf{X},Y,\mathcal{G}))\). Denote one such path by \(p\). Assume for contradiction that \(p\) contains a collider and denote the collider by \(C\). As \(p\) is open given \(\mathrm{adjust}(\mathbf{X},Y,\mathcal{G})\), a descendant of \(C\) is in \(\mathrm{adjust}(\mathbf{X},Y,\mathcal{G})\). This implies that \(\mathrm{an}(C,\mathcal{G})\cap\mathrm{forb}(\mathbf{X},Y,\mathcal{G})=\emptyset\), as otherwise all descendants of \(C\) would be in \(\mathrm{forb}(\mathbf{X},Y,\mathcal{G})\) and could not be in \(\mathrm{adjust}(\mathbf{X},Y,\mathcal{G})\). As \(Y\in\mathrm{forb}(\mathbf{X},Y,\mathcal{G})\), it follows that at least one of the nodes adjacent to \(C\) on \(p\) must be a non-endpoint non-collider on \(p\). Denote one such node by \(B\). As \(B\in\mathrm{pa}(C,\mathcal{G})\), \(B\not\in\mathrm{forb}(\mathbf{X},Y,\mathcal{G})\) and \(B\in\mathrm{adjust}(\mathbf{X},Y,\mathcal{G})\). But then \(p\) is not open given \(\mathrm{adjust}(\mathbf{X},Y,\mathcal{G})\), which is a contradiction. Hence, \(p\) does not contain a collider. As \(p\) is non-causal, \(p\) cannot be directed towards \(Y\), and as we assume that \(Y\in\mathrm{de}(\mathbf{X},\mathcal{G})\), \(p\) cannot be directed towards \(X\). Hence, \(p\) is a path of the form \(X\leftarrow\dots\leftarrow A\rightarrow\dots\rightarrow Y\), where every non-endpoint is a non-collider not in \(\mathrm{adjust}(\mathbf{X},Y,\mathcal{G})\). It follows that \(A\) is in \(\mathrm{forb}(\mathbf{X},Y,\mathcal{G})\) and thus is a descendant of \(\mathbf{X}\), which implies that \(X\in\mathrm{de}(\mathrm{cn}(\mathbf{X},Y,\mathcal{G}),\mathcal{G})\) and hence \(\mathbf{X}\cap\mathrm{de}(\mathrm{cn}(\mathbf{X},Y,\mathcal{G}),\mathcal{G})\ne\emptyset\).
	
	Assume now that \(\mathbf{X}\cap\mathrm{de}(\mathrm{cn}(\mathbf{X},Y,\mathcal{G}),\mathcal{G})\ne\emptyset\). Pick a node from \(\mathbf{X}\cap\mathrm{de}(\mathrm{cn}(\mathbf{X},Y,\mathcal{G}),\mathcal{G})\) and denote it by \(X^*\). Then there must exist a node \(C*\in\mathrm{cn}(\mathbf{X},Y,\mathcal{G})\) such that there is a path of the form \(X^*\leftarrow\dots\leftarrow C^*\rightarrow\dots\rightarrow\) where that all non-endpoint non-colliders on the path are in the forbidden set. This path cannot be blocked by any set of non-forbidden nodes.
\end{proof}

\begin{lemma}[Theorem 5.6 in \citealp{PerkovicKalischMaathuis2017}]
	\label{adjust}
	Let \(\mathbf{X}\) and \(\mathbf{Y}\) be disjoint node sets in a causal maxPDAG \(\mathcal{G}\) such that \(\mathcal{G}\) is amenable relative to \((\mathbf{X},\mathbf{Y})\), and let \(\mathrm{adjust}(\mathbf{X},\mathbf{Y},\mathcal{G})=\mathrm{possan}(\mathbf{X}\cup\mathbf{Y},\mathcal{G})\setminus(\mathbf{X}\cup\mathbf{Y}\cup\mathrm{forb}(\mathbf{X},\mathbf{Y},\mathcal{G}))\). Then a valid adjustment set relative to \((\mathbf{X},\mathbf{Y})\) in \(\mathcal{D}\) exists if and only if all proper non-causal definite-status paths from \(\mathbf{X}\) to \(\mathbf{Y}\) are blocked by \(\mathrm{adjust}(\mathbf{X},\mathbf{Y},\mathcal{G})\) in \(\mathcal{G}\).
\end{lemma}

\begin{proposition}
	\label{keyprop}
	Let \(\mathcal{G}\) be a causal maxPDAG with node set \(\mathbf{V}\), and let \(\mathbf{X}\subset\mathbf{V}\) and \(Y\in\mathbf{V}\setminus\mathbf{X}\) such that \(\mathcal{G}\) is amenable relative to \((\mathbf{X},Y)\) and there exists a valid adjustment set relative to \((\mathbf{X},Y)\) in \(\mathcal{G}\). Denote the set of DAGs represented by \(\mathcal{G}\) by \([\mathcal{G}]=\{\mathcal{D}_1, \mathcal{D}_2, \dots, \mathcal{D}_M\}\). Then the forbidden projection \(\mathcal{G}^{\mathbf{X}Y}\) is the causal maxPDAG representing the DAGs in \(\{\mathcal{D}_1^{\mathbf{X}Y}, \mathcal{D}_2^{\mathbf{X}Y}, \dots, \mathcal{D}_M^{\mathbf{X}Y}\}\).
\end{proposition}
\begin{proof}
	We only consider the case that \(Y\in\mathrm{possde}(\mathbf{X},\mathcal{G})\), as otherwise the proposition follows trivially from the fact that \(\mathcal{G}^{\mathbf{X}Y}=\mathcal{G}\). By Lemma \ref{mylemma}, \(Y\in\mathrm{possde}(\mathbf{X},\mathcal{G})\) implies \(Y\in\mathrm{de}(\mathbf{X},\mathcal{G})\). We know from Propositions \ref{nobiDAG} and \ref{nobiPDAG} that none of \(\mathcal{G}^{\mathbf{X}Y}, \mathcal{D}_1^{\mathbf{X}Y},\) \(\mathcal{D}_2^{\mathbf{X}Y}, \dots, \mathcal{D}_M^{\mathbf{X}Y}\) contain any bi-directed edges. Consider edges present in the latent projections but not in the original graphs: For the maxPDAG \(\mathcal{G}\), denote the set of edges in \(\mathcal{G}^{\mathbf{X}Y}\) but not in \(\mathcal{G}\) by \(\mathbf{e}(\mathcal{G})\), and define analogous sets \(\mathbf{e}(\mathcal{D}_1), \mathbf{e}(\mathcal{D}_2), \dots, \mathbf{e}(\mathcal{D}_M)\) for the DAGs \(\mathcal{D}_1, \mathcal{D}_2, \dots, \mathcal{D}_M\). None of \(\mathbf{e}(\mathcal{G}),\mathbf{e}(\mathcal{D}_1), \mathbf{e}(\mathcal{D}_2), \dots, \mathbf{e}(\mathcal{D}_M)\) contain any undirected edges. Further, any directed edges in any of \(\mathbf{e}(\mathcal{G}),\mathbf{e}(\mathcal{D}_1), \mathbf{e}(\mathcal{D}_2), \dots, \mathbf{e}(\mathcal{D}_M)\) are into \(Y\). This is because for every \(\mathcal{G}'\in\{\mathcal{G},\mathcal{D}_1, \mathcal{D}_2, \dots, \mathcal{D}_M\}\), an edge in \(\mathbf{e}(\mathcal{G}')\) corresponds to a directed path with at least one forbidden non-endpoint node in \(\mathcal{G}'\). If an edge in \(\mathbf{e}(\mathcal{G}')\) was into a node \(V\in\mathbf{V}\setminus(\mathbf{X}\cup\{Y\})\), then \(V\) would  itself be forbidden, which is a contradiction. If an edge in \(\mathbf{e}(\mathcal{G}')\) was into a node \(X\in\mathbf{X}\), then \(X\) would be in \(\mathrm{de}(\mathrm{cn}(\mathbf{X},Y,\mathcal{G}'),\mathcal{G}')\), which by Lemma \ref{generalisedEmaCorollary} contradicts our assumption that a valid adjustment set exists relative to \((\mathbf{X},Y)\) in \(\mathcal{G}'\). Hence, all edges in all of \(\mathbf{e}(\mathcal{G}),\mathbf{e}(\mathcal{D}_1), \mathbf{e}(\mathcal{D}_2), \dots, \mathbf{e}(\mathcal{D}_M)\) are into \(Y\). In fact, by Lemma \ref{adddir} below, \(\mathbf{e}(\mathcal{G})=\mathbf{e}(\mathcal{D}_1)=\mathbf{e}(\mathcal{D}_2)=\dots= \mathbf{e}(\mathcal{D}_M)=\mathbf{e}\). The graphs \(\mathcal{G}^{\mathbf{X}Y},\mathcal{D}_1^{\mathbf{X}Y}, \mathcal{D}_2^{\mathbf{X}Y}, \dots, \mathcal{D}_M^{\mathbf{X}Y}\) can thus be constructed by copying the induced subgraphs of \(\mathcal{G},\mathcal{D}_1, \mathcal{D}_2, \dots, \mathcal{D}_M\) with respect to \((\mathbf{V}\setminus\mathrm{forb}(\mathbf{X},Y,\mathcal{G}))\cup(\mathbf{X},Y)\) and adding the edges in \(\mathbf{e}\). Hence, \(\mathcal{G}^{\mathbf{X}Y}\) represents the DAGs in \(\{\mathcal{D}_1^{\mathbf{X}Y}, \mathcal{D}_2^{\mathbf{X}Y}, \dots, \mathcal{D}_M^{\mathbf{X}Y}\}\) in the sense that every directed edge in \(\mathcal{G}^{\mathbf{X}Y}\) is also present in all DAGs in \(\{\mathcal{D}_1^{\mathbf{X}Y}, \mathcal{D}_2^{\mathbf{X}Y}, \dots, \mathcal{D}_M^{\mathbf{X}Y}\}\), and for every undirected edge \(V_i-V_j\) in \(\mathcal{G}^{\mathbf{X}Y}\), there is at least one DAG in \(\{\mathcal{D}_1^{\mathbf{X}Y}, \mathcal{D}_2^{\mathbf{X}Y}, \dots, \mathcal{D}_M^{\mathbf{X}Y}\}\) with \(V_i\rightarrow V_j\) and at least one with \(V_i\leftarrow V_j\).
	
	In order show that \(\mathcal{G}^{\mathbf{X}Y}\) has all the characteristics of a maxPDAG, we show that \(\mathcal{G}^{\mathbf{X}Y}\) is closed under Meek's rules. Referring to Figure~\ref{fig:Meek}, we argue that the graphs on the  left-hand sides of Rules 1 -- 4 cannot be induced subgraphs of \(\mathcal{G}^{\mathbf{X}Y}\). Assume for contradiction that the left-hand graph of Rule 1, \(\rightarrow -\), was an induced subgraph of \(\mathcal{G}^{\mathbf{X}Y}\). As this graph is not an induced subgraph of \(\mathcal{G}\) by assumption, and all of \(\mathbf{e}(\mathcal{G}),\mathbf{e}(\mathcal{D}_1), \mathbf{e}(\mathcal{D}_2), \dots, \mathbf{e}(\mathcal{D}_M)\) consist of only directed edges into \(Y\), we can conclude that the directed edge in \(\rightarrow -\) is into \(Y\), i.e.\ \(\rightarrow Y-\). Hence, \(Y\) shares an undirected edge with some node \(V\) in \(\mathcal{G}\), but this means that \(V\) is a forbidden node in some \(\mathcal{D}\in[\mathcal{G}]\), which is not allowed according to Lemma \ref{forbforb}. By similar arguments, none of the graphs on the left-hand sides of Rules 1 -- 4 in Figure~\ref{fig:Meek} is an induced subgraph of \(\mathcal{G}^{\mathbf{X}Y}\).
\end{proof}

\begin{lemma}
	\label{adddir}
	Let \(\mathcal{G}\) be a causal maxPDAG with node set \(\mathbf{V}\), and let \(\mathbf{X}\subset\mathbf{V}\) and \(Y\in\mathbf{V}\setminus\mathbf{X}\) such that \(\mathcal{G}\) is amenable relative to \((\mathbf{X},Y)\) and there exists a valid adjustment set relative to \((\mathbf{X},Y)\) in \(\mathcal{G}\). Define \(\mathbf{F}=\mathrm{forb}(\mathbf{X},Y,\mathcal{G})\setminus(\mathbf{X}\cup\{Y\})\) and pick a node \(V_1\in\mathbf{V}\setminus\mathbf{F}\). Then the following two statements are equivalent:
	\begin{enumerate}
		\item[(i)] A DAG \(\mathcal{D}\in[\mathcal{G}]\) contains a directed path \(p=(V_1,V_2,\dots,V_K=Y)\) such that all non-endpoint nodes on \(p\) are in \(\mathbf{F}\).
		\item[(ii)] The maxPDAG \(\mathcal{G}\) contains a directed path \(q=(V_1=W_1,W_2,\dots,W_Q=Y)\) such that all non-endpoint nodes on \(q\) are in \(\mathbf{F}\).
	\end{enumerate}
\end{lemma}
\begin{proof}
	Statement (ii) implies that the directed path \(p\) is present in all DAGs in \([\mathcal{G}]\) by the defining properties of a maxPDAG. Hence, we only show that (i) implies (ii). Again by the properties of maxPDAGs, the sequence of nodes \((V_1,V_2,\dots,V_K=Y)\) forms a possibly directed path from \(V_1\) to \(Y\) in \(\mathcal{G}\). We first show that no node in \(\{V_2,\dots,V_K=Y\}\) shares an undirected edge with \(V_1\). Suppose, for contradiction, that node \(V_k, 2\le k\le K\) shares an undirected edge with \(V_1\) and distinguish two cases: (1) \(V_1\in\mathbf{X}\), (2) \(V_1\in\mathbf{V}\setminus\mathbf{X}\). The first case contradicts our assumption that \(\mathcal{G}\) is amenable relative to \((\mathbf{X},Y)\). The second case implies that \(V_1\), as a possible descendant of \(V_k\), is in \(\mathbf{F}\), but we chose \(V_1\) such that \(V_1\in\mathbf{V}\setminus\mathbf{F}\). Hence, no node in \(\{V_2,\dots,V_K=Y\}\) shares an undirected edge with \(V_1\). We can thus apply Lemma \ref{mylemma} and conclude that a subsequence of \((V_1,V_2,\dots,V_K=Y)\) forms a directed path from \(V_1\) to \(Y\) in \(\mathcal{G}\), which implies that statement (ii) holds.
\end{proof}

\begin{lemma}[Lemma E.8 in \citetalias{HenckelPerkovicMaathuis2019}]
	\label{forbforb}
	Let \(\mathbf{X}\) and \(\mathbf{Y}\) be disjoint node sets in a maxPDAG \(\mathcal{G}\), such that \(\mathcal{G}\) is amenable relative to \((\mathbf{X},\mathbf{Y})\), and let \(\mathcal{D}\) be a DAG in \([\mathcal{G}]\). Then \(\mathrm{forb}(\mathbf{X},\mathbf{Y},\mathcal{G})=\mathrm{forb}(\mathbf{X},\mathbf{Y},\mathcal{D})\).
\end{lemma}

\begin{proposition}
	\label{VASPDAG}
	Let \(\mathcal{G}\) be a causal maxPDAG with node set \(\mathbf{V}\) and let \(\mathbf{X}\subset\mathbf{V}\) and \(Y\in\mathbf{V}\setminus\mathbf{X}\) such that \(\mathcal{G}\) is amenable relative to \((\mathbf{X},Y)\). Then a set \(\mathbf{Z}\) is a valid adjustment set relative to \((\mathbf{X},Y)\) in \(\mathcal{G}\) if and only if it is a valid adjustment set relative to \((\mathbf{X},Y)\) in the forbidden projection \(\mathcal{G}^{\mathbf{X}Y}\).
\end{proposition}
\begin{proof}
	Let \(\mathcal{D}\in[\mathcal{G}]\) and \(D^{\mathbf{X}Y}\) its forbidden projection. By Proposition \ref{VAS}, a set \(\mathbf{Z}\) is a valid adjustment set relative to \((\mathbf{X},Y)\) in \(\mathcal{D}\) if and only if it is a valid adjustment set relative to \((\mathbf{X},Y)\) in \(\mathcal{D}^{\mathbf{X}Y}\). By Proposition \ref{keyprop}, the set \([\mathcal{G}^{\mathbf{X}Y}]\) contains exactly the forbidden projections of all DAGs in \([\mathcal{G}]\). Hence if \(\mathbf{Z}\) is a valid adjustment set in all \(\mathcal{D}\in[\mathcal{G}]\), then it is a valid adjustment set in all \(\mathcal{D}^{\mathbf{X}Y}\in[\mathcal{G}^{\mathbf{X}Y}]\) and vice versa.
\end{proof}

To summarise, the forbidden projection for amenable causal maxPDAGs has very similar properties as the forbidden projection for causal DAGs, as long as we consider a singleton outcome node \(Y\): Bi-directed edges in the projection indicate the lack of a valid adjustment set; if a valid set exists, the forbidden projection is a maxPDAG itself, preserving all the information relevant to causal effect identification via adjustment; in particular, all valid sets can be read off the forbidden projection as well as the original graph.

Finally, we now generalise Definition \ref{Oset2} of the \(\mathbf{O}^*\)-set and its optimality property in Proposition \ref{OOequal} to amenable maxPDAGs with singleton \(Y\).

\begin{definition}
	Let \(\mathcal{G}\) be a causal maxPDAG with node set \(\mathbf{V}\), let \(\mathbf{X}\subset\mathbf{V}\) and \(Y\in\mathbf{V}\setminus\mathbf{X}\) such that \(\mathcal{G}\) is amenable relative to \((\mathbf{X},Y)\), and let \(\mathcal{G}^{\mathbf{X}Y}\) be the corresponding forbidden projection. We define \(\mathbf{O}^*(\mathbf{X},Y,\mathcal{G})\) as:
	\[
	\mathbf{O}^*(\mathbf{X},Y,\mathcal{G})=\mathrm{pa}(Y,\mathcal{G}^{\mathbf{XY}})\setminus \mathbf{X}.
	\]
\end{definition}

\begin{proposition}
	Let \(\mathcal{G}\) be a causal maxPDAG with node set \(\mathbf{V}\), let \(\mathbf{X}\subset\mathbf{V}\) and \(Y\in\mathbf{V}\setminus\mathbf{X}\) such that \(Y\in\mathrm{possde}(\mathbf{X}, \mathcal{D})\) and \(\mathcal{G}\) is amenable relative to \((\mathbf{X},Y)\), let \(\mathbf{Z}\) be a valid adjustment set relative to \((\mathbf{X},Y)\) in \(\mathcal{G}\) and let \(\mathbf{O}^*=\mathbf{O}^*(\mathbf{X},Y,\mathcal{G})\). Then \(\mathbf{O}^*\) is a valid adjustment set relative to \((\mathbf{X},Y)\) in \(\mathcal{G}\) and if the variables \(\mathbf{V}\) follow a linear causal model compatible with \(\mathcal{G}\), then, for every \(X_i\in\mathbf{X}\), \(a.var(\hat{\beta}_{yx_i.\mathbf{x}_{-i}\mathbf{o}^*})\le a.var(\hat{\beta}_{yx_i.\mathbf{x}_{-i}\mathbf{z}})\).
\end{proposition}

\begin{proof}
	We prove this by showing that \(\mathbf{O}(\mathbf{X},Y,\mathcal{G})\) and \(\mathbf{O}^*(\mathbf{X},Y,\mathcal{G})\) are equal and invoking Propositions \ref{Osuff} and \ref{eff}. By Lemma \ref{mylemma}, \(Y\in\mathrm{possde}(\mathbf{X},\mathcal{G})\) implies \(Y\in\mathrm{de}(\mathbf{X},\mathcal{G})\). Then the equivalence follow directly from Lemma \ref{OO} in combination with Proposition \ref{OOequal}: For every DAG \(\mathcal{D}\in[\mathcal{G}]\), \(\mathbf{O}(\mathbf{X},Y,\mathcal{G})=\mathbf{O}(\mathbf{X},Y,\mathcal{D})=\mathbf{O}^*(\mathbf{X},Y,\mathcal{D})=\mathbf{O}^*(\mathbf{X},Y,\mathcal{G})\).
\end{proof}

\begin{lemma}[Lemma E.7 in \citetalias{HenckelPerkovicMaathuis2019}]
	\label{OO}
	Let \(\mathbf{X}\) and \(\mathbf{Y}\) be disjoint node sets in a maxPDAG \(\mathcal{G}\) such that \(\mathcal{G}\) is amenable relative to \((\mathbf{X},\mathbf{Y})\) in \(\mathcal{G}\), let	 \(\mathbf{Y}\subseteq\mathrm{possde}(\mathbf{X},\mathcal{G})\), and let \(\mathcal{D}\) be a DAG in \([\mathcal{G}]\). Then	  \(\mathbf{O}(\mathbf{X},\mathbf{Y},\mathcal{D})=\mathbf{O}(\mathbf{X},\mathbf{Y},\mathcal{G})\).
\end{lemma}

\section[Appendix D: Proof for Section 4]{Proof for Section 4}
\label{appD}

\setcounter{theorem}{10}
\begin{proposition}
	Let $X$ and $Y$ be nodes in a causal CPDAG or maxPDAG $\mathcal{G} = (\mathbf{V},\mathbf{E})$, such that $\mathbf{V}$ follows a causal linear model compatible with $\mathcal{G}$ with Gaussian errors. Let $\widehat{\Theta}^{\mathbf{P}}$ and $\widehat{\Theta}^{\mathbf{O}}$ be the multisets returned by semi-local IDA and optimal IDA respectively, applied to \(X\), \(Y\) and \(\mathcal{G}\), with the subsets of \(\mathrm{sib}(X,\mathcal{G})\) considered in the same order for both. Then, for $i \in\{1 \dots, k\}$, with $k = |\widehat{\Theta}^{\mathbf{P}}| = |\widehat{\Theta}^{\mathbf{O}}|$,
	\begin{enumerate}
		\item $\mathbb{E}[\widehat{\Theta}^{\mathbf{P}}_i] = \mathbb{E}[\widehat{\Theta}^{\mathbf{O}}_i]$ 
		and 
		\item $a.var(\widehat{\Theta}^{\mathbf{P}}_i) \geq a.var(\widehat{\Theta}^{\mathbf{O}}_i)$. 
	\end{enumerate}
\end{proposition}

\begin{proof}
	Consider any set $\mathbf{S}_i \subseteq \mathrm{sib}(X,\mathcal{G})$. \cite{PerkovicKalischMaathuis2017} showed that there exists a DAG \(\mathcal{D} \in [\mathcal{G}]\), such that $\mathbf{P}_i=\mathrm{pa}(X,\mathcal{D})=\mathbf{S}_i \cup \mathrm{pa}(X,\mathcal{G})$, if and only if directing the edges in the neighbourhood of $X$ according to $\mathbf{P}_i$ and applying Meek's orientation rules results in a valid maxPDAG $\mathcal{G}'_i$. If this is not the case for $\mathbf{S}_i$, both algorithms discard $\mathbf{S}_i$ at their respective line 8. We can hence suppose that there exists a DAG \(\mathcal{D} \in [\mathcal{G}]\), such that $\mathbf{P}_i = \mathrm{pa}(X,\mathcal{D})=\mathbf{S}_i \cup \mathrm{pa}(X,\mathcal{G})$.	
	
	Suppose that $Y \in \mathrm{possde}(X,\mathcal{G}'_i)$. In this case $\widehat{\Theta}^{\mathbf{O}}_i = \hat{\beta}_{yx.\mathbf{o}_i}$, where $\mathbf{O}_i=\mathbf{O}(X,Y,\mathcal{G}'_i)$. As $\mathcal{G}'_i$ is amenable by construction, it follows from Lemma \ref{OO} in Appendix~\ref{appC} that \(\mathbf{O}_i=\mathbf{O}(X,Y,\mathcal{D})\). Further,  $Y \in \mathrm{possde}(X,\mathcal{G}'_i)$ implies that  $Y \notin \mathrm{pa}(X,\mathcal{G}'_i)$ and thus $\widehat{\Theta}^{\mathbf{P}}_i = \hat{\beta}_{yx.\mathbf{p}_i}$. 
	By Proposition \ref{Osuff}, $\mathbf{O}_i$ is a valid adjustment set relative to $(X,Y)$ in $\mathcal{D}$, and clearly the same holds for $\mathbf{P}_i$. Since we suppose multivariate Gaussianity, this implies that both \(\hat{\beta}_{yx.\mathbf{o}_i}\) and \(\hat{\beta}_{yx.\mathbf{p}_i}\) are consistent estimators of \(\tau_{yx}(\mathcal{D})\), and  \(\mathbb{E}[\hat{\beta}_{yx.\mathbf{o}_i}]= \mathbb{E}[\hat{\beta}_{yx.\mathbf{p}_i}]=\tau_{yx}(\mathcal{D})\).
	
	Further, by Lemmas E.4 and E.5 of the Supplement of \citetalias{HenckelPerkovicMaathuis2019}, \(\mathbf{P}_i\setminus\mathbf{O}_i\) is conditionally independent of \(Y\) given \(\{X\}\cup\mathbf{P}_i\), and \(\mathbf{O}_i\setminus\mathbf{P}_i\) is conditionally independent of \(X\) given \(\mathbf{P}_i\), respectively. These two independencies allow us to invoke Lemma C.2 of \citetalias{HenckelPerkovicMaathuis2019} and conclude that \(\sigma_{yy.x\mathbf{o}_i} \leq \sigma_{yy.x\mathbf{p}_i}\) as well as \(\sigma_{xx.\mathbf{o}_i} \geq \sigma_{xx.\mathbf{p}_i}\). As we assume a multivariate Gaussian distribution, it follows that\[a.var(\hat{\beta}_{yx.\mathbf{o}_i}) = \frac{\sigma_{yy.x\mathbf{o}_i}}{\sigma_{xx.\mathbf{o}_i}} \leq \frac{\sigma_{yy.x\mathbf{p}_i}}{\sigma_{xx.\mathbf{p}_i}} = a.var(\hat{\beta}_{yx.\mathbf{p}_i}).\]
	
	Suppose now that $Y \notin \mathrm{possde}(X,\mathcal{G}'_i)$. Then $Y \notin \mathrm{de}(X,\mathcal{D})$, hence $\tau_{yx}(\mathcal{D})=0$. As $Y \notin \mathrm{possde}(X,\mathcal{G}'_i)$, $\widehat{\Theta}^{\mathbf{O}}_i = 0$ and as a result $a.var(\widehat{\Theta}^{\mathbf{O}}_i) =0$. 
	If $Y \in \mathrm{pa}(X,\mathcal{G}'_i)$, then $\widehat{\Theta}^{\mathbf{P}_i} = 0$ and as a result $a.var(\widehat{\Theta}^{\mathbf{P}}_i) =0$. If $Y \notin \mathrm{possde}(X,\mathcal{G}'_i) \cup \mathrm{pa}(X,\mathcal{G}'_i)$, then $\widehat{\Theta}^{\mathbf{P}}_i = \hat{\beta}_{yx.\mathbf{p}_i}$ and by nature of parent sets $\mathbb{E}[\hat{\beta}_{yx.\mathbf{p}_i}]=0$.
	Clearly, $a.var(\widehat{\Theta}^{\mathbf{P}}_i) >0$ in this case.
\end{proof}

\newpage
\section[Appendix E: Violin Plots]{Violin Plots}
\label{appE}

\begin{figure}[h!]
	\begin{center}
	\includegraphics[height=16.5cm, trim=0 0.5cm 0 0, clip]{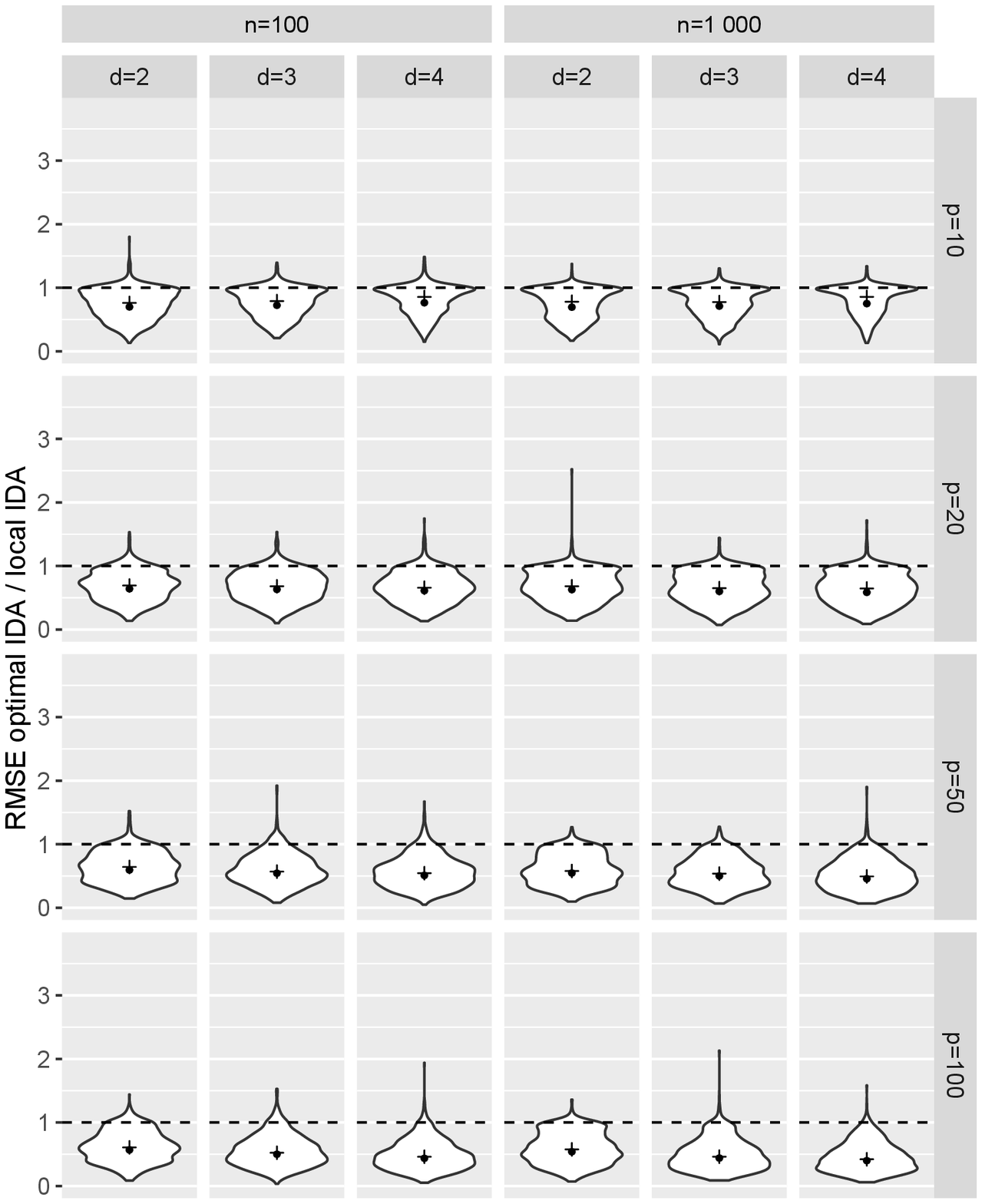}
	\end{center}
	\label{fig:violin_app1}
	\caption{Violin plots of the relative mean squared error (RMSE) \(r\)  over 1000 repetitions for scenarios with different numbers of nodes (\(p\)), expected number of neighbours per node (\(d\)), and sample sizes (\(n\)). Optimal and semi-local IDA were applied to the true CPDAG \(\mathcal{G}\). The dots mark the geometric means, the plus signs the medians.}
\end{figure}
\newpage
\begin{figure}[h!]
	\begin{center}
	\includegraphics[height=16.6cm, trim=0 0.5cm 0 0, clip]{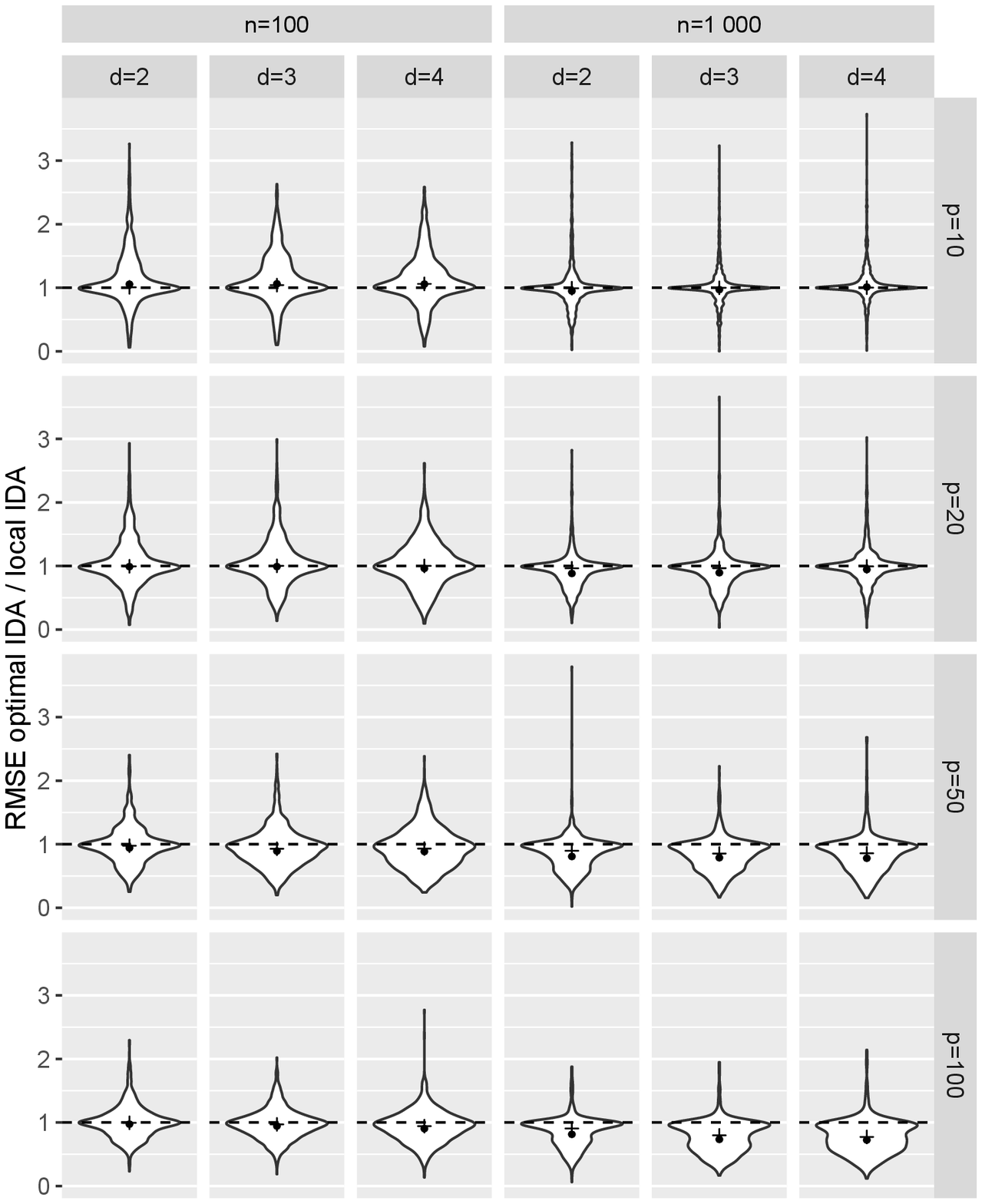}
	\end{center}
	\label{fig:violin_app2}
	\caption{Violin plots of the relative mean squared error (RMSE) \(r^*\)  over 1000 repetitions for scenarios with different numbers of nodes (\(p\)), expected number of neighbours per node (\(d\)), and sample sizes (\(n\)). Optimal and semi-local IDA were applied to the estimated CPDAG \(\mathcal{G}^*\). The dots mark the geometric means, the plus signs the medians.}
\end{figure}

\vskip 0.2in
\bibliography{EfficientBib3}

\begin{thebibliography}{72}
\providecommand{\natexlab}[1]{#1}
\providecommand{\url}[1]{\texttt{#1}}
\expandafter\ifx\csname urlstyle\endcsname\relax
  \providecommand{\doi}[1]{doi: #1}\else
  \providecommand{\doi}{doi: \begingroup \urlstyle{rm}\Url}\fi

\bibitem[Akaike(1974)]{Akaike1974}
Hirotugu Akaike.
\newblock A new look at the statistical model identification.
\newblock \emph{IEEE Transactions on Automatic Control}, 19\penalty0
  (6):\penalty0 716--723, 1974.

\bibitem[Andersson et~al.(1997)Andersson, Madigan, and
  Perlman]{AnderssonMadiganPerlman1997}
Steen~A. Andersson, David Madigan, and Michael~D. Perlman.
\newblock A characterization of {M}arkov equivalence classes for acyclic
  digraphs.
\newblock \emph{The Annals of Statistics}, 25\penalty0 (2):\penalty0 505--541,
  1997.

\bibitem[Belloni et~al.(2014)Belloni, Chernozhukov, and
  Hansen]{BelloniChernozhukovHansen2014}
Alexandre Belloni, Victor Chernozhukov, and Christian Hansen.
\newblock Inference on treatment effects after selection among high-dimensional
  controls.
\newblock \emph{The Review of Economic Studies}, 81\penalty0 (2):\penalty0
  608--650, 2014.

\bibitem[Berk et~al.(2013)Berk, Brown, Buja, Zhang, and Zhao]{Berketal2013}
Richard Berk, Lawrence Brown, Andreas Buja, Kai Zhang, and Linda Zhao.
\newblock Valid post-selection inference.
\newblock \emph{The Annals of Statistics}, 41\penalty0 (2):\penalty0 802--837,
  2013.

\bibitem[Brookhart et~al.(2006)Brookhart, Schneeweiss, Rothman, Glynn, Avorn,
  and St{\"u}rmer]{Brookhartetal2006}
M.~Alan Brookhart, Sebastian Schneeweiss, Kenneth~J. Rothman, Robert~J. Glynn,
  Jerry Avorn, and Til St{\"u}rmer.
\newblock Variable selection for propensity score models.
\newblock \emph{American Journal of Epidemiology}, 163\penalty0 (12):\penalty0
  1149--1156, 2006.

\bibitem[Chernozhukov et~al.(2018)Chernozhukov, Chetverikov, Demirer, Duflo,
  Hansen, Newey, and Robins]{Chernozhukovetal2018}
Victor Chernozhukov, Denis Chetverikov, Mert Demirer, Esther Duflo, Christian
  Hansen, Whitney Newey, and James Robins.
\newblock {Double/debiased machine learning for treatment and structural
  parameters}.
\newblock \emph{The Econometrics Journal}, 21\penalty0 (1):\penalty0 C1--C68,
  2018.

\bibitem[Chickering(2002)]{Chickering2002}
David~Maxwell Chickering.
\newblock Optimal structure identification with greedy search.
\newblock \emph{Journal of Machine Learning Research}, 3\penalty0
  (Nov):\penalty0 507--554, 2002.

\bibitem[Dawid and Didelez(2010)]{DawidDidelez2010}
A.~Philip Dawid and Vanessa Didelez.
\newblock Identifying the consequences of dynamic treatment strategies: A
  decision-theoretic overview.
\newblock \emph{Statistics Surveys}, 4:\penalty0 184--231, 2010.

\bibitem[de~Luna et~al.(2011)de~Luna, Waernbaum, and
  Richardson]{deLunaWaernbaumRichardson2011}
Xavier de~Luna, Ingeborg Waernbaum, and Thomas~S Richardson.
\newblock Covariate selection for the nonparametric estimation of an average
  treatment effect.
\newblock \emph{Biometrika}, 98\penalty0 (4):\penalty0 861--875, 2011.

\bibitem[Derryberry et~al.(2018)Derryberry, Aho, Edwards, and
  Peterson]{Derryberryetal2018}
DeWayne Derryberry, Ken Aho, John Edwards, and Teri Peterson.
\newblock Model selection and regression t-statistics.
\newblock \emph{The American Statistician}, 72\penalty0 (4):\penalty0 379--381,
  2018.

\bibitem[Dukes and Vansteelandt(2020{\natexlab{a}})]{DukesVansteelandt2020}
Oliver Dukes and Stijn Vansteelandt.
\newblock How to obtain valid tests and confidence intervals after propensity
  score variable selection?
\newblock \emph{Statistical Methods in Medical Research}, 29\penalty0
  (3):\penalty0 677--694, 2020{\natexlab{a}}.

\bibitem[Dukes and Vansteelandt(2020{\natexlab{b}})]{DukesVansteelandt2020b}
Oliver Dukes and Stijn Vansteelandt.
\newblock Inference on treatment effect parameters in potentially misspecified
  high-dimensional models.
\newblock \emph{Biometrika}, 2020{\natexlab{b}}.

\bibitem[Engelmann et~al.(2015)Engelmann, Amann, Ott-R{\"o}tzer, N{\"u}tzel,
  Reinders, Reinders, Thasler, Kristl, Teufel, Huber, Oefner, Spang, and
  Hellerbrand]{Engelmannetal2015}
Julia~C. Engelmann, Thomas Amann, Birgitta Ott-R{\"o}tzer, Margit N{\"u}tzel,
  Yvonne Reinders, J{\"o}rg Reinders, Wolfgang~E. Thasler, Theresa Kristl,
  Andreas Teufel, Christian~G. Huber, Peter~J. Oefner, Rainer Spang, and Claus
  Hellerbrand.
\newblock Causal modeling of cancer-stromal communication identifies {PAPPA} as
  a novel stroma-secreted factor activating {NF}$\kappa${B} signaling in
  hepatocellular carcinoma.
\newblock \emph{PLoS Computational Biology}, 11\penalty0 (5):\penalty0
  e1004293, 2015.

\bibitem[Greenland and Pearce(2015)]{GreenlandPearce2015}
Sander Greenland and Neil Pearce.
\newblock Statistical foundations for model-based adjustments.
\newblock \emph{Annual Review of Public Health}, 36:\penalty0 89--108, 2015.

\bibitem[Guo and Perkovi{\'c}(2020)]{GuoPerkovic2020}
F.~Richard Guo and Emilija Perkovi{\'c}.
\newblock Efficient least squares for estimating total effects under linearity
  and causal sufficiency.
\newblock \emph{arXiv preprint arXiv:2008.03481v2}, 2020.

\bibitem[Hahn(1998)]{Hahn1998}
Jinyong Hahn.
\newblock On the role of the propensity score in efficient semiparametric
  estimation of average treatment effects.
\newblock \emph{Econometrica}, 66:\penalty0 315--331, 1998.

\bibitem[Harrell(2010)]{Harrell2010}
Frank~E. Harrell, Jr.
\newblock \emph{Regression modeling strategies}.
\newblock Springer, 5th edition, 2010.

\bibitem[Henckel et~al.(2019)Henckel, Perkovi{\'c}, and
  Maathuis]{HenckelPerkovicMaathuis2019}
Leonard Henckel, Emilija Perkovi{\'c}, and Marloes~H. Maathuis.
\newblock Graphical criteria for efficient total effect estimation via
  adjustment in causal linear models.
\newblock \emph{arXiv preprint arXiv:1907.02435}, 2019.

\bibitem[Hirano et~al.(2003)Hirano, Imbens, and Ridder]{HiranoImbensRidder2003}
Keisuke Hirano, Guido~W. Imbens, and Geert Ridder.
\newblock Efficient estimation of average treatment effects using the estimated
  propensity score.
\newblock \emph{Econometrica}, 71\penalty0 (4):\penalty0 1161--1189, 2003.

\bibitem[Kalisch et~al.(2012)Kalisch, M\"achler, Colombo, Maathuis, and
  B\"uhlmann]{Kalischetal2012}
Markus Kalisch, Martin M\"achler, Diego Colombo, Marloes~H. Maathuis, and Peter
  B\"uhlmann.
\newblock Causal inference using graphical models with the {R} package {pcalg}.
\newblock \emph{Journal of Statistical Software}, 47\penalty0 (11):\penalty0
  1--26, 2012.

\bibitem[Kalisch et~al.(2019)Kalisch, Hauser, Maechler, Colombo, Entner, Hoyer,
  Hyttinen, Peters, Andri, Perkovic, Nandy, Ruetimann, Stekhoven, Schuerch, and
  Eigenmann]{Kalischetal2019}
Markus Kalisch, Alain Hauser, Martin Maechler, Diego Colombo, Doris Entner,
  Patrik Hoyer, Antti Hyttinen, Jonas Peters, Nicoletta Andri, Emilija
  Perkovic, Preetam Nandy, Philipp Ruetimann, Daniel Stekhoven, Manuel
  Schuerch, and Marco Eigenmann.
\newblock \emph{pcalg: Methods for Graphical Models and Causal Inference},
  2019.
\newblock URL \url{https://CRAN.R-project.org/package=pcalg}.
\newblock R package version 2.6-6.

\bibitem[Kleinbaum and Kupper(1978)]{KleinbaumKupper1978}
David~G. Kleinbaum and Lawrence~L. Kupper.
\newblock \emph{Applied regression analysis and other multivariable methods}.
\newblock Duxbury Press, 1978.

\bibitem[Kn{\"u}ppel and Stang(2010)]{KnueppelStang2010}
Sven Kn{\"u}ppel and Andreas Stang.
\newblock {DAG} program: Identifying minimal sufficient adjustment sets.
\newblock \emph{Epidemiology}, 21\penalty0 (1):\penalty0 159, 2010.

\bibitem[Kuroki and Cai(2004)]{KurokiCai2004}
Manabu Kuroki and Zhihong Cai.
\newblock Selection of identifiability criteria for total effects by using path
  diagrams.
\newblock In \emph{Proceedings of the Twentieth Conference on Uncertainty in
  Artificial Intelligence (UAI-04)}, pages 333--340, 2004.

\bibitem[Kuroki and Miyakawa(2003)]{KurokiMiyakawa2003}
Manabu Kuroki and Masami Miyakawa.
\newblock Covariate selection for estimating the causal effect of control plans
  by using causal diagrams.
\newblock \emph{Journal of the Royal Statistical Society: Series B (Statistical
  Methodology)}, 65\penalty0 (1):\penalty0 209--222, 2003.

\bibitem[Le et~al.(2013)Le, Liu, Tsykin, Goodall, Liu, Sun, and Li]{Leetal2013}
Thuc~Duy Le, Lin Liu, Anna Tsykin, Gregory~J Goodall, Bing Liu, Bing-Yu Sun,
  and Jiuyong Li.
\newblock Inferring {microRNA}--{mRNA} causal regulatory relationships from
  expression data.
\newblock \emph{Bioinformatics}, 29\penalty0 (6):\penalty0 765--771, 2013.

\bibitem[Lee et~al.(2016)Lee, Sun, Sun, and Taylor]{Leeetal2016}
Jason~D. Lee, Dennis~L. Sun, Yuekai Sun, and Jonathan~E. Taylor.
\newblock Exact post-selection inference, with application to the lasso.
\newblock \emph{The Annals of Statistics}, 44\penalty0 (3):\penalty0 907--927,
  2016.

\bibitem[Leeb and P{\"o}tscher(2008)]{LeebPoetscher2008}
Hannes Leeb and Benedikt~M. P{\"o}tscher.
\newblock Can one estimate the unconditional distribution of
  post-model-selection estimators?
\newblock \emph{Econometric Theory}, 24\penalty0 (2):\penalty0 338--376, 2008.

\bibitem[Li et~al.(2005)Li, Dennis~Cook, and Nachtsheim]{LiCookNachtsheim2005}
Lexin Li, R.~Dennis~Cook, and Christopher~J. Nachtsheim.
\newblock Model-free variable selection.
\newblock \emph{Journal of the Royal Statistical Society: Series B (Statistical
  Methodology)}, 67\penalty0 (2):\penalty0 285--299, 2005.

\bibitem[Lockhart et~al.(2014)Lockhart, Taylor, Tibshirani, and
  Tibshirani]{Lockhartetal2014}
Richard Lockhart, Jonathan Taylor, Ryan~J. Tibshirani, and Robert Tibshirani.
\newblock A significance test for the lasso.
\newblock \emph{Annals of Statistics}, 42\penalty0 (2):\penalty0 413--468,
  2014.

\bibitem[Lounici(2008)]{Lounici2008}
Karim Lounici.
\newblock Sup-norm convergence rate and sign concentration property of {L}asso
  and {D}antzig estimators.
\newblock \emph{Electronic Journal of Statistics}, 2:\penalty0 90--102, 2008.

\bibitem[Lunceford and Davidian(2004)]{LuncefordDavidian2004}
Jared~K. Lunceford and Marie Davidian.
\newblock Stratification and weighting via the propensity score in estimation
  of causal treatment effects: A comparative study.
\newblock \emph{Statistics in Medicine}, 23\penalty0 (19):\penalty0 2937--2960,
  2004.

\bibitem[Luo et~al.(2018)Luo, Huang, and Cao]{LuoHuangCao2018}
Jiawei Luo, Wei Huang, and Buwen Cao.
\newblock A novel approach to identify the {miRNA-mRNA} causal regulatory
  modules in cancer.
\newblock \emph{IEEE/ACM Transactions on Computational Biology and
  Bioinformatics}, 15\penalty0 (1):\penalty0 309--315, 2018.

\bibitem[Maathuis and Colombo(2015)]{MaathuisColombo2015}
Marloes~H. Maathuis and Diego Colombo.
\newblock A generalised back-door criterion.
\newblock \emph{The Annals of Statistics}, 43\penalty0 (3):\penalty0
  1060--1088, 2015.

\bibitem[Maathuis et~al.(2009)Maathuis, Kalisch, and
  B{\"u}hlmann]{MaathuisKalischBuhlmann2009}
Marloes~H. Maathuis, Markus Kalisch, and Peter B{\"u}hlmann.
\newblock Estimating high-dimensional intervention effects from observational
  data.
\newblock \emph{The Annals of Statistics}, 37\penalty0 (6A):\penalty0
  3133--3164, 2009.

\bibitem[Maathuis et~al.(2010)Maathuis, Colombo, Kalisch, and
  B{\"u}hlmann]{Maathuisetal2010}
Marloes~H. Maathuis, Diego Colombo, Markus Kalisch, and Peter B{\"u}hlmann.
\newblock Predicting causal effects in large-scale systems from observational
  data.
\newblock \emph{Nature Methods}, 7\penalty0 (4):\penalty0 247--248, 2010.

\bibitem[Meek(1995)]{Meek1995}
Christopher Meek.
\newblock Causal inference and causal explanation with background knowledge.
\newblock In \emph{Proceedings of the Eleventh Conference on Uncertainty in
  Artificial Intelligence (UAI-95)}, pages 403--410, 1995.

\bibitem[Montgomery et~al.(2012)Montgomery, Peck, and
  Vining]{MontgomeryPeckVining2012}
Douglas~C. Montgomery, Elizabeth~A. Peck, and C.~Geoffrey Vining.
\newblock \emph{Introduction to linear regression analysis}.
\newblock John Wiley \& Sons, 5th edition, 2012.

\bibitem[Murtaugh(2014)]{Murtaugh2014}
Paul~A. Murtaugh.
\newblock In defense of p values.
\newblock \emph{Ecology}, 95\penalty0 (3):\penalty0 611--617, 2014.

\bibitem[Nandy et~al.(2017)Nandy, Maathuis, and
  Richardson]{NandyMaathuisRichardson2017}
Preetam Nandy, Marloes~H. Maathuis, and Thomas~S. Richardson.
\newblock Estimating the effect of joint interventions from observational data
  in sparse high-dimensional settings.
\newblock \emph{The Annals of Statistics}, 45\penalty0 (2):\penalty0 647--674,
  2017.

\bibitem[Pearl(2001)]{Pearl2001}
Judea Pearl.
\newblock Direct and indirect effects.
\newblock In \emph{Proceedings of the Seventeenth Conference on Uncertainty in
  Artificial Intelligence (UAI-01)}, pages 411--420, 2001.

\bibitem[Pearl(2009)]{Pearl2009}
Judea Pearl.
\newblock \emph{Causality: Models, reasoning, and inference}.
\newblock Cambridge University Press, 2nd edition, 2009.

\bibitem[Perkovi\'{c}(2020)]{Perkovic2019}
Emilija Perkovi\'{c}.
\newblock Identifying causal effects in maximally oriented partially directed
  acyclic graphs.
\newblock In \emph{Proceedings of the Thirty-Sixth Conference on Uncertainty in
  Artificial Intelligence (UAI-20)}, page ID: 229, 2020.

\bibitem[Perkovi\'{c} et~al.(2015)Perkovi\'{c}, Textor, Kalisch, and
  Maathuis]{Perkovicetal2015}
Emilija Perkovi\'{c}, Johannes Textor, Markus Kalisch, and Marloes~H. Maathuis.
\newblock A complete generalized adjustment criterion.
\newblock In \emph{Proceedings of the Thirty-First Conference on Uncertainty in
  Artificial Intelligence (UAI-15)}, pages 682--691, 2015.

\bibitem[Perkovi\'{c} et~al.(2017)Perkovi\'{c}, Kalisch, and
  Maathuis]{PerkovicKalischMaathuis2017}
Emilija Perkovi\'{c}, Markus Kalisch, and Maloes~H. Maathuis.
\newblock Interpreting and using {CPDAG}s with background knowledge.
\newblock In \emph{Proceedings of the Thirty-Third Conference on Uncertainty in
  Artificial Intelligence (UAI-17)}, page ID: 120, 2017.

\bibitem[Perkovi\'{c} et~al.(2018)Perkovi\'{c}, Textor, Kalisch, and
  Maathuis]{Perkovicetal2018}
Emilija Perkovi\'{c}, Johannes Textor, Markus Kalisch, and Marloes~H. Maathuis.
\newblock Complete graphical characterization and construction of adjustment
  sets in {M}arkov equivalence classes of ancestral graphs.
\newblock \emph{Journal of Machine Learning Research}, 18\penalty0
  (220):\penalty0 1--62, 2018.

\bibitem[{R Core Team}(2019)]{R}
{R Core Team}.
\newblock \emph{R: A Language and Environment for Statistical Computing}.
\newblock R Foundation for Statistical Computing, Vienna, Austria, 2019.
\newblock URL \url{https://www.R-project.org/}.

\bibitem[Ramsey(2014)]{Ramsey2014}
Joseph~D. Ramsey.
\newblock A scalable conditional independence test for nonlinear, non-gaussian
  data.
\newblock \emph{arXiv preprint arXiv:1401.5031v2}, 2014.

\bibitem[Richardson(2003)]{Richardson2003}
Thomas Richardson.
\newblock Markov properties for acyclic directed mixed graphs.
\newblock \emph{Scandinavian Journal of Statistics}, 30\penalty0 (1):\penalty0
  145--157, 2003.

\bibitem[{Richardson} et~al.(2017){Richardson}, {Evans}, {Robins}, and
  {Shpitser}]{Richardsonetal2017}
Thomas~S. {Richardson}, Robin~J. {Evans}, James~M. {Robins}, and Ilya
  {Shpitser}.
\newblock Nested {M}arkov properties for acyclic directed mixed graphs.
\newblock \emph{arXiv preprint arXiv:1701.06686}, art. arXiv:1701.06686, 2017.

\bibitem[Rinaldo et~al.(2019)Rinaldo, Wasserman, and
  G'Sell]{RinaldoWassermanGSell2019}
Alessandro Rinaldo, Larry Wasserman, and Max G'Sell.
\newblock Bootstrapping and sample splitting for high-dimensional,
  assumption-lean inference.
\newblock \emph{The Annals of Statistics}, 47\penalty0 (6):\penalty0
  3438--3469, 2019.

\bibitem[Robins(1986)]{Robins1986}
James Robins.
\newblock A new approach to causal inference in mortality studies with a
  sustained exposure period—application to control of the healthy worker
  survivor effect.
\newblock \emph{Mathematical Modelling}, 7\penalty0 (9--12):\penalty0
  1393--1512, 1986.

\bibitem[Robins and Greenland(1992)]{RobinsGreenland1992}
James~M. Robins and Sander Greenland.
\newblock Identifiability and exchangeability for direct and indirect effects.
\newblock \emph{Epidemiology}, 3\penalty0 (2):\penalty0 143--155, 1992.

\bibitem[Rotnitzky and Smucler(2020)]{RotnitzkySmucler2019}
Andrea Rotnitzky and Ezequiel Smucler.
\newblock Efficient adjustment sets for population average treatment effect
  estimation in non-parametric causal graphical models.
\newblock \emph{Journal of Machine Learning Research}, 21\penalty0
  (188):\penalty0 1--86, 2020.

\bibitem[Schwarz(1978)]{Schwarz1978}
Gideon Schwarz.
\newblock Estimating the dimension of a model.
\newblock \emph{The Annals of Statistics}, 6\penalty0 (2):\penalty0 461--464,
  1978.

\bibitem[Shah and Peters(2020)]{ShahPeters2020}
Rajen~D. Shah and Jonas Peters.
\newblock The hardness of conditional independence testing and the generalised
  covariance measure.
\newblock \emph{The Annals of Statistics}, 48\penalty0 (3):\penalty0
  1514--1538, 2020.

\bibitem[Shimizu et~al.(2006)Shimizu, Hoyer, Hyv{\"a}rinen, and
  Kerminen]{Shimizuetal2006}
Shohei Shimizu, Patrik~O. Hoyer, Aapo Hyv{\"a}rinen, and Antti Kerminen.
\newblock A linear non-{G}aussian acyclic model for causal discovery.
\newblock \emph{Journal of Machine Learning Research}, 7:\penalty0 2003--2030,
  2006.

\bibitem[Shortreed and Ertefaie(2017)]{ShortreedErtefaie2017}
Susan~M. Shortreed and Ashkan Ertefaie.
\newblock Outcome-adaptive lasso: Variable selection for causal inference.
\newblock \emph{Biometrics}, 73\penalty0 (4):\penalty0 1111--1122, 2017.

\bibitem[Shpitser et~al.(2010)Shpitser, VanderWeele, and
  Robins]{ShpitserVanderWeeleRobins2010}
Ilya Shpitser, Tyler VanderWeele, and James~M. Robins.
\newblock On the validity of covariate adjustment for estimating causal
  effects.
\newblock In \emph{Proceedings of the Twenty-Sixth Conference on Uncertainty in
  Artificial Intelligence (UAI-10)}, pages 527--536, 2010.

\bibitem[Shpitser et~al.(2014)Shpitser, Evans, Richardson, and
  Robins]{Shpitseretal2014}
Ilya Shpitser, Robin~J. Evans, Thomas~S. Richardson, and James~M. Robins.
\newblock Introduction to nested {M}arkov models.
\newblock \emph{Behaviormetrika}, 41\penalty0 (1):\penalty0 3--39, 2014.

\bibitem[Smucler et~al.(2019)Smucler, Rotnitzky, and
  Robins]{SmuclerRotnitzkyRobins2019}
Ezequiel Smucler, Andrea Rotnitzky, and James~M. Robins.
\newblock A unifying approach for doubly-robust \(l_1\) regularized estimation
  of causal contrasts.
\newblock \emph{arXiv preprint arXiv:1904.03737v3}, 2019.

\bibitem[Smucler et~al.(2020)Smucler, Sapienza, and
  Rotnitzky]{SmuclerSapienzaRotnitzky2020}
Ezequiel Smucler, Facundo Sapienza, and Andrea Rotnitzky.
\newblock Efficient adjustment sets in causal graphical models with hidden
  variables.
\newblock \emph{arXiv preprint arXiv:1912.00306}, 2020.

\bibitem[Spirtes et~al.(2000)Spirtes, Glymour, and Scheines]{Spirtesetal2000}
Peter Spirtes, Clark Glymour, and Richard Scheines.
\newblock \emph{Causation, prediction, and search}.
\newblock MIT press, 2000.

\bibitem[Textor and Li\'{s}kiewicz(2011)]{TextorLiskiewicz2011}
Johannes Textor and Maciej Li\'{s}kiewicz.
\newblock Adjustment criteria in causal diagrams: An algorithmic perspective.
\newblock In \emph{Proceedings of the 27th Conference on Uncertainty in
  Artificial Intelligence (UAI-11)}, pages 681--688, 2011.

\bibitem[Tian and Pearl(2003)]{TianPearl2003}
Jin Tian and Judea Pearl.
\newblock On the identification of causal effects.
\newblock Technical Report R-290-L, Department of Computer Science, University,
  University of California, Los Angeles, 2003.

\bibitem[Tibshirani(1996)]{Tibshirani1996}
Robert Tibshirani.
\newblock Regression shrinkage and selection via the lasso.
\newblock \emph{Journal of the Royal Statistical Society: Series B (Statistical
  Methodology)}, 58\penalty0 (1):\penalty0 267--288, 1996.

\bibitem[van Kampen(2006)]{vanKampen2006}
Dirk van Kampen.
\newblock The {S}chizotypic {S}yndrome {Q}uestionnaire ({SSQ}): Psychometrics,
  validation and norms.
\newblock \emph{Schizophrenia Research}, 84\penalty0 (2--3):\penalty0 305--322,
  2006.

\bibitem[van Kampen(2014)]{vanKampen2014}
Dirk van Kampen.
\newblock The {SSQ} model of schizophrenic prodromal unfolding revised: An
  analysis of its causal chains based on the language of directed graphs.
\newblock \emph{European Psychiatry}, 29\penalty0 (7):\penalty0 437--448, 2014.

\bibitem[Verma and Pearl(1990)]{VermaPearl1990}
TS~Verma and Judea Pearl.
\newblock Equivalence and synthesis of causal models.
\newblock Technical Report R-150, Department of Computer Science, University,
  University of California, Los Angeles, 1990.

\bibitem[Witte and Didelez(2019)]{WitteDidelez2019}
Janine Witte and Vanessa Didelez.
\newblock Covariate selection strategies for causal inference: Classification
  and comparison.
\newblock \emph{Biometrical Journal}, 61\penalty0 (5):\penalty0 1270--1289,
  2019.

\bibitem[Zhao and Yu(2006)]{ZhaoYu2006}
Peng Zhao and Bin Yu.
\newblock On model selection consistency of {L}asso.
\newblock \emph{Journal of Machine Learning Research}, 7:\penalty0 2541--2563,
  2006.

\bibitem[Zou and Hastie(2005)]{ZouHastie2005}
Hui Zou and Trevor Hastie.
\newblock Regularization and variable selection via the elastic net.
\newblock \emph{Journal of the {R}oyal {S}tatistical {S}ociety: {S}eries B
  ({S}tatistical {M}ethodology)}, 67\penalty0 (2):\penalty0 301--320, 2005.

\end{thebibliography}

\end{document}